\documentclass[11pt,english]{amsart}

\usepackage[margin=2cm]{geometry}

\usepackage{color,xcolor}

\usepackage[pdfstartpage=1,colorlinks=true,bookmarks=false,pdfstartview={FitH}]{hyperref}
\usepackage{soul}
\usepackage{cite}
\usepackage{amsmath,amssymb}
\usepackage{mathtools}
\usepackage{verbatim}
\usepackage{comment}
\usepackage{esint}
\usepackage{cleveref}
\usepackage{enumerate}

\usepackage{mathrsfs}
\newcommand{\sE}{\mathscr{E}}
\newcommand{\sG}{\mathscr{G}}
\newcommand{\sL}{\mathscr{L}}
\newcommand{\sN}{\mathscr{N}}
\newcommand{\sX}{\mathscr{X}}
\newcommand{\sY}{\mathscr{Y}}

\usepackage[english]{babel}
\usepackage{makecell}

\usepackage{datetime}

\DeclarePairedDelimiter\abs{\lvert}{\rvert}
\makeatletter
\let\oldabs\abs
\def\abs{\@ifstar{\oldabs}{\oldabs*}}
\makeatletter
\newcommand{\vast}{\bBigg@{4}}
\newcommand{\Vast}{\bBigg@{5}}
\makeatother

\newcommand{\eps}{\varepsilon}
\newcommand{\Beta}{\mathrm{B}}

\newcommand{\R}{\mathbb{R}}
\newcommand{\N}{\mathbb{N}}
\newcommand{\bS}{\mathbb{S}}
\newcommand{\p}{\partial}

\newcommand{\Ds}{(-\Delta)^{s}}
\newcommand{\Ints}{(-\Delta)^{-s}}

\newcommand{\PV}{\textnormal{P.V.}\,}
\newcommand{\loc}{\textnormal{loc}}
\newcommand{\norm}[2][]{\left\|{#2}\right\|_{#1}}
\newcommand{\seminorm}[2][]{\left[{#2}\right]_{#1}}
\newcommand{\sign}{\textnormal{sign}\,}
\newcommand{\set}[1]{\left\{#1\right\}}

\newcommand{\textif}{\text{ if }}
\newcommand{\textas}{\text{ as }}
\newcommand{\texton}{\text{ on }}

\newcommand{\textae}{\text{ a.e. }}

\newcommand{\textfor}{\text{ for }}
\newcommand{\textand}{\text{ and }}

\newcommand{\oneset}[1]{\mathbf{1}_{\set{#1}}}

\newcommand{\dist}{{\rm dist}\, }

\newcommand{\bG}{\mathbb{G}}

\newcommand{\bP}{\mathbb{P}}
\newcommand{\bT}{\mathbb{T}}
\newcommand{\bV}{\mathbb{V}}

\newcommand{\cE}{\mathcal{E}}

\newcommand{\cN}{\mathcal{N}}
\newcommand{\cG}{\mathcal{G}}
\newcommand{\cH}{\mathcal{H}}

\newcommand{\cK}{\mathcal{K}}
\newcommand{\cL}{\mathcal{L}}
\newcommand{\cP}{\mathcal{P}}
\newcommand{\cR}{\mathcal{R}}

\newcommand{\cT}{\mathcal{T}}
\newcommand{\cV}{\mathcal{V}}

\newcommand{\angles}[1]{\left\langle{#1}\right\rangle}

\newcommand{\Dh}{(-\Delta)^{\frac12}}

\newcommand{\Dhn}{(-\Delta_{\R^n})^{\frac12}}
\newcommand{\DhN}{(-\Delta_{\R^N})^{\frac12}}

\newcommand{\Dsn}{(-\Delta_{\R^n})^{s}}
\newcommand{\DsN}{(-\Delta_{\R^N})^{s}}
\newcommand{\Intsn}{(-\Delta_{\R^n})^{-s}}
\newcommand{\IntsN}{(-\Delta_{\R^N})^{-s}}
\newcommand{\Inthn}{(-\Delta_{\R^n})^{-\frac12}}
\newcommand{\InthN}{(-\Delta_{\R^N})^{-\frac12}}
\DeclareMathOperator{\diam}{diam}

\DeclareMathOperator{\Hyperg}{\mbox{ }_2 F_{1}}

\theoremstyle{plain}
\newtheorem{thm}{Theorem}[section]
\newtheorem{lem}[thm]{Lemma}
\newtheorem{cor}[thm]{Corollary}
\newtheorem{prop}[thm]{Proposition}

\newtheorem{conj}[thm]{Conjecture}

\newtheorem*{prop*}{Proposition}
\newtheorem*{thm*}{Theorem}

\theoremstyle{definition}
\newtheorem{defn}[thm]{Definition}

\theoremstyle{remark}
\newtheorem{remark}{Remark}[section]
\newcommand{\bremark}{\begin{remark} \em}
\newcommand{\eremark}{\end{remark} }
\newcommand{\COMMENT}[1]{}

\numberwithin{equation}{section}

\definecolor{g1}{rgb}{0,0.5,0.1}
\definecolor{g2}{rgb}{0,0.6,0}
\definecolor{r2}{rgb}{0.8,0,0}

\vbadness=\maxdimen
\hbadness=\maxdimen

\begin{document}

\title[Fractional Lane--Emden--Serrin equation]
    {
From fractional Lane--Emden--Serrin equation---existence, multiplicity and local behaviors via classical ODE---to fractional Yamabe metrics with singularity of ``maximal'' dimension
    }

\author{Hardy Chan}
\email[H.~Chan]{hardy.chan@unibas.ch}
\address[H.~Chan]{Department of Mathematics, ETH Z\"{u}rich and Instituto de Ciencias Matem\'aticas and Department of Mathematics, University of Basel}

\author{Azahara DelaTorre}
\email[A. DelaTorre]{azahara.delatorrepedraza@uniroma1.it}
\address[A. DelaTorre]{Dipartimento di Matematica Guido Castelnuovo, Sapienza Università di Roma}

\begin{abstract}

Point singularities of solutions to the classical Lane--Emden--Serrin equation have a polyhomogeneous asymptotic expansion whose logarithmic corrections are determined by a first order ODE. Surprisingly, we are able to discover such an ODE for the fractional Lane--Emden--Serrin equation, and therefore give a short classification for the precise local behavior of its solutions up to the second order involving a double logarithm. This seems to be the first time that a nonlocal equation is associated to a genuinely local ODE in one dimension. New non-existence, existence and multiplicity results for the corresponding Dirichlet problem are also discussed.

Moreover, we construct complete $s$-fractional Yamabe metrics in $\mathbb{R}^n$ which are singular along a smooth submanifold of dimension $(n-2s)/2$, via direct integral asymptotic analysis with global geometric weights. This covers the missing borderline case as suggested by the deep work of Schoen and Yau. While such dimension is maximal in the class of distributional solutions, we conjecture the existence of complete metrics, understood in a suitably generalized sense, with prescribed singularities of strictly higher dimensions.

\end{abstract}

\maketitle

\section{Introduction}

\subsection{Overview}

This work is devoted to studying singular positive solutions to the fractional Lane--Emden equation with Serrin-critical exponent (or Lane--Emden--Serrin equation for short)
\begin{equation}\label{eq:intro-1.1}
\DsN u=u^{p}
\qquad \text{ in } B_1\setminus\set{0} \subset \R^N,
\quad \text{ for } s\in(0,1),\,p=\tfrac{N}{N-2s},
\end{equation}
and to the so-called singular fractional Yamabe problem in conformal geometry,
\begin{equation}\label{eq:intro-1.2}
\Dsn v=v^{\frac{n+2s}{n-2s}}
\qquad \text{ in } \R^n\setminus \Sigma^k,
\end{equation}
where $\Sigma$ is a smooth submanifold of dimension $k=\frac{n-2s}{2}$.

Here \eqref{eq:intro-1.1} dictates the (singular) behavior of solutions to \eqref{eq:intro-1.2}, as seen in the model case $\Sigma=\R^k$ where $\R^n\setminus \Sigma=(\R^{n-k}\times \setminus\set{0})\times \R^k \cong \R^N\setminus\set{0}$ if we denote the co-dimension by $N=n-k$ and consider solutions independent of the (tangential) variables in $\Sigma=\R^k$, leading to \eqref{eq:intro-1.1} with  $p=\frac{N+k+2s}{N+k-2s}$, a Sobolev-subcritical exponent in dimension $N$.

Analytically, the criticality of the exponent $p=\frac{N}{N-2s}$ in \eqref{eq:intro-1.1} was first observed in 1965 for $s=1$ by J. Serrin, who classified isolated singularities of quasilinear equations \cite{Serrin1965}. In the setting of Lane--Emden equation, his result corresponds to the fact that non-removable singularities behave like the fundamental solution for $p\in(1,\frac{N}{N-2})$, as proved by P.-L. Lions \cite{Lions1980}. 
Geometrically, the deep result of Schoen and Yau \cite{SY} bounds, when $s=1$, the dimension of the singular set $\Sigma^k$ in \eqref{eq:intro-1.2} from above by $\frac{n-2s}{2}$, as long as $v$ is the conformal factor of a \emph{complete} Yamabe metric. At the ``maximal'' dimension $k=\frac{n-2s}{2}$, the Sobolev-critical exponent $\frac{n+2s}{n-2s}=\frac{N}{N-2s}$ is precisely Serrin-critical in the co-dimension $N=\frac{n+2s}{2}$. ``Maximality'' depends on the notion of solutions and it will be clarified in \Cref{sec:state-art-construction}.

The main contributions of our work, which  concern \eqref{eq:intro-1.1}--\eqref{eq:intro-1.2} for $s\in(0,1)$, are the following:
\begin{itemize}
\item The surprising equivalence of the fractional equation \eqref{eq:intro-1.1} to a local ODE \eqref{eq:intro-asymp-ODE} in the asymptotic regime, leading to a short and completely local proof for the classification of the singular behavior (\Cref{thm:behavior-rad}). To the best of our knowledge, this is the first genuinely local argument for an equation of fractional order without extending to a higher dimension.
\item Liouville theorems under minimal assumptions for solutions to \eqref{eq:intro-1.1} but in an exterior domain (\Cref{thm:Liouville-int}, \Cref{thm:Liouville-diff}).
\item Existence of solutions to \eqref{eq:intro-1.1} via integral asymptotic analysis of the nonlocal ODE (\Cref{thm:rad-s}).
\item Construction of singular solutions to \eqref{eq:intro-1.2} (\Cref{th:Yamabe}),  for $s=\frac12$ due to parity, by a simplified gluing method.

\end{itemize}

Regarding the singular Yamabe problem, we hope that our paper also serves as a bridge between the classical studies and a new phenomenon peculiar to nonlocality (see \Cref{table} on page \pageref{table}). On the one hand, with \Cref{th:Yamabe}, we complete the analogy ($k\leq\frac{n-2s}{2}$) with the classical case. On the other hand, we expect that \Cref{conj:main} on the existence of complete Yamabe metrics with \emph{very high} dimensional singularities (for some precise values of $k>\frac{n-2s}{2}$) opens up a new research direction.

\subsection{Organization}
The literature concerning qualitative properties of solutions to \eqref{eq:intro-1.1} as well as the geometric background and motivations for \eqref{eq:intro-1.2} are introduced 
in \Cref{sec:EF}--\Cref{sec:geometric} respectively, where we also point out our contributions. The main results are presented  in \Cref{sec:main-results}.

In \Cref{sec:behavior-rad} we prove the classification of local behavior of \eqref{eq:intro-1.1} using the asymptotic ODE. In \Cref{sec:Liouville} we show Liouville-type results. In \Cref{sec:rad} we collect computations involving an adapted conformal fractional Laplacian on radial functions. These are used as fine barriers in the construction of the singular profile in \Cref{sec:radial}. Next, singular $\frac12$-Yamabe metrics with a prescribed singular set of dimension $\frac{n-1}{2}$ are built in \Cref{sec:Yamabe}. Finally, in the appendices, technical calculations are recorded, and well-developed arguments are refined and tailored for our equations. 

\normalcolor

\subsection{Notations}
The functional setting and the notions of solution are introduced in \Cref{app:functional}.

\begin{itemize}
	\item Constants are \emph{universal} when they depend only on $N,k,n$, $s$, $\Omega$ and $\Sigma$. Generic universal constants are denoted by $C$ (big constants) and $c$ (small constants). 
	\item $B_r^n(x):= \{ y\in\R^n: \ \|y-x\| <r \}$. The dimension will be omitted when understood.
    \item $\angles{x}=\sqrt{1+|x|^2}$ is the Japanese bracket.
	\item 	$a_+:=\max\{a,0\}$ is the positive part of $a$.
	\item $f\lesssim g$ means 
	$f\leq C g$, $f\lesssim_{\beta} g$ means $f\leq C(\beta)g$ and 
    $f\asymp g$ means both $f\lesssim g$ and $g\lesssim f$.
\item $f\sim g$ means that $f=(1+o(1))g$, i.e. $f/g \to 1$ in the specified limit.
\item Singular integrals are understood in the principal value sense if necessary (even if not emphasized).	
	
\end{itemize}

\subsection{Remark on the title}
In our previous work \cite{Chan_DelaTorre1}, we announced that \Cref{th:Yamabe}  
would appear in a paper entitled ``Singular solutions for a critical fractional Yamabe problem.''

\subsection{Added comment}
A few days after the first draft of the present work appeared online, H. Chen announced a preprint  \cite{Chen2021} where he provides yet another (other than \cite{WW}) alternate and  independent classification of local behavior, %
determining the first order up to a scalar multiple, by means of fine comparisons together with direct while technical computations on $\R^N$ in Cartesian coordinates.

\section{The Lane--Emden--Serrin equation}
\label{sec:EF}

\subsection{Local singular behavior for classical Lane--Emden equations
}
\label{subsec:LEeq}
The Lane--Emden equation
\[
{-\Delta}u=u^p
    \qquad \text{ in } B_1\subset\R^N
\]
has been studied for a long time \cite{Chandrasekhar} as a model for stellar structure in astrophysics. %
Central to the mathematical understanding of solutions is the possible formation of isolated singularities. By the translation invariance of the equation, we place the singularity at the origin. For distributional solutions, they start to exist at the Serrin-critical exponent $p=\frac{N}{N-2}$ and are well-understood up to and including the Sobolev-critical exponent $p=\frac{N+2}{N-2}$ due to the combined work of Lions, Gidas, Spruck, Caffarelli and Aviles \cite{Lions1980, GS, CGS, Aviles-2} in the 1980's, namely
\[
u(x) \asymp
\begin{cases}
|x|^{-\frac{2}{p-1}}
    & \textfor p\in(\frac{N}{N-2},\frac{N+2}{N-2}],\\
|x|^{-(N-2)}(\log\frac{1}{|x|})^{-\frac{N-2}{2}}
    & \textfor p=\frac{N}{N-2},
\end{cases}
\]
as $x\to 0$. 
(A more precise expansion can be made. See also \cite{BV} where a Hardy potential is present, and \cite{BPV,DMP2007} where isolated boundary singularities are investigated.) In order to motivate the forthcoming discussions, we observe that the Emden--Fowler transformation (restricted to the radial case for simplicity)
\[
u(r)=r^{-\frac{2}{p-1}}v(t),
    \qquad
t=-\log r\to+\infty,
\]
turns the Lane--Emden equation into an autonomous ODE
\[
-v_{tt}
-a_p v_t
+b_pv
=v^p,
\]
with constant coefficients
\[
a_p=-\left(
    N-2-\frac{4}{p-1}
\right)
\geq 0
    \quad \textfor \quad
p \leq \frac{N+2}{N-2},
\]
\[
b_p=\frac{2}{p-1}
\left(
    N-2-\frac{2}{p-1}
\right)
\geq 0
    \quad \textfor \quad
p \geq \frac{N}{N-2},
\]
so that detailed structures of the solution can be revealed via a phase-plane analysis. The criticality of the Serrin exponent is seen from the fact 
$b_{\frac{N}{N-2}}=0$. Then, asymptotically $v$ is no longer the constant $b_p^{\frac{1}{p-1}}$, but determined by the ODE

\begin{equation}\label{eq:intro-ODE}
{-}a_{\frac{N}{N-2}}
v_t=v^{\frac{N}{N-2}}.
\end{equation}
For more details, we refer to the concise book \cite{QSbook}.

\subsection{Local singular behavior for fractional Lane--Emden equations }

We now focus on the nonlocal counterpart
\begin{equation}\label{eq:intro-1.4}
(-\Delta_{\R^N})^s u=u^{p}\quad \text{ in } B_1\setminus\set{0}.
\end{equation}
Here $s\in(0,1)$ and the fractional Laplacian is the %
integro-differential operator defined by the principal value integral

$$(-\Delta_{\R^N})^s u(x)=C_{N,s} \PV {\int_{\R^N}}\frac{u(x)-u(y)}{|x-y|^{N+2s}}\,dy,
\qquad
C_{N,s}
=\dfrac{
	2^{2s}
	\Gamma(\frac{N+2s}{2})
}{
	\Gamma(2-s)
	\pi^{\frac{N}{2}}
}s(1-s).
$$
In this setting, local singular behaviors are already known for Serrin-supercritical exponents \cite{CJSX, YZ2}, namely
\[
u(x) \asymp
|x|^{-\frac{2s}{p-1}}
\quad \textfor p\in(\tfrac{N}{N-2s},\tfrac{N+2s}{N-2s}],
\]
as $x\to 0$. Moreover, under the fractional Emden--Fowler transformation 
\[
u(r)=r^{-\frac{2s}{p-1}}v(t),
	\qquad
t=-\log r,
\]
the equation \eqref{eq:intro-1.4} for radial functions becomes 
\begin{equation}\label{eq:intro-PDE}
\PV\int_{\R}
    \tilde{K}_0(t-\bar{t})
    [v(t)-v(\bar{t})]
\,d\bar{t}
+b_{s,p} v(t)
=v(t)^{p},
\end{equation}
where the kernel $\tilde{K}_0$ has a singularity like the one-dimensional fractional Laplacian and decays exponentially at infinity. The left hand side is a conjugation of (conformal to) the fractional Laplacian in $\R^N$. Analogous to the local case $s=1$, we have that
\[
b_{s,p}\geq 0
    \quad \textfor p\geq \tfrac{N}{N-2s},
\]
and equality holds exactly when $p=\frac{N}{N-2s}$ (see \Cref{cor:K-int-0} and \Cref{lem:emden_change}). 
In this Serrin-critical case, as one of the main contributions of the present paper, we are able to show (in \Cref{sec:behavior-rad}) that the asymptotic behavior of $v$, which solves the integro-differential equation \eqref{eq:intro-PDE}, is still driven by an ODE
\begin{equation}\label{eq:intro-asymp-ODE}
{-}a_{s,\frac{N}{N-2s}}v_t=v^{\frac{N}{N-2s}},
    \quad \text{ for large $t$},
\end{equation}
which has exactly the form \eqref{eq:intro-ODE}. This is surprising and it seems to be the first instance where a nonlocal equation is associated directly to a scalar first order ODE, unlike in the Caffarelli--Silvestre extension \cite{CSext} or the infinite (coupled) system of second order ODEs introduced in \cite{fat,fat-survey}. From this formulation we deduce a short, yet complete, classification of local behavior of the fractional Lane--Emden--Serrin equation  (\Cref{thm:behavior-rad}), alongside the recent independent work of Wei and Wu \cite{WW} who adapted the classical work \cite{Aviles-1,Aviles-2} %
to the fractional case using the extension. As a new contribution, we determine the precise second order term where a $\log t$ correction is crucial, extending \cite{Veron1981} to the nonlocal  framework.

\medskip

In essence, the ODE \eqref{eq:intro-asymp-ODE} is obtained through the Emden--Fowler transform of \eqref{eq:intro-1.4} in its integral form. One crucially observes that the resulting equation \eqref{eq:v-2-rad} is asymptotically autonomous (i.e. its integral kernel is a constant up to a small error). 
Thus, we will proceed to determine the local behavior as follows:
\begin{enumerate}[(i)]
\item We bound the kernel in \eqref{eq:v-2-rad} from below and obtain a differential inequality \eqref{eq:v-3-3}. This implies the upper bound (\Cref{prop:upper-rough}).
\item We refine \eqref{eq:v-2-rad} to \eqref{eq:v-4-bar} using the upper bound and the superlinearity of the power (\Cref{lem:eq-v-4-bar}).
\item We employ a \emph{local} continuity argument (\Cref{prop:lower-rough}) to show that a singular lower bound (sufficient for absorbing the error) at any point propagates to the whole half-line. This yields the exact main order expansion \eqref{eq:v-exact}.
\item The second order is obtained by the linearized ODE around the first order term. After a delicate cancellation in the logarithmic scale (in the Emden--Fowler variable) we arrive at an exact ODE \eqref{eq:phi-ODE}, whose right-hand side is a reciprocal function.
\end{enumerate}
In the non-radial case, the argument in terms of the spherical average requires an extra use of Jensen's inequality for the first order term. Nonetheless, the spirit and procedure remain completely the same. These ideas are sketched in \Cref{sec:ODE-ideas} and are made rigorous in the rest of \Cref{sec:behavior-rad}.

\normalcolor

\normalcolor

\subsection{Non-existence of global solutions}

\normalcolor

Dancer, Du and Guo \cite{Dancer-Du-Guo2011} proved that non-negative exterior solutions to the Lane--Emden--Serrin equation must be trivial. Here the classification is possible because for the particular Serrin-critical exponent, model (power type) solutions have a logarithmic correction and there is no consistent way to assign the power of this logarithm at infinity. Thus, the equation only needs to be satisfied in an exterior domain. Indeed, the proof of \cite[Theorem 2.3]{Dancer-Du-Guo2011} is based on an ODE asymptotic analysis.

In \Cref{thm:Liouville-int} (see also \Cref{thm:Liouville-diff}), we generalize the classification in \cite{Dancer-Du-Guo2011} to the fractional setting using an integral asymptotic analysis. These Liouville theorems sharpen the ones in the literature in various ways. 
First, it is known from \cite{CLO2005,CLL2017} and \cite[Remark 1.2]{JLX} that 
\emph{entire} 
solutions to the fractional Lane-Emden equation with $0<p<\frac{N+2s}{N-2s}$ are trivial. For the particular exponent $\frac{N}{N-2s}$, we classify non-negative 
solutions to \eqref{eq:ext} that are possibly \emph{singular}. Secondly, 
it is remarked in \cite[Remark 2.2]{BQ} that solutions to \eqref{eq:ext-Kelvin} that are \emph{locally bounded} around the origin do not exist. In \Cref{thm:Liouville-diff}, no boundedness nor \emph{any} growth condition is assumed. Besides, the equation (for $u$) only needs to be satisfied in an \emph{exterior} domain.

\section{Geometric motivation}
\label{sec:geometric}

\subsection{Singular Yamabe problem}

The classical Yamabe problem in a compact Riemannian manifold $(M,g)$ asks for a conformal metric with constant scalar curvature \cite{LP}. In 1988, the profound study of Schoen and Yau \cite{SY} on the singularities of complete Yamabe metrics showed upper bounds on their Hausdorff dimension, which depend on the geometry of $M$ in general, and equal $\frac{n-2}{2}$ in $\bS^n$. For the construction of singular solutions in the latter case, we refer the readers to \Cref{table} on page \pageref{table}.
Moreover, any distributional solution to the associated PDE-problem was conjectured in \cite{SY} to provide a complete metric for the Yamabe problem. %
This has been disproved by Pacard \cite{Pacard-2} in ambient dimensions $n=4,6$ and Chen and Lin \cite{CL} for $n\geq 9$. %

\subsection{Singular fractional Yamabe problem}

The fractional Yamabe problem for $s\in(0,1)$ is then posed in parallel to the classical one:
finding a conformal metric with constant fractional curvature $Q_s^g:=P_s^g(1)$, a one-parameter family of intrinsic curvatures with good conformal properties \cite{CG,mar_survey}. Here $P_s^g$ is the conformal fractional Laplacian which is defined using the relation between scattering operators of asymptotically hyperbolic metrics \cite{GZ} (see also \cite{FG,MM}) and Dirichlet-to-Neumann operators of degenerate uniformly elliptic boundary value problems. In particular, $P_s^g$ satisfies the conformal property
\begin{equation}\label{eq:conf}
P^{g_v}_s (f)=v^{-\frac{n+2s}{n-2s}}P_s^g(v f),
	\qquad
g_v:=v^{\frac{4}{n-2s}}g,\,
v>0,
	\qquad
f\in C^\infty(M).
\end{equation}
Solvability results on general manifolds are shown in \cite{GQ,GW,KMW18,MN2}.

Now, we restrict our attention to the flat ambient manifold $(\R^n,|dz|^2)$ %
(which is conformal to round sphere via stereographic projection) where
$P_s^{|dz|^2}=\Dsn$. The conformal property \eqref{eq:conf} with $f=1$
yields the fractional Yamabe equation
\begin{equation*}
\Dsn v=Q_s^{g_v} v^{\frac{n+2s}{n-2s}}
	\quad \text{ in } \R^n.
\end{equation*}
By rescaling we may assume $Q_s^{g_v}=1$. As in the local case, it is known \cite{CLO2006,JLX} that the only regular solutions %
are the %
(fractional) bubbles  %
which represent the spherical metrics and are extremals of the fractional Sobolev inequality. 

In the singular fractional Yamabe problem, one allows the Yamabe metric to be singular, namely $v\to+\infty$ on approaching a singular submanifold $\Sigma$.
In \Cref{th:Yamabe}, we construct such $v$ when $s=1/2$ and $\dim\Sigma=(n-2s)/2$, a ``critical'' dimension to be clarified below.

\normalcolor

\subsubsection{Arbitrarily singular incomplete metrics}

The nonlocal analogue of the conjecture of Schoen and Yau fails as in the local case \cite{Pacard-2,CL}. Indeed, incomplete metrics blowing up on the whole $\R^n$ are constructed based on new fast-decay solutions (in the stability regime) and variational methods, as shown in \cite[Theorem 1.4]{ACGW} for suitable parameters $(n,s)$ including $n\geq 9$, $s\sim 1^-$. These building blocks opened up a way to the resolution of the singular fractional Yamabe problem when the singular set $\Sigma^k$ is a smooth submanifold \cite{fat}. %

\subsubsection{Dimensional restriction for complete metrics}

In the geometric problem, one asks that the solution to the Yamabe PDE still yields a complete metric, in the presence of a singularity of positive dimension $k\leq n$.
Gonz\'alez, Mazzeo and Sire \cite{GMS} found that %
$k$ necessarily verifies %
\begin{equation}\label{eq:Gamma-quot}
\Gamma\left(\tfrac{n-2k+2s}{4}\right)
\Big/
\Gamma\left(\tfrac{n-2k-2s}{4}\right)
\geq 0,
\end{equation}
which is satisfied, in particular, when $k\leq \frac{n-2s}{2}$. In fact, \eqref{eq:Gamma-quot} is equivalent to the condition
\begin{equation}\label{eq:Gamma-quot-equiv}
k\in\left(
        -\infty,\frac{n-2s}{2}
    \right]
\cup \bigcup_{j=0}^{\infty}
    \left(
        \frac{n+2s}{2}+4j,
        \frac{n-2s}{2}+2+4j
    \right)
\cup \bigcup_{j=0}^{\infty}
    \left(
        \frac{n+2s}{2}+2+4j,
        \frac{n-2s}{2}+4+4j
    \right].
\end{equation}
Note that each finite interval in the above union has length $2-2s$ and so they degenerate as $s\to 1^-$. In this case, \eqref{eq:Gamma-quot} reduces to the simple inequality $k\leq \frac{n-2}{2}$.

\begin{remark}
The condition \eqref{eq:Gamma-quot} was originally obtained by checking the strict positivity of the fractional curvature $Q_s$ when the conformal factor is a pure power of the distance $r$ to the model singular set (i.e. $\R^k$), namely $r^{-\frac{n-2s}{2}}$ in the notation of this paper. Thus, the condition \eqref{eq:Gamma-quot} is stated in \cite{GMS} as a strict inequality. However, since polyhomogeneous functions of $r$ are allowed, the positivity of $Q_s$ can still be achieved for the critical dimension with the help of a logarithmic factor (see \cite{Pacard-1, Chan_DelaTorre1} and \Cref{th:Yamabe}). For this reason, we have found it appropriate to reformulate the condition \eqref{eq:Gamma-quot} allowing equality.
\end{remark}

\subsubsection{State-of-the-art of construction of complete singular metrics}
\label{sec:state-art-construction}
\normalcolor

We will distinguish between four cases depending on the dimension $k$ of the singular set $\Sigma^k$. We summarize the known constructions in the following table and give detailed explanations below. We emphasize that the last case $k>\frac{n-2s}{2}$ represents a completely new phenomenon which is absent in the classical singular Yamabe problem.

\renewcommand{\arraystretch}{1.3}

\begin{center}
\begin{table}[!h]
\begin{tabular}{|c|c|c|c|c|}
\hline
$\dim(\Sigma)$ & $k=0$ & $k\in(0,\frac{n-2s}{2})$ & $k=\frac{n-2s}{2}$ & $k>\frac{n-2s}{2}$ \\
\hline
\hline
$s=1$ & \cite{Schoen1989-var,Schoen_iso,MP2} & \cite{MP} & \cite{Pacard-1,Chan_DelaTorre1} & Non-existence \cite{SY} \\
\hline
$s\in(0,1)$ & \cite{DPGW,ADGW,DG}
& \cite{fat} & \Cref{th:Yamabe} & \makecell{\Cref{conj:main} \\ on existence} \\
\hline
\end{tabular}
\medskip
\caption{Summary of constructions of complete singular Yamabe metrics}
\label{table}
\end{table}
\end{center}

The methods developed in \cite{fat,fat-survey} %
could not cover the singular submanifold of critical dimension $k=\frac{n-2s}{2}$ due to the limitations of the techniques used there, where homogeneity was crucial. Instead, we follow our construction in \cite{Chan_DelaTorre1}, an alternative to \cite{Pacard-1}, to prove \Cref{th:Yamabe}.

\bigskip
Whether the singularity dimension $\frac{n-2s}{2}$ is maximal depends on the precise notion of solution to the Yamabe equation. Standard bootstrap regularity shows the local boundedness of \emph{distributional} solutions to \eqref{eq:intro-1.1} with $p\in(1,\frac{N}{N-2s})$, i.e. $k>\frac{n-2s}{2}$.

Drastically different from the classical case, the Gamma quotient condition \eqref{eq:Gamma-quot}, as introduced in \cite{GMS}, can be satisfied for $k>\frac{n-2s}{2}$. It seems to be a highly non-trivial task to build (non-distributional) solutions singular on manifolds of such high dimensions. Nonetheless, from
\eqref{eq:Gamma-quot-equiv}, we conjecture that they exist for the following parameters.
\begin{conj}\label{conj:main}
Suppose $n\geq 3$, $s\in(0,1)$, $k\in (\frac{n+2s}{2},n)$ satisfy either
\begin{enumerate}
\item $\frac{n+2s}{2}+4j<k<\frac{n-2s}{2}+2+4j$, or
\item $\frac{n+2s}{2}+2+4j<k\leq \frac{n-2s}{2}+4+4j$,
\end{enumerate}
for some $j=0,1,\dots$
Then for any smooth compact $k$-dimensional submanifold $\Sigma^k$ without boundary in $\R^n$, there exists a complete Yamabe metric, understood in a generalized sense, that is singular on $\Sigma^k$. \end{conj}

\begin{remark}
Even in the model case, the conformal factor of the singular metrics cannot be understood classically because of the low integrability near the singularity. An extended notion of solutions (e.g. \cite{DSV-def})
would be necessary. We are pursuing this direction in a coming work. 
\end{remark}

\section{Main results}
\label{sec:main-results}

\subsection{Classification of local behavior}

We show the exact singular behavior of solutions to
\begin{equation}\label{eq:loc-beh}\begin{dcases}
		\Ds u=u^{p}
		& \text{ in } B_1\setminus\set{0},\\
		u=g
		& \text{ in } \R^N\setminus B_1,\\
		u>0
		& \text{ in } B_1\setminus \set{0},
\end{dcases}
\qquad p=\frac{N}{N-2s},
\end{equation}
for an exterior datum $g$ with finite Poisson integral, namely
\begin{equation}\label{eq:g-finite}
\norm[\tilde{L}^1_{2s}(\R^N\setminus B_1)]{g}
:=\int_{\R^N\setminus B_1}
\dfrac{
	|g(y)|
}{
	(|y|^2-1)^s
	|y|^N
}
\,dy
<+\infty.
\end{equation}
More precisely, we have the following.

\begin{thm}[Exact local behavior]
	\label{thm:behavior-rad}
	Let $u\in C^2(B_1\setminus\set{0})$ be a solution of \eqref{eq:loc-beh}--\eqref{eq:g-finite}.	
	Then, either $u$ has a removable singularity at the origin, or
	\[
	u(x)
	=\dfrac{1}{
		|x|^{N-2s}
		\left(\log\frac{1}{|x|}\right)^{\frac{N-2s}{2s}}
	}
	\left(
		c_0
		+c_1\frac{\log\log\frac{1}{|x|}}{\log\frac{1}{|x|}}
		+O\left(
			\frac{1}{\log\frac{1}{|x|}}
		\right)
	\right)
	\quad \textas x\to 0^+,
	\]
	where $c_0>0$, $c_1<0$ are constants given in \eqref{eq:c0-explicit} and \Cref{rmk:c1}.
\end{thm}

\begin{remark}
	 \Cref{thm:behavior-rad} remains valid for general smooth $\Omega\subset \R^N$ in place of $B_1$ provided $0\in\Omega$. %
\end{remark}

The $\log\log$-correction was first observed by L. V\'{e}ron \cite[Lemme 3.3]{Veron1981} in 1981 when $s=1$.

Our proof has a different flavor from Aviles \cite{Aviles-2} and is given in \Cref{sec:behavior-rad}, with main ideas sketched in \Cref{sec:ODE-ideas}.

Recall (from \Cref{subsec:LEeq}) that when $s=1$, the Emden--Fowler transformation $v$ of the solution $u$ satisfies the ODE
\[
-a_{\frac{N}{N-2}} \p_t v=v^{\frac{N}{N-2}},
\]
asymptotically as $t=-\log r\to+\infty$, for $a_{\frac{N}{N-2}}>0$. Surprisingly, by exploiting the integral equation associated to \eqref{eq:loc-beh} under the Emden--Fowler transformation, the local behavior in the fractional setting is also revealed to be driven by an ODE of the same form, namely \eqref{eq:v-4-bar}. %

We believe that our method is of independent interest, since we transform a nonlocal problem into the study of a local first order ODE in one dimension. This seems to be the first time that this technique is used in nonlocal problems.

\begin{remark}
	We emphasize that the scalar ODE \eqref{eq:v-4-bar} is valid only \emph{asymptotically}. Indeed, a nonlocal equation cannot be equivalent to a local one, as independently observed in %
	\cite[Proposition 3.2]{Chen2021}.
\end{remark}

\begin{remark}
	In dimension $N=1$, the half Laplacian can be factorized as the composition of the Hilbert transformation $H$ and the ordinary first derivative. Since $H^2=-\mathrm{Id}$, inverting $H$ leads to the equivalence
	\[
	\Dh u=f(x,u)
	\quad \text{ in } \R
	\qquad \Longleftrightarrow \qquad
	u'=-H(f(x,u))
	\quad \text{ in } \R.
	\]
	This can be considered as a global analogue of \eqref{eq:v-4-bar}. We thank Enno Lenzmann for pointing this out to us during the conference ``Calculus of Variations and PDEs: recent developments and future directions'' at ETH Z\"{u}rich in June 2021.
\end{remark}

{

The equivalent classification has been done for the whole range $p\in(1,\tfrac{N+2s}{N-2s}] \setminus \{\tfrac{N}{N-2s}\}$ in various works by Caffarelli, Jin, Sire and Xiong, Chen and Quaas, and Yang and Zou in \cite{CJSX,CQ,YZ1,YZ2}. %
(See also %
\cite{JdQSX} for the study of local behaviors near higher dimensional singularities.)
During the preparation of this manuscript, we have learned about the recent work of Wei and Wu \cite{WW} based on adapting Aviles's energy method 
to the nonlocal setting by using the extension. Here we provide an alternative proof which allows to cover equations with more general integro-differential operators that admit a Green function comparable to the Riesz kernel. %
\color{blue}
}

\bigskip

\subsection{Non-existence of exterior solutions}

We now give Liouville-type results inspired by asymptotic expansions. As in the local case, the equation
\begin{equation}\label{eq:ext}
\begin{dcases}
	\Ds u=u^{\frac{N}{N-2s}}
	& \text{ in } \R^N\setminus B_1,\\
	u\geq 0
	& \text{ in } \R^N\setminus \set{0}. %
\end{dcases}\end{equation}
in an exterior domain does not possess non-trivial solutions.

\begin{thm}[Liouville theorem for exterior solutions, integral version]
	\label{thm:Liouville-int}
	Suppose $u \in L^{\frac{N}{N-2s}}_{-2s}(\R^N)$  satisfies %
	the integral inequality
	\begin{equation*}
		\begin{dcases}
			u(x)
			\geq
			\int_{\R^N\setminus B_1}
			\dfrac{
				u(y)^{\frac{N}{N-2s}}
			}{
				|x-y|^{N-2s}
			}
			\,dy
			& \,\forall x\in\R^N\setminus B_1,\\
			u\geq 0
			& \text{ in } \R^N\setminus\set{0},
		\end{dcases}
	\end{equation*}
	and the Harnack inequality
	\begin{equation}\label{eq:Harnack-statement}
		\sup_{B_{2R}\setminus B_{R/2}} u
		\leq C \inf_{B_{2R}\setminus B_{R/2}} u,
		\qquad \forall R>4.
	\end{equation}
	Then $u\equiv 0$ in $\R^N \setminus B_1$.
\end{thm}
\begin{thm}[Liouville theorem for exterior solutions, integro-differential version]
	\label{thm:Liouville-diff}
	Suppose  $u\in L^1_{2s}(\R^N)\cap L^{\frac{N}{N-2s}}_{-2s}(\R^N)$ is a distributional solution to \eqref{eq:ext}.
	Then, $u \equiv 0$ in $\R^N$.
\end{thm}

The proof is given in \Cref{sec:Liouville}, in terms of the Kelvin transform of $u$. %
It is based on an asymptotic analysis (on $u$) at infinity, where no consistent behavior is possible. Heuristically, one would expect that the solution is polyhomogeneous as in \eqref{eq:logpoly-infty}, while no such function would satisfy \eqref{eq:ext},
as seen by the formal computation (justified in \Cref{prop:loghomo-infty})
\begin{equation}\label{eq:asymp-infty}
\Ds \dfrac{(\log r)^\nu}{r^{N-2s}}
\sim \nu\kappa_1 \dfrac{(\log r)^{\nu-1}}{r^N},
	\qquad \text{ as } r\to\infty,
\end{equation}
where $\kappa_1>0$. 
Indeed, no $\nu>0$ would solve $\nu-1=\frac{N}{N-2s}\nu$. We emphasize that \eqref{eq:asymp-infty} is to be contrasted with
\begin{equation*}
	\Ds \dfrac{1}{r^{N-2s}(\log\frac1r)^{\nu}}
	\sim \nu\kappa_1 \dfrac{1}{r^N(\log\frac1r)^{\nu+1}},
	\qquad \text{ as } r\to 0^+,
\end{equation*}
where $\nu=\frac{N-2s}{2s}$ does give the correct asymptotic behavior (see \Cref{prop:loghomo}). Here the \emph{opposite sign} in $\nu$ can be traced back to the definition of $\kappa_1$ in \eqref{eq:Ki}. %

\subsection{Existence of profile}
Concerning the existence of radial singular solutions to
\begin{equation}\label{eq:main}\begin{cases}
		\Ds u=u^{\frac{N}{N-2s}}
		& \text{ in } B_1\setminus\set{0},\\
		u=0
		& \text{ in } \R^N\setminus B_1,
\end{cases}\end{equation}
we have (writing $u(x)=u(r)$ for $r=|x|$):
\begin{thm}\label{thm:rad-s}
	Let $s\in(0,1)$, $N>2s$. There exists $\eps_1\in(0,1)$ such that for any $\eps\in(0,\eps_1)$, \eqref{eq:main} has a positive solution $u_\eps \in L^{\frac{N}{N-2s}}(B_1) \cap C^{\infty}(B_1\setminus \set{0}) \cap C^s(\R^N\setminus\set{0})$ such that,%
	\[
	u_\eps(r)
	=
	\begin{dcases}
		\dfrac{1}{
			r^{N-2s}
			\left(\log\frac{1}{\eps r}\right)^{\frac{N-2s}{2s}}
		}
		\left[
		c_0
		+c_1\dfrac{
			\log\log\frac{1}{\eps r}
		}{
			\log \frac{1}{\eps r}
		}
		+O\left(
		\dfrac{
			|\log\eps|^{-\gamma}
		}{
			(\log\frac{1}{r})^{\frac32}
		}
		\right)
		\right]
		& \textas r\to 0^+,\\
		O\Big(
		|\log\eps|^{-\gamma}
		(1-r^2)^s
		\Big)
		& \textas r\to 1^-.
	\end{dcases}\]
	where $\gamma=\min\set{\frac{N-2s}{2s},\frac12}$ and the constants $c_i=c_i(N,s)$ are defined in \eqref{eq:c0}--\eqref{eq:c1}.
\end{thm}

Our proof, given in \Cref{sec:radial}, is based on fine asymptotic expansions obtained using the integro-differential operator, combined with barrier and fixed point arguments.   As in our previous work \cite{Chan_DelaTorre1}, we employ gluing methods in weighted $L^\infty$ spaces to avoid unnecessary derivative estimates. %

By comparing the formulae for the constants $c_0$, $c_1$ from \Cref{thm:behavior-rad} (through $\Ints$) and from \Cref{thm:rad-s} (through $\Ds$), we obtain two integral identities (see \Cref{lem:const-int}).

\subsection{Yamabe metric with high dimensional singularity}

\begin{thm}\label{th:Yamabe}
Let $n\geq 3$ be an odd integer, $k=\frac{n-1}{2}$ and $\Sigma\subset\R^n$ be a $k$-dimensional smooth compact submanifold without boundary. Then there exists a smooth positive solution to
\[
\Dh u=u^{\frac{n+1}{n-1}}
    \quad \text{ in } \R^n \setminus \Sigma,
\]
which grows like $\dfrac{c_0}{d_\Sigma^{N-1}(\log\frac{1}{d_\Sigma})^{N-1}}$ near %
$\Sigma$ %
and decays like $d_\Sigma^{-(n-1)}$ at infinity. Here $d_\Sigma$ denotes the distance function to $\Sigma$ and $c_0$ is given in \eqref{eq:c0}. Moreover, the associated Yamabe metric is complete.
\end{thm}

Equivalently, by considering the harmonic extension of the conformal factor, the upper-half space $(\R^{n+1}_+, |dz|^2)$ admits a scalar-flat Escobar metric (conformal metric with zero scalar curvature and constant mean curvature on the boundary) which is singular along a $\frac{n-1}{2}$-submanifold $\Sigma$ of $\p\R^{n+1}_+=\R^n$: %
\begin{equation*}%
\begin{cases}
-\Delta u=0
    & \text{ in } \R^{n+1}_+,\\
\p_{\nu}u=u^{\frac{n+1}{n-1}}
    & \texton \p\R^{n+1}_+\setminus \Sigma.
\end{cases}\end{equation*}

The proof of \Cref{th:Yamabe} is given in \Cref{sec:Yamabe}. The approximate solution obtained from \Cref{thm:rad-s} with $s=\frac12$ %
induces an initial error which is sufficiently small in terms of both size and order of growth. Indeed, the error comes from the geometry of the singularity and is algebraic instead of logarithmic. The main difficulty lies in the construction of barriers via nonlocal computations %
in Fermi coordinates. %

\bigskip

\subsection{Arbitrary closed singularity}

\begin{thm}
	Let $\Omega\subset\R^N$ be a bounded domain and $\Sigma\subset\Omega$ be a closed subset. Then there exist two distinct sequences $u_\ell^{(1)}$, $u_\ell^{(2)}$ of positive very weak solutions to
	\begin{equation*}%
		\begin{cases}
			\Ds u=u^{\frac{N}{N-2s}}
			& \text{ in } \Omega,\\
			u=0
			& \text{ in } \R^N\setminus\Omega,
	\end{cases}\end{equation*}
	such that $u_\ell(x)\to+\infty$ as $x\to\Sigma$. Moreover, $u_\ell^{(1)} \to 0$ in $L^{\frac{N}{N-2s}}(\Omega)$ and $u_\ell^{(2)}$ converges to a non-trivial regular solution in $L^{\frac{N}{N-2s}}(\Omega)$.
\end{thm}

The proof is essentially a direct generalization of the variational argument of Pacard \cite{Pacard-3} to the nonlocal case (see also \cite{CL,ACGW}). A sketch can be found in \cite{CD-survey}.

\subsection{Radial symmetry}

By the direct method of moving plane, we have:

\begin{prop}[\!\!\cite{CLL2017, Dou2016}]
\label{prop:sym}
Let $u>0$ be a classical solution to
\[\begin{dcases}
\Ds u=u^{\frac{N}{N-2s}}
    & \text{ in } B_1\setminus\set{0},\\
u=0
    & \text{ in } \R^N\setminus B_1.
\end{dcases}\]
Then $u=u(r)$ is radially symmetric and $u'(r)<0$ for $r\in(0,1)$.
\end{prop}

In \cite{CD-survey}, we give a unified argument regardless of the removability of the singularity.

\section{Classification of local behavior}
\label{sec:behavior-rad}

In this section we prove \Cref{thm:behavior-rad}. The novelty is to transform the one-dimensional nonlocal equation into a \emph{first order ODE} in the asymptotic regime, via the corresponding integral equation. %
This seems to be new in the ``non-local community", although its root lies in the theory of delay differential equations. This method is robust for nonlocal operators whose Green's kernel is comparable to Riesz potential in the interior.

\subsection{Sketch of main ideas}
\label{sec:ODE-ideas}

Let us illustrate the key idea 
by presenting the argument for radially symmetric solutions $u(r)$ to the homogeneous Dirichlet problem, which in integral formulation reads $u=\cG_{B_1}[u^{\frac{N}{N-2s}}]$. Thus the Emden--Fowler transform $v(t)=r^{N-2s}u(r)$, $t=-\log r$, solves
	\begin{equation*}%
		v(t)
		=\int_{0}^{\infty}
		\int_{0}^{\pi}
		F\left(
		\dfrac{
			(e^{2t}-1)
			(1-e^{-2\bar{t}})
		}{
			1+e^{2(t-\bar{t})}
			-2e^{t-\bar{t}}\cos\theta
		}
		\right)
		\dfrac{
			\cH^{N-2}(\bS^{N-2})
			\sin^{N-2}\theta
			\,d\theta
		}{
			\bigl(
			1+e^{2(t-\bar{t})}
			-2e^{t-\bar{t}}\cos\theta
			\bigr)^{\frac{N-2s}{2}}
		}
		v(\bar{t})^{\frac{N}{N-2s}}
		\,d\bar{t},
	\end{equation*}
for $t\geq 0$, where $F(R)$ is a bounded increasing function with $F(0)=0$ and $F(+\infty)-F(R) \leq CR^{-\frac{N-2s}{2}}$, explicitly given in \eqref{eq:F}. We consider the asymptotics for large $t$. Since the singular kernel (which will be reduced to $\tilde\cK$ below) is exponentially small for $\bar{t}\leq \frac{t}{2}$, and $F$ approaches, as $\bar{t}\geq \frac{t}{2}\to+\infty$, the constant $F(+\infty)$ exponentially, we see that
	\begin{equation}\label{eq:v-2-rad}
	v(t)
	=\int_{\frac{t}{2}}^{\infty}
\tilde	\cK(t-\bar{t})
	v(\bar{t})^{\frac{N}{N-2s}}
	\,d\bar{t}
	+O(e^{-\frac{N-2s}{2}t}),
	\quad
	\tilde	\cK(t-\bar{t})
		\asymp
		\begin{dcases}
			1
			& \textfor t-\bar{t}\leq -1,\\
			\frac{1-|t-\bar{t}|^{2s-1}}{2s-1}
			& \textfor |t-\bar{t}|\leq 1,\\
		e^{-(N-2s)(t-\bar{t})}
			& \textfor t-\bar{t} \geq 1.
		\end{dcases}
	\end{equation}
Let us neglect the exponential error. Using Harnack's inequality on $\bar{t}\in[t-1,t+1]$ and the integrability of $\tilde\cK(t-\bar{t})$ there, we obtain
\[
v(t)
\geq c\int_t^\infty
	v(\bar{t})^{\frac{N}{N-2s}}
\,d\bar{t}.
\]
This is just an ordinary differential inequality in disguise:
\[
[-V'(t)]^{\frac{N-2s}{N}} \geq c V(t),
	\qquad \text{ for } \quad
V(t):=\int_t^\infty v(\bar{t})^{\frac{N}{N-2s}} \,d\bar{t}.
\]
It implies immediately $V(t)\leq Ct^{-\frac{N-2s}{2s}}$, and the same bound for $v(t)$. From this upper bound and the superlinearity of the power, we can refine \eqref{eq:v-2-rad} to
		\begin{equation*}%
			v(t)
			=
			\kappa
			\int_{t}^{\infty}
			v(\bar{t})^{\frac{N}{N-2s}}
			\,d\bar{t}
			+O\left(
			t^{-\frac{N}{2s}}\log t
			\right),
		\end{equation*}
for an explicit constant $\kappa>0$. Now, either the singularity is removable, or $v$ is eventually large enough to absorb the error and, from that point on, this ODE yields the exact singular behavior at the first order.

To get to the second order, we need a precise expansion of the integral of the kernel $\tilde \cK$ in \eqref{eq:v-2-rad} over an interval of length $C\log t$ ($C\gg 1$), so that terms of order $O(\log t)$ are cancelled. The equation improves to
\begin{align*}
v(t)
&=\kappa\int_{t}^{\infty}
	 v(\bar{t})^{\frac{N}{N-2s}}
\,d\bar{t}
+\tilde{\kappa} t^{-\frac{N}{2s}}
+O\Big(t^{-\frac{N}{2s}-1}(\log t)^3\Big),
\end{align*}
where $\tilde{\kappa}<0$ is another exact constant. Finally, the second order correction $\phi(t)$ satisfies
\begin{align*}
\phi(t)
&=\frac{N}{2s}\int_{t}^{\infty}
	\bar{t}^{-1}
	\phi(\bar{t})
\,d\bar{t}
+\tilde{\kappa} t^{-\frac{N}{2s}}
+O\Big(t^{-\frac{N}{2s}-1}(\log t)^4\Big),
\end{align*}
an ODE with integrating factor $t^{\frac{N-2s}{2s}}$,
due to the coefficient $\frac{N}{2s}$ in front of the integral. The term $\tilde{\kappa}t^{-1}$ leads to the exact logarithmic second order behavior.

\bigskip

In the rest of this \Cref{sec:behavior-rad} we prove \Cref{thm:behavior-rad} in details. Let us point out the missing ingredients:
\begin{itemize}
\item The error in \eqref{eq:v-2-rad} is not necessarily signed. Thus, one needs to argue that if the upper bound for $v$ (or equivalently for $V$) eventually does not hold, then the error is absorbed in a suitably large region where the ODE argument can be carried out to yield a contradiction.
\item If the exterior data is not radial, one requires Jensen's (or Harnack's) inequality when taking spherical means. While equality is almost achieved, the error has to be estimated quantitatively.
\end{itemize}

\normalcolor

\subsection{Integral Emden--Fowler transformation}

Note that $u\in L^{\frac{N}{N-2s}}(B_1) \cap \tilde{L}^1_{2s}(\R^N\setminus B_1)$ is also a Green--Poisson solution to \eqref{eq:loc-beh}, namely, in polar coordinates,
\[
u(r,\omega)
=\cG_{B_1}[u^{\frac{N}{N-2s}}](r,\omega)
    +\cP_{B_1}[g](r,\omega),
\quad \text{ for } r\in(0,1),\,\omega\in\bS^{N-1}.
\]%
We define the Emden--Fowler transform
\begin{equation}\label{eq:uv}
v(t,\omega):=r^{N-2s}u(r,\omega),
	\qquad
t=-\log r.
\end{equation}
Note that from $u\in L^{\frac{N}{N-2s}}(B_1)$ we know that $v\in L^{\frac{N}{N-2s}}((0,+\infty)\times\bS^{N-1})$. %
\begin{lem}
It holds
\begin{equation}\label{eq:v-1}
\begin{split}
	v(t,\omega)
	&=\int_{0}^{\infty}
	\int_{\bS^{N-1}}
	F\left(
	\dfrac{
		(e^{2t}-1)
		(1-e^{-2\bar{t}})
	}{
		1+e^{2(t-\bar{t})}
		-2e^{t-\bar{t}}
		\omega\cdot\bar{\omega}
	}
	\right)
	\dfrac{
		v(\bar{t},\bar{\omega})^{\frac{N}{N-2s}}
	}{
		\left(
		1+e^{2(t-\bar{t})}
		-2e^{t-\bar{t}}
		\omega\cdot\bar{\omega}
		\right)^{\frac{N-2s}{2}}
	}
	\,d\cH^{N-1}_{\bar{\omega}}
	\,d\bar{t}\\
	&\quad\;
	+e^{-(N-2s)t}\cP_{B_1}[g](e^{-t},\omega),
\end{split}\end{equation}
for $t\geq 0$ and $\omega\in\bS^{N-1}$, where
	\begin{equation}\label{eq:F}
	F(R)
	=
	\dfrac{
		\Gamma(\frac{N}{2})
	}{
		2^{2s}
		\pi^{\frac{N}{2}}
		\Gamma(s)^2
	}
	\int_{0}^{R}
	\dfrac{
		\tau^{s-1}
	}{
		(\tau+1)^{\frac{N}{2}}
	}
	\,d\tau
	\end{equation}
is a non-negative, bounded increasing function that satisfies $F(+\infty)=C_{N,-s}$, the normalization constant for the Riesz kernel given in \eqref{eq:C-Ns}.
\end{lem}

\begin{proof}
Recalling \eqref{eq:green_B1}, we have
	\[
	u(x)
	=\int_{B_1}
	F\left(
	\dfrac{
		(1-|x|^2)
		(1-|\bar{x}|^2)
	}{
		|x-\bar{x}|^2
	}
	\right)
	\dfrac{
		u(\bar{x})^{\frac{N}{N-2s}}
	}{
		|x-\bar{x}|^{N-2s}
	}
	\,d\bar{x}
	+\cP_{B_1}[g](x).
	\]
In polar coordinates $x=(r,\omega)$, $\bar{x}=(\bar{r},\bar{\omega})$, the function $v(t,\omega)$ as defined in \eqref{eq:uv} satisfies %
	\[\begin{split}
	v(-\log r,\omega)
	&=\int_{0}^{1}
	\int_{\bS^{N-1}}
	F\left(
	\dfrac{
		(r^{-2}-1)
		\bigl(
		1-\bar{r}^2
		\bigr)
	}{
		1+
		\bigl(
		\frac{\bar{r}}{r}
		\bigr)^2
		-2\frac{\bar{r}}{r}\omega\cdot\bar{\omega}
	}
	\right)
	\dfrac{
		v(-\log\bar{r},\bar{\omega})^{\frac{N}{N-2s}}
	}{
		\bigl(
		1+
		\bigl(
		\frac{\bar{r}}{r}
		\bigr)^2
		-2\frac{\bar{r}}{r}\omega\cdot\bar{\omega}
		\bigr)^{\frac{N-2s}{2}}
	}
	\,d\cH^{N-1}_{\bar{\omega}}
	\,\dfrac{d\bar{r}}{\bar{r}}\\
&\quad\;
	+r^{N-2s}\cP_{B_1}[g](r,\omega).
\end{split}\]
	Changing variables to $r=e^{-t}$, $\bar{r}=e^{-\bar{t}}$ yields \eqref{eq:v-1}.
\end{proof}

\subsection{Upper bound}

We start with a restatement of the Harnack inequality \cite[Proposition 3.1]{YZ1}. %

\begin{lem}\label{lem:harnack-t}
	For any $t\geq 2$,
	\[
	\sup_{[t-1,t+1]\times \bS^{N-1}}v
	\leq C\inf_{[t-1,t+1]\times \bS^{N-1}}v.
	\]
\end{lem}
We control the contribution in the region close to the boundary $\bar{t}=0$ by an exponentially small error.
\begin{lem}
For any $t\geq 4$ and $\omega\in\bS^{N-1}$,
\begin{equation}\label{eq:v-2}
	v(t,\omega)
	=\int_{t/2}^{\infty}\int_{\bS^{N-1}}
		\cK(t-\bar{t},\omega\cdot\bar{\omega})
		v(\bar{t},\bar{\omega})^{\frac{N}{N-2s}}
	\,d\cH^{N-1}_{\bar{\omega}}\,d\bar{t}
	+O(e^{-\frac{N-2s}{2}t}),
\end{equation}
where
\begin{equation}\label{eq:cK}
		\cK(t-\bar{t},\omega\cdot\bar{\omega})
		:=
		\dfrac{
			C_{N,-s}
		}{
			\bigl(
			1+e^{2(t-\bar{t})}
			-2e^{t-\bar{t}}\omega\cdot\bar{\omega}
			\bigr)^{\frac{N-2s}{2}}
		}
		\asymp
		\begin{dcases}
			1
			& \textfor t-\bar{t}\leq -1,\\
		e^{-(N-2s)(t-\bar{t})}
			& \textfor t-\bar{t} \geq 1,
		\end{dcases}
\end{equation}
\begin{equation}\label{eq:cK-int}
\int_{t-1}^{t+1}\int_{\bS^{N-1}}
		\cK(t-\bar{t},\omega\cdot\bar{\omega})
\,d\cH^{N-1}_{\bar{\omega}}\,d\bar{t}
\asymp 1,
	\quad \text{ independently of $t\geq 2$ and $\omega\in\bS^{N-1}$}.
\end{equation}
Here the constant in the error depends only on $N$, $s$, $\norm[L^{\frac{N}{N-2s}}((0,\infty)\times\bS^{N-1})]{v}$ and $\norm[\tilde{L}^1_{2s}(\R^N\setminus B_1)]{g}$.
\end{lem}

\begin{proof}
We rewrite \eqref{eq:v-1}, singling out the dominant term, as
\begin{equation*}\begin{split}
	v(t,\omega)
	&=\int_{t/2}^{\infty}
	\int_{\bS^{N-1}}
	F(+\infty)
	\dfrac{
		v(\bar{t},\bar{\omega})^{\frac{N}{N-2s}}
	}{
		\left(
		1+e^{2(t-\bar{t})}
		-2e^{t-\bar{t}}
		\omega\cdot\bar{\omega}
		\right)^{\frac{N-2s}{2}}
	}
	\,d\cH^{N-1}_{\bar{\omega}}
	\,d\bar{t}\\
	&\quad\;
	-\int_{t/2}^{\infty}
	\int_{\bS^{N-1}}
	\left[
	F(+\infty)
	-F\left(
	\dfrac{
		(e^{2t}-1)
		(1-e^{-2\bar{t}})
	}{
		1+e^{2(t-\bar{t})}
		-2e^{t-\bar{t}}
		\omega\cdot\bar{\omega}
	}
	\right)
	\right]
	\dfrac{
		v(\bar{t},\bar{\omega})^{\frac{N}{N-2s}}
		\,d\cH^{N-1}_{\bar{\omega}}
		\,d\bar{t}
	}{
		\left(
		1+e^{2(t-\bar{t})}
		-2e^{t-\bar{t}}
		\omega\cdot\bar{\omega}
		\right)^{\frac{N-2s}{2}}
	}\\
	&\quad\;
	+\int_{0}^{t/2}
	\int_{\bS^{N-1}}
	F\left(
	\dfrac{
		(e^{2t}-1)
		(1-e^{-2\bar{t}})
	}{
		1+e^{2(t-\bar{t})}
		-2e^{t-\bar{t}}
		\omega\cdot\bar{\omega}
	}
	\right)
	\dfrac{
		v(\bar{t},\bar{\omega})^{\frac{N}{N-2s}}
	}{
		\left(
		1+e^{2(t-\bar{t})}
		-2e^{t-\bar{t}}
		\omega\cdot\bar{\omega}
		\right)^{\frac{N-2s}{2}}
	}
	\,d\cH^{N-1}_{\bar{\omega}}
	\,d\bar{t}\\
	&\quad\;
	+e^{-(N-2s)t}\cP_{B_1}[g](e^{-t},\omega)\\
&=:\int_{t/2}^{\infty}\int_{\bS^{N-1}}
	\cK(t-\bar{t},\omega\cdot\bar{\omega})
	v(\bar{t},\bar{\omega})^{\frac{N}{N-2s}}
\,d\cH^{N-1}_{\bar{\omega}}\,d\bar{t}
+I_1+I_2+I_3.
\end{split}\end{equation*}
We estimate the error terms one by one. For $I_1$, since $F$ is increasing, for any $\bar{t}\geq \frac{t}{2} \geq 2$ and $\omega,\bar{\omega}\in\bS^{N-1}$,
	\[\begin{split}
		F(+\infty)
		-F\left(
		\dfrac{
			(e^{2t}-1)
			(1-e^{-2\bar{t}})
		}{
			1+e^{2(t-\bar{t})}
			-2e^{t-\bar{t}}\omega\cdot\bar{\omega}
		}
		\right)
		&
		\leq
		F(+\infty)-F(c e^{t})
		\leq
		C\int_{c e^{t}}^{\infty}
		\tau^{-\frac{N-2s}{2}-1}
		\,d\tau
		\leq
		Ce^{-\frac{N-2s}{2}t}.
	\end{split}\]
Using \Cref{lem:harnack-t} twice, the singular kernel is bounded in $L^1$ when $|t-\bar{t}|\leq 1$ and in $L^\infty$ otherwise:
\begin{align*}
|I_1|
&\leq Ce^{-\frac{N-2s}{2}t}
	\left (
		\int_{t/2}^{t-1}
		+\int_{t+1}^{\infty}
	\right )
	\int_{\bS^{N-1}}
		\frac{
			v(\bar{t},\bar{\omega})^{\frac{N}{N-2s}}
		}{
			|e^{t-\bar{t}}-1|^{N-2s}
		}
	\,d\cH^{N-1}_{\bar{\omega}}\,d\bar{t}\\
&\quad\;
	+Ce^{-\frac{N-2s}{2}t}
	\int_{t-1}^{t+1}
		\int_{\bS^{N-1}}
			\frac{
				v(t,e_1)
			}{
				\left(
				1+e^{2(t-\bar{t})}
				-2e^{t-\bar{t}}
				\omega\cdot\bar{\omega}
				\right)^{\frac{N-2s}{2}}
			}
	\,d\cH^{N-1}_{\bar{\omega}}\,d\bar{t}\\
&\leq C\norm[L^{\frac{N}{N-2s}}((0,\infty)\times\bS^{N-1})]{v}^{\frac{N}{N-2s}}e^{-\frac{N-2s}{2}t}.
\end{align*}
For $I_2$ we get exponential decay from the singular kernel from $t-\bar{t}\geq \frac{t}{2}$, namely
\begin{align*}
|I_2|
&\leq CF(+\infty)
	\int_{0}^{t/2}\int_{\bS^{N-1}}
	\frac{
		v(\bar{t},\bar{\omega})
	}{
		|e^{t-\bar{t}}-1|^{N-2s}
	}
	\,d\cH^{N-1}_{\bar{\omega}}\,d\bar{t}
\leq C\norm[L^{\frac{N}{N-2s}}((0,\infty)\times\bS^{N-1})]{v}^{\frac{N}{N-2s}}e^{-\frac{N-2s}{2}t}.
\end{align*}
For $I_3$ we simply use the bound of the Poisson integral \eqref{eq:Poisson-bound} so that, recalling $t\geq 4$,
\[
|I_3|
\leq C\norm[\tilde{L}^1_{2s}(\R^N\setminus B_1)]{g}
	e^{-(N-2s)t}.
\]
Combining these estimates, \eqref{eq:v-2} follows. The properties of $\cK$ are immediate.
\end{proof}

We now replace the integrable, singular kernel by a regular one, using Harnack's inequality.

\begin{lem}
	For any $t\geq 4$,
\begin{align}
\label{eq:v-3-1}
	v(t,\omega)
\geq c\int_{t/2}^{\infty}
e^{-(N-2s)(t-\bar{t})_+}
\int_{\bS^{N-1}}
v(\bar{t},\bar{\omega})^{\frac{N}{N-2s}}
\,d\cH^{N-1}_{\bar{\omega}}
\,d\bar{t}
-Ce^{-\frac{N-2s}{2}t},\\
\label{eq:v-3-2}
	v(t,\omega)
\leq C\int_{t/2}^{\infty}
e^{-(N-2s)(t-\bar{t})_+}
\int_{\bS^{N-1}}
v(\bar{t},\bar{\omega})^{\frac{N}{N-2s}}
\,d\cH^{N-1}_{\bar{\omega}}
\,d\bar{t}
+Ce^{-\frac{N-2s}{2}t},
\end{align}
where $C,c>0$ are constants depending only on $N$, $s$, $\norm[L^{\frac{N}{N-2s}}((0,\infty)\times\bS^{N-1})]{v}$ and $\norm[\tilde{L}^1_{2s}(\R^N\setminus B_1)]{g}$.
\end{lem}

\begin{proof}
We recall from \eqref{eq:cK} that $\cK(t-\bar{t},\omega\cdot\bar{\omega}) \asymp e^{-(N-2s)(t-\bar{t})_+}$ for $|t-\bar{t}|\geq 1$ and $\omega,\bar{\omega}\in\bS^{N-1}$. For $|t-\bar{t}|\leq 1$ we use \Cref{lem:harnack-t} twice and \eqref{eq:cK-int} to obtain
\begin{align*}
\int_{t-1}^{t+1}\int_{\bS^{N-1}}
	\cK(t-\bar{t},\omega\cdot\bar{\omega})
	v(\bar{t},\bar{\omega})^{\frac{N}{N-2s}}
\,d\cH^{N-1}_{\bar{\omega}}\,d\bar{t}
&\asymp 
v(t,\omega)^{\frac{N}{N-2s}}
\int_{t-1}^{t+1}\int_{\bS^{N-1}}
\cK(t-\bar{t},\omega\cdot\bar{\omega})
\,d\cH^{N-1}_{\bar{\omega}}\,d\bar{t}
\\
\asymp v(t,\omega)^{\frac{N}{N-2s}}
&\asymp v(t,\omega)^{\frac{N}{N-2s}}
\int_{t-1}^{t+1}\int_{\bS^{N-1}}
	e^{-(N-2s)(t-\bar{t})_+}
\,d\cH^{N-1}_{\bar{\omega}}\,d\bar{t}
\\
&
\asymp \int_{t-1}^{t+1}\int_{\bS^{N-1}}
	e^{-(N-2s)(t-\bar{t})_+}
	v(\bar{t},\bar{\omega})^{\frac{N}{N-2s}}
\,d\cH^{N-1}_{\bar{\omega}}\,d\bar{t}.
\end{align*}
This completes the proof.
\end{proof}

In the following, we are interested in spherical average
\[
\overline{v}(t)
:=\fint_{\bS^{N-1}}
	v(t,\omega)
\,d\cH^{N-1}_{\omega}.
\]
It is clear that $\overline{v}$ also satisfies Harnack's inequality. In fact,
\begin{equation}\label{eq:harnack-v-bar}
\sup_{[t-1,t+1]}\overline{v}
\leq \sup_{[t-1,t+1]\times\bS^{N-1}}v
\leq C\inf_{[t-1,t+1]\times\bS^{N-1}}v
\leq C\inf_{[t-1,t+1]}\overline{v}.
\end{equation}
Moreover, we will show that $\overline{v}(t)$ satisfies an ordinary differential inequality and hence a sharp upper bound.

\begin{prop}\label{prop:upper-rough}
	For any $t\geq 8$ and $\omega\in\bS^{N-1}$,
	\[
	v(t,\omega)
	\leq Ct^{-\frac{N-2s}{2s}}.
	\]
Here $C>0$ depends only on $N$, $s$, $\norm[L^{\frac{N}{N-2s}}((0,\infty)\times\bS^{N-1})]{v}$ and $\norm[\tilde{L}^1_{2s}(\R^N\setminus B_1)]{g}$.
\end{prop}

\begin{proof}
Using \eqref{eq:v-3-1} and Jensen's inequality (or \eqref{eq:harnack-v-bar}), %
for any $t\geq 4$,
\begin{align*}
v(t,\omega)
&\geq
	c\int_{t}^{\infty}
		\fint_{\bS^{N-1}}
			v(\bar{t},\bar{\omega})^{\frac{N}{N-2s}}
		\,d\cH^{N-1}_{\bar{\omega}}
	\,d\bar{t}
	-Ce^{-\frac{N-2s}{2}t}\\
&\geq
	c\int_{t}^{\infty}
	\left (
		\fint_{\bS^{N-1}}
			v(\bar{t},\bar{\omega})
		\,d\cH^{N-1}_{\bar{\omega}}
	\right )^{\frac{N}{N-2s}}
	\,d\bar{t}
	-Ce^{-\frac{N-2s}{2}t}.
\end{align*}
Averaging over $\bS^{N-1}$, we have
\begin{equation}\label{eq:v-3-3}
\overline{v}(t)
\geq c\int_{t}^{\infty}
		\overline{v}(\bar{t})^{\frac{N}{N-2s}}
	\,d\bar{t}
	-Ce^{-\frac{N-2s}{2}t},
		\qquad \text{ for } t\geq 4.
\end{equation}
By \Cref{lem:harnack-t}, it suffices to show that
\[
\limsup_{t\to+\infty}
	t^{\frac{N-2s}{2s}}
	\overline{v}(t)
\leq C<+\infty,
\]
Suppose on the contrary that for some large $M>0$ to be chosen later,
\begin{equation}\label{eq:v-tk-new}
\overline{v}(t_k)
\geq M t_k^{-\frac{N-2s}{2s}},
	\qquad \text{ along some } t_k\to+\infty.
\end{equation}
For $t\in[\frac{t_k}{4},\frac{t_k}{2}]$, we absorb the exponential error using \eqref{eq:harnack-v-bar} as follows:
\[
\begin{split}
C e^{-\frac{N-2s}{2}t}
&\leq Ce^{-\frac{N-2s}{8}t_k}
\leq C %
	t_k^{-\frac{N}{2s}}
\leq CM^{-\frac{N}{N-2s}} %
	\overline{v}(t_k)^{\frac{N}{N-2s}}
=CM^{-\frac{N}{N-2s}} 
	\int_{t_k}^{t_k+1}
		\overline{v}(t_k)^{\frac{N}{N-2s}}
	\,d\bar{t}
	\\
&
\leq
	CM^{-\frac{N}{N-2s}}
	\int_{t_k}^{t_k+1}
		\overline{v}(\bar{t})^{\frac{N}{N-2s}}
	\,d\bar{t}
\leq CM^{-\frac{N}{N-2s}}\int_{t}^{\infty}
	\overline{v}(\bar{t})^{\frac{N}{N-2s}}
\,d\bar{t}.
\end{split}
\]
Thus
\[
\overline{v}(t)
\geq (c-CM^{-\frac{N}{N-2s}})
\int_t^{\infty}
	\overline{v}(\bar{t})^{\frac{N}{N-2s}}
\,d\bar{t},
	\qquad \text{ for } t\in[\tfrac{t_k}{4},\tfrac{t_k}{2}].
\]
For $M$ large enough, we see that 
\begin{equation*}
\cV(t)
:=\int_t^{\infty}
	\overline{v}(\bar{t})^{\frac{N}{N-2s}}
\,d\bar{t}> 0
\end{equation*}
satisfies an ordinary differential inequality
\[
[-\cV'(t)]^{\frac{N-2s}{N}}
\geq c \cV(t),
	\qquad \text{ for } t\in[\tfrac{t_k}{4},\tfrac{t_k}{2}],
\]
which yields
\[
\cV\left (
	\tfrac{t_k}{2}
\right)^{-\frac{2s}{N-2s}}
\geq
\cV\left (
	\tfrac{t_k}{2}
\right)^{-\frac{2s}{N-2s}}
-\cV\left (
	\tfrac{t_k}{4}
\right )^{-\frac{2s}{N-2s}}
\geq ct_k.
\]
This implies, using \eqref{eq:v-3-2} and \eqref{eq:harnack-v-bar},
\begin{align*}
v(t_k,\omega)
&\leq
	C\int_{t_k/2}^{\infty}
		\int_{\bS^{N-1}}
			(Cv(\bar{t}))^{\frac{N}{N-2s}}
		\,d\cH^{N-1}_{\bar{\omega}}
	\,d\bar{t}
	+Ce^{-\frac{N-2s}{2}t_k}
\leq
	C\cV\left (
		\tfrac{t_k}{2}
	\right )
	+Ce^{-\frac{N-2s}{2}t_k}
\leq
	Ct_k^{-\frac{N-2s}{2s}},
\end{align*}
for any $\omega\in\bS^{N-1}$. Averaging over $\omega\in\bS^{N-1}$ yields a contradiction to \eqref{eq:v-tk-new}, as $t_k\to+\infty$, provided $M$ is large enough. Using \eqref{eq:harnack-v-bar} again, the desired estimate follows.
\end{proof}

\subsection{Precise asymptotic equation}

The kernel in \eqref{eq:cK} is asymptotically constant as $t-\bar{t}\to-\infty$:

\begin{lem}
	Let $M\geq 1$ be fixed large. For any $\bar{t}\geq t+M\log t$, $t\geq 8$ and $\omega,\bar{\omega}\in\bS^{N-1}$,
	\[
	|\cK(t-\bar{t},\omega\cdot\bar{\omega})
	-C_{N,-s}|
	\leq Ct^{-M},
	\]
for $C>0$ depending only on $N$ and $s$.
\end{lem}

\begin{lem}\label{lem:eq-v-4-bar}
Write $\kappa=C_{N,-s}\cH^{N-1}(\bS^{N-1})=2^{1-2s}\frac{\Gamma(\frac{N-2s}{2})}{\Gamma(\frac{N}{2})\Gamma(s)}$. For any $t\geq 8$ and $\omega\in\bS^{N-1}$,
\begin{equation}\label{eq:v-4}
\begin{split}
v(t,\omega)
=\kappa
	\int_{t}^{\infty}
		\overline{v}(\bar{t})^{\frac{N}{N-2s}}
	\,d\bar{t}
	+O\left(
		t^{-\frac{N}{2s}}
		\log t
	\right).
\end{split}\end{equation}
Consequently, the spherical average $\overline{v}$ satisfies
\begin{equation}\label{eq:v-4-bar}
\begin{split}
\overline{v}(t)
=\kappa
	\int_{t}^{\infty}
		\overline{v}(\bar{t})^{\frac{N}{N-2s}}
	\,d\bar{t}
	+O\left(
		t^{-\frac{N}{2s}}
		\log t
	\right).
\end{split}\end{equation}
Here the constant in the error depends only on $N$, $s$, $\norm[L^{\frac{N}{N-2s}}((0,\infty)\times\bS^{N-1})]{v}$ and $\norm[\tilde{L}^1_{2s}(\R^N\setminus B_1)]{g}$.
\end{lem}

\begin{proof}
Let $M>0$ be a constant to be chosen. For $t\geq 8$ and $\omega\in\bS^{N-1}$, we split \eqref{eq:v-2} as
\begin{equation}\label{eq:v-2-bar}
\begin{split}
v(t,\omega)
&=C_{N,-s}\int_{t}^{\infty}\int_{\bS^{N-1}}
	v(\bar{t},\bar{\omega})^{\frac{N}{N-2s}}
\,d\cH^{N-1}_{\bar{\omega}}\,d\bar{t}\\
&\quad\;+\int_{t/2}^{t-M\log t}\int_{\bS^{N-1}}
	\cK(t-\bar{t},\omega\cdot\bar{\omega})
	v(\bar{t},\bar{\omega})^{\frac{N}{N-2s}}
\,d\cH^{N-1}_{\bar{\omega}}\,d\bar{t}\\
&\quad\;+\int_{t-M\log t}^{t}\int_{\bS^{N-1}}
	\cK(t-\bar{t},\omega\cdot\bar{\omega})
	v(\bar{t},\bar{\omega})^{\frac{N}{N-2s}}
\,d\cH^{N-1}_{\bar{\omega}}\,d\bar{t}\\
&\quad\;+\int_{t}^{t+M\log t}\int_{\bS^{N-1}}
	[\cK(t-\bar{t},\omega\cdot\bar{\omega})
	-C_{N,-s}]
	v(\bar{t},\bar{\omega})^{\frac{N}{N-2s}}
\,d\cH^{N-1}_{\bar{\omega}}\,d\bar{t}\\
&\quad\;+\int_{t+M\log t}^{\infty}\int_{\bS^{N-1}}
	[\cK(t-\bar{t},\omega\cdot\bar{\omega})
	-C_{N,-s}]
	v(\bar{t},\bar{\omega})^{\frac{N}{N-2s}}
\,d\cH^{N-1}_{\bar{\omega}}\,d\bar{t}
+O(e^{-\frac{N-2s}{2}t})\\
&=:I_1+I_2+I_3+I_4+I_5+O(e^{-\frac{N-2s}{2}t}).
\end{split}\end{equation}
The dominant term is $I_1$. Using \eqref{eq:cK}, \eqref{eq:cK-int}, \eqref{eq:harnack-v-bar} and \Cref{prop:upper-rough}, we estimate the errors one by one.
\begin{align*}
|I_2|
&\leq \int_{t/2}^{t-M\log t}\int_{\bS^{N-1}}
	Ce^{-(N-2s)M\log t}
	v(\bar{t},\bar{\omega})^{\frac{N}{N-2s}}
\,d\cH^{N-1}_{\bar{\omega}}\,d\bar{t}
\leq Ct^{-(N-2s)M},\\
|I_3|+|I_4|
&\leq \int_{t-M\log t}^{t+M\log t}\int_{\bS^{N-1}}
	C
	(Ct^{-\frac{N-2s}{2s}})^{\frac{N}{N-2s}}
\,d\cH^{N-1}_{\bar{\omega}}\,d\bar{t}
\leq CMt^{-\frac{N}{2s}}\log t,\\
|I_5|
&\leq \int_{t+M\log t}^{\infty}\int_{\bS^{N-1}}
	Ct^{-M}
	v(\bar{t},\bar{\omega})^{\frac{N}{N-2s}}
\,d\cH^{N-1}_{\bar{\omega}}\,d\bar{t}
\leq Ct^{-M}.
\end{align*}
Therefore, by fixing $M=\frac{N}{2s(N-2s)}+\frac{N}{2s}+1$, \eqref{eq:v-2-bar} becomes 
\begin{equation}\label{eq:v-3-bar}
v(t,\omega)
=C_{N,-s}\int_{t}^{\infty}\int_{\bS^{N-1}}
	v(\bar{t},\bar{\omega})^{\frac{N}{N-2s}}
\,d\cH^{N-1}_{\bar{\omega}}\,d\bar{t}
+O(t^{-\frac{N}{2s}}\log t),
	\quad \text{ for } t\geq 8,\,\omega\in\bS^{N-1}.
\end{equation}
We observe that by iterating \eqref{eq:v-3-bar} itself, the spherical mean and the power function on the right hand side of \eqref{eq:v-3-bar} almost commute. Indeed, let us denote
\[
\bV(t)
:=C_{N,-s}\int_{t}^{\infty}\int_{\bS^{N-1}}
	v(\bar{t},\bar{\omega})^{\frac{N}{N-2s}}
\,d\cH^{N-1}_{\bar{\omega}}\,d\bar{t}
\leq Ct^{-\frac{N-2s}{2s}},
\]
\[
\cR(t,\omega)
:=v(t,\omega)-\bV(t)
=O(t^{-\frac{N}{2s}}\log t),
	\qquad
\overline{\cR}(t)
:=\fint_{\bS^{N-1}}
	\cR(t,\omega)
\,d\cH^{N-1}_{\omega}
=O(t^{-\frac{N}{2s}}\log t),
\]
and compute, for $t\geq 8$, using \eqref{eq:v-3-bar} and its spherical average on both sides,
\begin{align*}
0\leq
&\fint_{\bS^{N-1}}
	v(t,\omega)^{\frac{N}{N-2s}}
\,d\cH^{N-1}_{\omega}
-\overline{v}(t)^{\frac{N}{N-2s}}\\
&=
\fint_{\bS^{N-1}}
	\left (
		\bV(t)
		+\cR(t,\omega)
	\right )^{\frac{N}{N-2s}}
\,d\cH^{N-1}_{\omega}
-\left (
	\bV(t)
	+\overline{\cR}(t)
\right )^{\frac{N}{N-2s}}\\
&=\fint_{\bS^{N-1}}
	\frac{N}{N-2s}
	\int_{0}^{1}
		\left (
			\bV(t)
			+\overline{\cR}(t)
			+\tau(\cR(t,\omega)-\overline{\cR}(t))
		\right )^{\frac{2s}{N-2s}}
	\,d\tau
	\cdot
	(\cR(t,\omega)-\overline{\cR}(t))
\,d\cH^{N-1}_{\omega}\\
&\leq
	Ct^{-\frac{N}{2s}-1}\log t.
\end{align*}
Here we have used Jensen's inequality again. %
Therefore, \eqref{eq:v-4-bar} follows.
\end{proof}

	\subsection{Exact singular behavior}
	
	\begin{prop}\label{prop:lower-rough}
		Suppose
		\begin{equation}\label{eq:limsup-ass}
			\ell
			=\limsup_{t\to+\infty}
			t^{\frac{N-2s}{2s}}
			\overline{v}(t)
			>0.
		\end{equation}
		Then for any $\omega\in\bS^{N-1}$,
		\begin{equation}\label{eq:v-exact}
		v(t,\omega)
		=c_0 t^{-\frac{N-2s}{2s}}
		\big(
			1+O(t^{-1}(\log t)^2)
		\big),
			\qquad \text{ as } t\to+\infty,
		\end{equation}
where, for $\kappa$ is given in \Cref{lem:eq-v-4-bar},
\begin{equation}\label{eq:c0-explicit}
c_0=\Bigl(\frac{N-2s}{2s\kappa}\Bigr)^{\frac{N-2s}{2s}}.
\end{equation}
The constant in the error depends only on $N$, $s$, $\norm[L^{\frac{N}{N-2s}}((0,\infty)\times\bS^{N-1})]{v}$ and $\norm[\tilde{L}^1_{2s}(\R^N\setminus B_1)]{g}$. %
\end{prop}

We remark that as $s\to 1^-$, $c_0\to \left(\frac{N-2}{\sqrt{2}}\right)^{N-2}$, recovering the constant in \cite{Veron1981}.
	
\begin{proof}
We carry out a continuation argument. Let $M>0$ be fixed large enough (with respect to the constant of the error term in \eqref{eq:v-4-bar}) and define
\[
\bT=\set{t\geq 8:
	\cV(t)
	\geq M
		t^{-\frac{N}{2s}}\log t
}.
\]
Note that \eqref{eq:limsup-ass} and \eqref{eq:v-4-bar} imply that $\bT\neq\varnothing$ for any choice of $M$. The main ingredient is to show that if $T\in \bT$ with $T\gg 1$ (depending on $M$), then $[T,T+1]\subset \bT$. 

{\medskip\noindent\bf Step 1: Local lower bound on $\cV(t)$.} Assuming $T\in\bT$, we claim that
\begin{equation}\label{eq:cV-lower}
\cV(t)\geq C^{-1}Mt^{-\frac{N}{2s}}\log t,
	\qquad \forall t\in[T,T+1].
\end{equation}
By \eqref{eq:harnack-v-bar} and \eqref{eq:v-4-bar}, for $|t-T|\leq 1$,
\[
\kappa \cV(T)-C_1T^{-\frac{N}{2s}}\log T
\leq \overline{v}(T)
\leq C\overline{v}(t)
\leq C\cV(t)+Ct^{-\frac{N}{2s}}\log t
\leq C\cV(t)+C_1T^{-\frac{N}{2s}}\log T.
\]
Hence, the hypothesis $T\in\bT$ and the choice $\frac{2C_1}{M}\leq \frac{\kappa}{2}$ imply
\[
\cV(t) 
\geq C^{-1}\cV(T) 
\geq C^{-1} M T^{-\frac{N}{2s}}\log T
\geq C^{-1} M t^{-\frac{N}{2s}}\log t,
	\quad \text{ for } t\in[T,T+1].
\]

{\medskip\noindent\bf Step 2: Solving the differential inequality locally.} 
Using \eqref{eq:cV-lower} on \eqref{eq:v-4-bar}, 
\begin{equation*}%
[-\cV'(t)]^{\frac{N-2s}{N}}
\leq \kappa \cV(t)
	+C t^{-\frac{N}{2s}}\log t
\leq C\cV(t),
	\qquad \text{ for } t\in[T,T+1].
\end{equation*}
This implies
\begin{equation}\label{eq:cV-ineq-sol}
\big[\cV(t)^{-\frac{2s}{N-2s}}\big]'
\leq C_2,
	\qquad \text{ for } t\in[T,T+1].
\end{equation}
Recall that $t\in \bT$ if and only if
\begin{equation}\label{eq:bT}
\cV(t)^{-\frac{2s}{N-2s}} \leq M^{-\frac{2s}{N-2s}} t^{\frac{N}{N-2s}} (\log t)^{-\frac{2s}{N-2s}}.
\end{equation}
Now \eqref{eq:cV-ineq-sol} and \eqref{eq:bT} implies that for $t\in[T,T+1]$,
\begin{align*}
\cV(t)^{-\frac{2s}{N-2s}}
&\leq \cV(T)^{-\frac{2s}{N-2s}}
	+C_2(t-T)
\leq 
	M^{-\frac{2s}{N-2s}} 
	T^{\frac{N}{N-2s}} 
	(\log T)^{-\frac{2s}{N-2s}}
	+C_2(t-T).
\end{align*}
Therefore, by the convexity of the function $t^{\frac{N}{N-2s}}(\log t)^{-\frac{2s}{N-2s}}$, if $T\geq T_0(C_2,M)$, then \eqref{eq:bT} holds for $t\in[T,T+1]$, that is, $[T,T+1]\subset \bT$.

{\medskip\noindent\bf Step 3: Continuation.} 
By \eqref{eq:limsup-ass}, there exists $T_1(N,s,M,\ell) \in \bT$ with $T_1 \geq T_0$. Using {\bf Step 2} inductively,
\[
[T_1,+\infty)\subset \bT.
\]

{\medskip\noindent\bf Step 4: Exact behavior.} 
Finally, using \eqref{eq:cV-ineq-sol} on $[T_1,+\infty)$ yields that for $t\geq T_2:=\cV(T_1)^{-\frac{2s}{N-2s}}$,
\begin{align*}
\cV(t)^{-\frac{2s}{N-2s}}
\leq \cV(T_1)^{-\frac{2s}{N-2s}}+C_2(t-T_1)
\leq Ct,
\end{align*}
reducing the ODE \eqref{eq:v-4-bar} to
\[
[-\cV'(t)]^{\frac{N-2s}{N}}
=\kappa \cV(t)
	\big(1+O(\cV(t)^{-1}t^{-\frac{N}{2s}}\log t)\big)
=\kappa \cV(t)
	\big(1+O(t^{-1}\log t)\big),
\qquad \text{ for } t\geq T_2.
\]
Rearranging,
\[
[\cV(t)^{-\frac{2s}{N-2s}}]'
=\frac{2s}{N-2s}\kappa^{\frac{N}{N-2s}}
+O(t^{-1}\log t),
	\qquad \text{ for } t\geq T_2.
\]
\[
\cV(t)^{-\frac{2s}{N-2s}}
=\cV(T_2)^{-\frac{2s}{N-2s}}
+\frac{2s}{N-2s}\kappa^{\frac{N}{N-2s}}t
+O\big(
	(\log t)^2-(\log T_2)^2
\big),
	\qquad \text{ for } t\geq T_2.
\]
For $t\geq T_3=T_3(T_2)$ big enough we can absorb the constants and arrive at
\[
\cV(t)
=\left (\frac{N-2s}{2s}\right )^{\frac{N-2s}{2s}}
\kappa^{-\frac{N}{2s}}
t^{-\frac{N-2s}{2s}}
\bigl(
	1+O(t^{-1}(\log t)^2)
\bigr),
	\qquad \text{ for } t\geq T_3.
\]
Now \eqref{eq:v-exact} follows from \eqref{eq:v-4}, as desired.
\end{proof}

\subsection{Removability of singularity}

\begin{lem}
If $u$ solves \eqref{eq:main} and
\begin{equation*}
	\lim_{r\to0}
	r^{N-2s}
	\left(
	\log\frac1r
	\right)^{\frac{N-2s}{2s}}
	u(r)
	=0,
\end{equation*}
then the singularity of $u$ at $0$ is removable.
\end{lem}

\begin{proof}
By assumption, we can view the nonlinearity as a linear term with a small singular potential $u^{\frac{2s}{N-2s}}=\frac{o(1)}{r^{2s}\log\frac1r}$. Now a standard comparison argument applies, using a slightly negative power (plus an infinitesimal but more singular term) as barrier. This shows that $u\in L^p$ for $p>\frac{N}{N-2s}$, hence bounded. See, e.g. \cite[Proposition 2.6]{JLX}.
\end{proof}

\subsection{Second order expansion}

We study the second order asymptotics in the case of singularity.

\begin{prop}
If $v(t,\omega)$ satisfies \eqref{eq:v-2} and
\begin{equation}\label{eq:exp-1st}
v(t,\omega)=c_0t^{-\frac{N-2s}{2s}}\big(1+O(t^{-1}(\log t)^2)\big) 
	\qquad \text{ as } t\to+\infty,
\end{equation}
then there exists $c_1<0$ given in \Cref{rmk:c1} such that
\[
v(t,\omega)
=c_0 t^{-\frac{N-2s}{2s}}
	+c_1 t^{-\frac{N}{2s}}\log t
	+O(t^{-\frac{N}{2s}}),
		\qquad \text{ as } t\to+\infty.
\]
\end{prop}

\begin{proof}
The second order expansion is given by the linearized equation around $c_0t^{-\frac{N-2s}{2s}}$. Moreover, in view of \eqref{eq:exp-1st}, the angular variable does not play a role. Thus, one may take the spherical average to reduce to the case of radial symmetry. Indeed here we assume it for notational simplicity.

Recall that $v(t)$ satisfies the equation
\[
v(t)
=\int_{\frac{t}{2}}^{\infty}
	K(t-\bar{t})v(\bar{t})^{\frac{N}{N-2s}}
\,d\bar{t}
+O(e^{-\frac{N-2s}{2}t}),
\]
as $t\to+\infty$, where
\[\begin{split}
K(t-\bar{t})
&=C_1\int_{0}^{\pi}
	\frac{
		\sin^{N-2}\theta
	}{
		(1+e^{2(t-\bar{t})}-2e^{t-\bar{t}}\cos\theta)^{\frac{N-2s}{2}}
	}
\,d\theta,
\end{split}\]
for an explicit constant $C_1=C_1(N,s)=C_{N,-s}\cH^{N-2}(\bS^{N-2})=2^{1-2s}\pi^{-\frac12}\frac{\Gamma(\frac{N-2s}{2})}{\Gamma(s)\Gamma(\frac{N-1}{2})}$. We keep in mind that for a fixed $s\in(0,1)$,
\begin{equation}\label{eq:K-behavior}
\begin{split}
K(t-\bar{t})
&=C_1\int_{0}^{\pi}
	\frac{
		\sin^{N-2}\theta
	}{
		(1+e^{2(t-\bar{t})}-2e^{t-\bar{t}}\cos\theta)^{\frac{N-2s}{2}}
	}
\,d\theta
\asymp
\begin{cases}
1,
	& s\in(\tfrac12,1),\\
\log|t-\bar{t}|,
	& s=\tfrac12,\\
|t-\bar{t}|^{-(1-2s)},
	& s\in(0,\tfrac12),
\end{cases}
\end{split}
\end{equation}
as $t-\bar{t}\to 0$. In particular, $K$ is locally integrable. Let $M=M(N,s)$ be fixed large and $\ell=M\log t$, so that $e^{-\ell}=t^{-M}$. Keeping \eqref{eq:exp-1st} in mind, we split the integral as follows:
\begin{align*}
\int_{\frac{t}{2}}^{t-\frac{\ell}{N-2s}}
	K(t-\bar{t})v(\bar{t})^{\frac{N}{N-2s}}
\,d\bar{t}
=\left(
\int_{\frac{t}{2}}^{t-\frac{\ell}{N-2s}}
+\int_{t-\frac{\ell}{N-2s}}^{t+\ell}
+\int_{t+\ell}^{\infty}
\right)
	K(t-\bar{t})v(\bar{t})^{\frac{N}{N-2s}}
\,d\bar{t}
=:I_1+I_2+I_3.
\end{align*}

For $I_1$, using $K(t-\bar{t})=O(e^{-(N-2s)(t-\bar{t})})$ as $t-\bar{t}\to+\infty$,
\begin{align*}
I_1
&=\int_{\frac{t}{2}}^{t-\frac{\ell}{N-2s}}
	 O(t^{-M})
	 O(t^{-\frac{N}{2s}})
\,d\bar{t}
=O(t^{-M-\frac{N-2s}{2s}})
=O(t^{-\frac{N}{2s}-1}),
\end{align*}
since $M\geq 2$.

For $I_3$, using $K(t-\bar{t})= C_2+O(e^{t-\bar{t}})$ as $t-\bar{t}\to-\infty$ where $C_2=C_2(N,s)=C_1\int_{0}^{\pi}\sin^{N-2}\theta\,d\theta=2^{1-2s}|\Beta(\frac{N}{2},-s)|^{-1}$, and the fact that
\begin{align*}
\int_{t}^{t+\ell}
	\bar{t}^{-\frac{N}{2s}}
\,d\bar{t}
&=t^{-\frac{N}{2s}}
\int_{0}^{\ell}
	\Big(
		1+\frac{\tau}{t}
	\Big)^{-\frac{N}{2s}}
\,d\tau
=t^{-\frac{N}{2s}}\ell
+O\big(t^{-\frac{N}{2s}-1}(\log t)^2\big),
\end{align*}
we have
\begin{align*}
I_3
&=\int_{t+\ell}^{\infty}
	 \Big(
	 	C_2+O(t^{-M})
	 \Big)
	 v(\bar{t})^{\frac{N}{N-2s}}
\,d\bar{t}\\
&=C_2\int_{t}^{\infty}
	 v(\bar{t})^{\frac{N}{N-2s}}
\,d\bar{t}
-C_2\int_{t}^{t+\ell}
	 \Big(
		 c_0^{\frac{N}{N-2s}}
		 \bar{t}^{-\frac{N}{2s}}
		 +O\big(t^{-\frac{N}{2s}-1}(\log t)^2\big)
	 \Big)
\,d\bar{t}
+O(t^{-M-\frac{N-2s}{2s}})\\
&=C_2\int_{t}^{\infty}
	 v(\bar{t})^{\frac{N}{N-2s}}
\,d\bar{t}
-C_2c_0^{\frac{N}{N-2s}}t^{-\frac{N}{2s}}\ell
+O\big(t^{-\frac{N}{2s}-1}(\log t)^3\big),
\end{align*}
again using $M\geq 2$.

For $I_2$, whose kernel $K(t-\bar{t})$ can possibly be singular, we expand $v(\bar{t})$:
\begin{align*}
I_2
&=\int_{t-\frac{\ell}{N-2s}}^{t+\ell}
	K(t-\bar{t})
	v(\bar{t})^{\frac{N}{N-2s}}
\,d\bar{t}
=\int_{t-\frac{\ell}{N-2s}}^{t+\ell}
	K(t-\bar{t})
	 \Big(
		 c_0^{\frac{N}{N-2s}}
		 \bar{t}^{-\frac{N}{2s}}
		 +O\big(t^{-\frac{N}{2s}-1}(\log t)^2\big)
	 \Big)
\,d\bar{t}\\
&=c_0^{\frac{N}{N-2s}}
\int_{t-\frac{\ell}{N-2s}}^{t+\ell}
	K(t-\bar{t})
	\bar{t}^{-\frac{N}{2s}}
\,d\bar{t}
+O\big(t^{-\frac{N}{2s}-1}(\log t)^3\big).
\end{align*}
Here the first integral can be expanded as follows:
\begin{align*}
\int_{t-\frac{\ell}{N-2s}}^{t+\ell}
	K(t-\bar{t})
	\bar{t}^{-\frac{N}{2s}}
\,d\bar{t}
&=\int_{-\frac{\ell}{N-2s}}^{\ell}
	K(-\tau)
	(t+\tau)^{-\frac{N}{2s}}
\,d\tau
=t^{-\frac{N}{2s}}
\Big(
	1+O(t^{-1}\log t)
\Big)
\int_{-\frac{\ell}{N-2s}}^{\ell}
	K(-\tau)
\,d\tau.
\end{align*}
One can show that the integral kernel blows up exactly like $\ell$, upon changing variable to $\sigma=e^{-\tau}$ and integrating by parts using \eqref{eq:K-behavior}:
\begin{align*}
&
\int_{-\frac{\ell}{N-2s}}^{\ell}
	K(-\tau)
\,d\tau
=\int_{e^{-\ell}}^{e^{\frac{\ell}{N-2s}}}
	K(\log \sigma)
\,\frac{d\sigma}{\sigma}\\
&\;=C_1
\left(
	\int_{e^{-\ell}}^{1}
	+\int_{1}^{\infty}
	-\int_{e^{\frac{\ell}{N-2s}}}^{\infty}
\right)
\int_{0}^{\pi}
	\frac{
		\sin^{N-2}\theta
	}{
		(1+\sigma^2-2\sigma\cos\theta)^{\frac{N-2s}{2}}
	}
\,d\theta
\,\frac{d\sigma}{\sigma}\\
&\;=C_1\int_{e^{-\ell}}^{1}
\int_{0}^{\pi}
	\frac{
		\sin^{N-2}\theta
	}{
		(1+\sigma^2-2\sigma\cos\theta)^{\frac{N-2s}{2}}
	}
\,d\theta
\,d(\log\sigma)
+C_1\int_{1}^{\infty}
\int_{0}^{\pi}
	\frac{
		\sin^{N-2}\theta
	}{
		(1+\sigma^2-2\sigma\cos\theta)^{\frac{N-2s}{2}}
	}
\,d\theta
\,\frac{d\sigma}{\sigma}\\
&\qquad
-O\left(
\int_{e^{\frac{\ell}{N-2s}}}^{\infty}
\int_{0}^{\pi}
	\frac{
		\sin^{N-2}\theta
	}{
		\sigma^{N-2s+1}
	}
\,d\theta
\,d\sigma
\right)\\
&\;=-C_1\int_{0}^{\pi}
	\frac{
		\sin^{N-2}\theta
	}{
		(1+e^{-2\ell}-2e^{-\ell}\cos\theta)^{\frac{N-2s}{2}}
	}
\,d\theta
\,(-\ell)
+C_1\int_{1}^{\infty}
\int_{0}^{\pi}
	\frac{
		\sin^{N-2}\theta
	}{
		(1+\sigma^2-2\sigma\cos\theta)^{\frac{N-2s}{2}}
	}
\,d\theta
\,\frac{d\sigma}{\sigma}
+O(e^{-\ell})\\
&\;=C_2\ell+C_3
+O(t^{-M}),
\end{align*}
where we take
\begin{equation}\label{eq:C3}
C_3=C_1\int_{1}^{\infty}
\int_{0}^{\pi}
	\frac{
		\sin^{N-2}\theta
	}{
		(1+\sigma^2-2\sigma\cos\theta)^{\frac{N-2s}{2}}
	}
\,d\theta
\,\frac{d\sigma}{\sigma}.
\end{equation}
We conclude that
\begin{align*}
v(t)
&=C_2\int_{t}^{\infty}
	 v(\bar{t})^{\frac{N}{N-2s}}
\,d\bar{t}
-C_2c_0^{\frac{N}{N-2s}}t^{-\frac{N}{2s}}\ell
+O\big(t^{-\frac{N}{2s}-1}(\log t)^3\big)
+c_0^{\frac{N}{N-2s}}t^{-\frac{N}{2s}}
	(C_2\ell+C_3+O(t^{-M})).
\end{align*}
Since the terms of order $t^{-\frac{N}{2s}}\ell$ cancel, by denoting $C_4=C_3c_0^{\frac{N}{N-2s}}$, the equation simplifies to
\begin{align*}
v(t)
&=C_2\int_{t}^{\infty}
	 v(\bar{t})^{\frac{N}{N-2s}}
\,d\bar{t}
+C_4t^{-\frac{N}{2s}}
+O\big(t^{-\frac{N}{2s}-1}(\log t)^3\big).
\end{align*}
Let $\cV(t)$, $\phi(t)$ and $\Phi(t)$ be such that
\begin{align*}
v(t)&=c_0t^{-\frac{N-2s}{2s}}+\phi(t),\\
\cV(t)&=\int_{t}^{\infty}v(\bar{t})^{\frac{N}{N-2s}}\,d\bar{t},\\
\Phi(t)&=\int_{t}^{\infty}\bar{t}^{-1}\phi(\bar{t})\,d\bar{t},
\end{align*}
so that
\begin{align*}
\cV(t)
&=\frac{2s}{N-2s}c_0^{\frac{N}{N-2s}}t^{-\frac{N-2s}{2s}}
	+\frac{N}{N-2s}c_0^{\frac{2s}{N-2s}}\Phi(t)
	+O\big(t^{-\frac{N}{2s}-2}(\log t)^4\big).
\end{align*}
The linearized equation on $\Phi(t)$ is given by
\begin{align*}
c_0t^{-\frac{N-2s}{2s}}-t\Phi'(t)
&=C_2\left(
	\frac{2s}{N-2s}c_0^{\frac{N}{N-2s}}t^{-\frac{N-2s}{2s}}
	+\frac{N}{N-2s}c_0^{\frac{2s}{N-2s}}\Phi(t)
\right)
+C_4t^{-\frac{N}{2s}}
+O\big(t^{-\frac{N}{2s}-1}(\log t)^4\big).
\end{align*} 
By the choice of $c_0$ that $C_2c_0^{\frac{2s}{N-2s}}=\frac{N-2s}{2s}$, the equation simplifies to
\begin{align*}
t\Phi'(t)+\frac{N}{2s}\Phi(t)
&=-C_4t^{-\frac{N}{2s}}
+O\big(t^{-\frac{N}{2s}-1}(\log t)^4\big).
\end{align*}
Thus
\begin{equation}
\label{eq:phi-ODE}
(t^{\frac{N}{2s}}\Phi(t))'
=-C_4t^{-1}
+O\big(t^{-2}(\log t)^4\big),
\end{equation}
which gives precise second order behavior:
\begin{align*}
\Phi(t)&=-C_4 t^{-\frac{N}{2s}}\log t+O(t^{-\frac{N}{2s}}),\\
\phi(t)&=-t\Phi'(t)=\frac{N}{2s}\Phi(t)+C_4t^{-\frac{N}{2s}}
	=-C_5t^{-\frac{N}{2s}}\log t+O(t^{-\frac{N}{2s}}),
\end{align*}
where $C_5=\frac{N}{2s}C_4$, as desired.
\end{proof}

\begin{remark}\label{rmk:c1}
We have $c_1=-C_5$ with
\begin{align*}
C_5&=\frac{N}{2s}C_4
=\frac{N}{2s}C_3c_0^{\frac{N}{N-2s}}
=\frac{N}{2s}C_1
\left(
\int_{1}^{\infty}
\int_{0}^{\pi}
	\frac{
		\sin^{N-2}\theta
	}{
		(1+\sigma^2-2\sigma\cos\theta)^{\frac{N-2s}{2}}
	}
\,d\theta
\,\frac{d\sigma}{\sigma}
\right)
\left(
	\frac{N-2s}{2s}C_2^{-1}
\right)^{\frac{N}{2s}}\\
&=\frac{N}{2s}
	\left(
		\frac{N-2s}{2s}
	\right)^{\frac{N}{2s}}
	\pi^{-\frac12}
	\frac{
		\Gamma(\frac{N}{2})
	}{
		\Gamma(\frac{N-1}{2})
	}
	\left(
		2^{1-2s}
		\frac{
			\Gamma(\frac{N-2s}{2})
		}{
			\Gamma(s)
			\Gamma(\frac{N}{2})	
		}
	\right)^{-\frac{N-2s}{2s}}
\int_{1}^{\infty}
\int_{0}^{\pi}
	\frac{
		\sin^{N-2}\theta
	}{
		(1+\sigma^2-2\sigma\cos\theta)^{\frac{N-2s}{2}}
	}
\,d\theta
\,\frac{d\sigma}{\sigma}.
\end{align*}
In the special case of $s=1$, we are able to use the quadratic transformation formula \cite[15.8.18]{DLMF} and Legendre duplication formula to compute
\begin{align*}
C_3(N,1)
&=2^{N-2}\Beta\left(
	\frac{N-1}{2},
	\frac{N-1}{2}
\right)C_1(N,1)
\int_{1}^{\infty}
	\Hyperg\left(
		\frac{N-2}{2},
		\frac{N-1}{2};
		N-1;
		-\frac{4\sigma}{(\sigma-1)^2}
	\right)
\,\frac{d\sigma}{\sigma(\sigma-1)^{N-2}}\\
&=2^{N-2}
	\frac{
		\Gamma(\frac{N-1}{2})^2
	}{
		\Gamma(N-1)
	}
	2^{-1}\pi^{-\frac12}
	\frac{
		\Gamma(\frac{N-2}{2})
	}{
		\Gamma(1)
		\Gamma(\frac{N-1}{2})
	}
\int_{1}^{\infty}
\Hyperg\left(
	N-2,N-1;N-1;\frac{1}{1-\sigma}
\right)
\,\frac{
	d\sigma
}{
	\sigma(\sigma-1)^{N-2}
}\\
&=2^{N-3}\pi^{-\frac12}
\frac{
	\Gamma(\frac{N-2}{2})
	\Gamma(\frac{N-1}{2})
}{
	\Gamma(N-1)
}
\int_{1}^{\infty}
\left(
	1-\frac{1}{1-\sigma}
\right)^{-(N-2)}
\,
	\frac{
		d\sigma
	}{
		\sigma(\sigma-1)^{N-2}
	}
\\
&=2^{N-3}\pi^{-\frac12}
\frac{
	2^{1-(N-2)}\pi^{\frac12}
	\Gamma(N-2)
}{
	\Gamma(N-1)
}
\int_{1}^{\infty}
\,\frac{d\sigma}{\sigma^{N-1}}
=\frac{1}{(N-2)^2},
\end{align*}
giving
\begin{align*}
C_5(N,1)
&=
\frac{N}{2}
C_3(N,1)
c_0(N,1)^{\frac{2}{N-2}}
c_0(N,1)
=\frac{N}{2}\frac{1}{(N-2)^2}
	\left(
		\frac{(N-2)^2}{2}
	\right)
	c_0(N,1)
=\frac{N}{4}c_0(N,1).
\end{align*}
This recovers the constant in \cite[Lemme 3.3]{Veron1981}. (Note the opposite sign: we study diffusion near the origin while V\'{e}ron \cite{Veron1981} studies absorption near infinity.)
\end{remark}

\section{Liouville theorems}
\label{sec:Liouville}

We will prove \Cref{thm:Liouville-int} and then \Cref{thm:Liouville-diff}. It is convenient to introduce the Kelvin transform of $u$, denoted by $v$ in this section, namely
\begin{equation*}
	v(x)=\frac{1}{|x|^{N-2s}}u\left(\frac{x}{|x|^2}\right),
	\qquad
	\Ds v(x)=\frac{1}{|x|^{N+2s}}u\left(
	\frac{x}{|x|^2}\right).
\end{equation*}
In particular, a (non-negative) solution $u$ of \eqref{eq:ext} solves
\begin{equation}\label{eq:ext-Kelvin}\begin{dcases}
		\Ds v=|x|^{-2s}v^{\frac{N}{N-2s}}
		& \text{ in } B_1\setminus \set{0},\\
		v\geq 0
		& \text{ in } \R^N\setminus \set{0}.
\end{dcases}\end{equation}

\subsection{Integral formulation}
Under the assumptions of \Cref{thm:Liouville-int}, it is immediate to see that the Kelvin transform $v$ of $u$ satisfies the integral inequality 
\begin{equation}
	\label{eq:ext-Kelvin-int2}
	v(x)
	\geq
	\int_{B_1}
	\dfrac{
		|y|^{-2s}
		v(y)^{\frac{N}{N-2s}}
	}{
		|x-y|^{N-2s}
	}
	\,dy,
	\qquad \forall x\in B_1\setminus\set{0},
\end{equation}
and the Harnack inequality
\begin{equation}\label{eq:Harnack-Kelvin-statement}
	\sup_{B_{2r}\setminus B_{r/2}} v
	\leq C \inf_{B_{2r}\setminus B_{r/2}} v,
	\qquad \forall r\in(0,1/4).
\end{equation}

\begin{proof}[Proof of \Cref{thm:Liouville-int}]
Keeping \eqref{eq:ext-Kelvin-int2} in mind, we argue as follows.
	
	\medskip
	{\noindent \bf Step 1: Lower bound.}
	For any $x\in B_{1/4}\setminus\set{0}$,
	\[
	v(x)
	\geq
	\int_{B_1\setminus B_{1/2}}
	\dfrac{
		|y|^{-2s}v(y)^{\frac{N}{N-2s}}
	}{
		|x-y|^{N-2s}
	}
	\,dy
	\geq
	c\norm[L^{\frac{N}{N-2s}}(B_1\setminus B_{1/2})]{v}^{\frac{N}{N-2s}}.
	\]
	
	\medskip
	{\noindent \bf Step 2: Upper bound.} For any $x\in B_{1/4}\setminus\set{0}$, using %
	\eqref{eq:Harnack-Kelvin-statement}, %
	\[\begin{split}
		v(x)
		&\geq
		\int_{B_{|x|/2}(x)}
		\dfrac{
			|y|^{-2s}
			v(y)^{\frac{N}{N-2s}}
		}{
			|x-y|^{N-2s}
		}
		\,dy
		\geq
		c\int_{B_{|x|/2}(x)}
		\dfrac{
			|x|^{-2s}
			v(x)^{\frac{N}{N-2s}}
		}{
			|x|^{N-2s}
		}
		\,dy
		\geq
		cv(x)^{\frac{N}{N-2s}}.
	\end{split}\]
	Rearranging yields $v(x)\leq C$. Note that here the constants $C,c$ %
	depend (explicitly) only on $N$ and $s$.
	
	\medskip
	{\noindent \bf Step 3: Conclusion.} For any $x\in B_{1/16}\setminus\set{0}$, using %
	{\bf Step 1}, %
	\[\begin{split}
		v(x)
		&\geq
		\int_{B_{1/4} \setminus B_{2|x|}}
		\dfrac{
			|y|^{-2s}
			v(y)^{\frac{N}{N-2s}}
		}{
			|x-y|^{N-2s}
		}
		\,dy
		\geq
		c\norm[L^{\frac{N}{N-2s}}(B_1\setminus B_{1/2})]{
			v
		}^{(\frac{N}{N-2s})^2}
		\int_{B_{1/4}\setminus B_{2|x|}}
		\dfrac{
			|y|^{-2s}
		}{
			|y|^{N-2s}
		}
		\,dy\\
		&\geq
		c\norm[L^{\frac{N}{N-2s}}(B_1\setminus B_{1/2})]{
			v
		}^{(\frac{N}{N-2s})^2}
		\log\frac{1}{8|x|}.
	\end{split}\]
	Using {\bf Step 2},
	\[
	\norm[L^{\frac{N}{N-2s}}(B_1\setminus B_{1/2})]{
		v
	}
	\leq
	\left(
	\dfrac{C}{
		\log\frac{1}{8|x|}
	}
	\right)^{(\frac{N-2s}{N})^2}
	\to 0
	\qquad \textas |x|\to 0.
	\]
	We conclude that $v\equiv 0$ in $B_1\setminus B_{1/2}$, i.e. $u\equiv 0$ on $B_2\setminus B_1$. A similar argument on $v$ in $B_1\setminus B_r$ for each $r\in(0,1/2)$ shows that $u\equiv 0$ on $\R^N\setminus B_1$. %
\end{proof}

\subsection{Integro-differential formulation}

\begin{proof}[Proof of \Cref{thm:Liouville-diff}]
Solutions $u$ of \eqref{eq:ext} satisfy the Harnack inequality \eqref{eq:Harnack-statement} by a generalization of \cite[Proposition 3.1]{YZ1}. Now the Green-Poisson formulation for the Kelvin transform $v$ as in \eqref{eq:ext-Kelvin} satisfies
\[\begin{split}
v(x)
&=\int_{B_1}
	\bG_{B_1}(x,y)
	|y|^{-2s}v(y)^{\frac{N}{N-2s}}
\,dy
+\int_{\R^N\setminus B_1}
	\bP_{B_1}(x,y)
	v(y)
\,dy\\
&\geq
	c(N,s,R)
\int_{B_1}
\dfrac{
	|y|^{-2s}
	v(y)^{\frac{N}{N-2s}}
}{
	|x-y|^{N-2s}
}
\,dy,
	\qquad \text{ for } |x|<R<1.
\end{split}\]
This is \eqref{eq:ext-Kelvin-int2} in the interior up to a constant. Now the proof of \Cref{thm:Liouville-int} can be repeated to yield $v\equiv 0$ in $B_{R/2}$. But then the global maximum principle forces $v\equiv 0$ on $\R^N$.
\end{proof}

\section{Computations in radial coordinates}\label{sec:rad}
In this section we study the fractional Laplacian and Riesz potential in polar coordinates and obtain corresponding asymptotic expressions for (log-)polyhomogeneous functions.

\subsection{The fractional Laplacian and the Riesz potential}
Following the ideas of \cite{DPGW}, a conjugation of the fractional Laplacian can be rewritten as an integro-differential operator with a well-behaved convolution kernel. In fact, this also applies to the Riesz potential. For $N\geq 3$ and $s\in(0,1)$, define
\[
K_{N,s}(\rho)
=
C_{N,s}
\cH^{N-2}(\bS^{N-2})
\rho^{2s-1}
\int_{0}^{\pi}
	\dfrac{
		\sin^{N-2}\theta
	}{
		{(1+\rho^2-2\rho\cos\theta)^{\frac{N+2s}{2}}}
	}
\,d\theta,
\]
\[
K_{N,-s}(\rho)
=
C_{N,-s}
\cH^{N-2}(\bS^{N-2})
\rho^{-1}
\int_{0}^{\pi}
	\dfrac{
		\sin^{N-2}\theta
	}{
		{(1+\rho^2-2\rho\cos\theta)^{\frac{N-2s}{2}}}
	}
\,d\theta.
\]
Here $\cH^{d}(\Omega)$ denotes the $d$-dimensional Hausdorff measure of $\Omega$. The constants (for both $s\in(0,1)$ and $s\in(-1,0)$) are given by
\begin{equation}\label{eq:C-Ns}
C_{N,s}
=\dfrac{
    2^{2s}
    \Gamma(\frac{N+2s}{2})
}{
    |\Gamma(-s)|
    \pi^{\frac{N}{2}}
}
=2^{2s}\pi^{-\frac{N}{2}}
    \dfrac{
        \Gamma(\frac{N+2s}{2})
    }{
        \Gamma(2-s)
    }
    s(1-s),
\qquad
\cH^{N-2}(\bS^{N-2})
=\dfrac{
    2\pi^{\frac{N-1}{2}}
}{
    \Gamma(\frac{N-1}{2})
}.
\end{equation}

\begin{prop}\label{prop:K}
The following hold.%
\begin{enumerate}

\item (Conformal fractional Laplacian in polar coordinates) If $v\in C^{2}(\R^N\setminus\set{0})$ is radially symmetric, then
\[
r^N\Ds\left(\dfrac{v(r)}{r^{N-2s}}\right)
=
\PV\int_{0}^{\infty}
K_{N,s}(\rho)(v(r)-v(r\rho))
\,d\rho,
	\qquad \text{ for } r>0.
\]
Here the principal value integral is defined as
\[
\lim_{\eps\searrow0}
\int_{1+\eps}^{\infty}
K_{N,s}(\rho)
\left(
(v(r)-v(r\rho))
-\rho^{N-2s}
\left(
v\left(\dfrac{r}{\rho}\right)-v(r)
\right)
\right)
\,d\rho.
\]
In particular,
\begin{equation}\label{eq:PV-upper}\begin{split}
		\abs{
			r^N\Ds\left(
			\dfrac{v(r)}{r^{N-2s}}
			\right)
		}
		&\lesssim
		|v(r)|
		+r|v'(r)|
		+r^2\sup_{[\frac{r}{2},2r]}|v''|
		+\int_{0}^{\frac12}
		v(r\rho)
		\,d\rho
		+\int_{2}^{\infty}
		\dfrac{
			v(r\rho)
		}{
			\rho^{N+1}
		}
		\,d\rho.
\end{split}\end{equation}
The term $r^2|v''(r)|$ can be omitted when $s\in(0,\frac12)$.
\item (Riesz potential in polar coordinates)
If $f\in L^\infty_{\loc}(\R^N\setminus\set{0}) \cap \bigl( |x|^{-N} L^p(\R^N) \bigr)$ for some $p\in(1,\frac{N}{2s})$, then
\[
r^{N-2s}\Ints\left(
\dfrac{f(r)}{r^{N}}
\right)
=
\int_{0}^{\infty}
K_{N,-s}(\rho)
f(r\rho)
\,d\rho,
	\qquad \text{ for } r>0.
\]
Moreover, this gives the unique solution to
\[
\Ds\left(\dfrac{v(r)}{r^{N-2s}}\right)
=\dfrac{f(r)}{r^N}
\qquad \textae \text{ in } \R^N.
\]

\item (Exact representation) Let $s\in(0,1)\cup(-1,0)$. Then
\begin{equation}\label{eq:K-2F1}
	K_{N,s}(\rho)
	=
	\frac{
		2^{2s+1}\Gamma(\frac{N+2s}{2})
	}{
		|\Gamma(-s)|
		\Gamma(\frac{N}{2})
	}
	\frac{
		\rho^{(2s)_+-1}
	}{
		|\rho-1|^{N+2s}
	}
	\Hyperg\left (
		\frac{N+2s}{2},
		\frac{N-1}{2};
		N-1;
		-\frac{4\rho}{|\rho-1|^2}
	\right ).
\end{equation}
Here $\Hyperg$ denotes the Gaussian hypergeometric function (see \Cref{Ap:hyper}).

\item (Asymptotic behaviors) Let $s\in(0,1)$.
Then we can assert that, up to a constant,
\begin{equation}\label{eq:K-asymp-1}
K_{N,s}(\rho)
\sim
\begin{dcases}
\rho^{2s-1}
&%
\textas \rho\to 0^+,\\
|\rho-1|^{-(1+2s)}
& %
\textas \rho \to 1,\\
\rho^{-(N+1)}
& %
\textas \rho \to +\infty,
\end{dcases}
\end{equation}
\begin{equation}\label{eq:K-asymp-2}
K_{N,-s}(\rho)
\sim
\begin{dcases}
\rho^{-1}\quad\quad\quad
&%
\textas \rho\to 0^+,\\
|\rho-1|^{2s-1}\quad\quad\quad
&
\textas \rho\to 1,\, \text{ for } s\neq \tfrac12,\\
\log\tfrac{1}{|\rho-1|}\quad\quad\quad
&
\textas \rho\to 1,\, \text{ for } s=\tfrac12,\\
\rho^{-(N+1-2s)}
&%
\textas \rho\to+\infty.
\end{dcases}
\end{equation}

\end{enumerate}
\end{prop}

We remark that while explicit asymptotic constants are not needed for our purposes, they can be obtained from the representation \eqref{eq:K-2F1}.

\begin{proof}
The computations of non-local operators on radially symmetric functions are well-known to the experts. For the readers' convenience, we provide details in \Cref{sec:computations}.
\end{proof}

\begin{remark}
Note that for $s\in(0,1)$, the kernel $K_{N,s}$ has been computed in \cite[Section 4.1]{fat} (see also \cite[Lemma 2.5]{DPGW}) in the Emden--Fowler variable $t=-\log r$. Indeed, we can write
\[
e^{-Nt}\Ds\left(
    e^{(N-2s)t}v(e^{-t})
\right)
=\PV\int_{\R}
    \tilde{K}_{N,s}(\tau)
    (v(t)-v(t+\tau))
\,d\tau
\]
where
\[
	\tilde{K}_{N,s}(\tau)
	=c{e^{-\frac{N-2s}{2}\tau}e^{-\frac{N+2s}{2} |\tau|}}\Hyperg\left(
	\frac{N+2s}{2},1+s,\frac{N}{2};e^{-2|\tau|}
	\right)
	\asymp
	\begin{dcases}
		|\tau|^{-1-2s}
		& \textas \tau\to0,\\
		e^{2s\tau}	
		& \textas \tau\to{-\infty},\\
		e^{-N\tau}	
		& \textas \tau\to{+\infty}.
	\end{dcases}
	\]

This exact expression only holds for radial functions (zeroth spherical mode). Similar asymptotic behaviors are known for higher modes. %
\end{remark}

\subsection{Asymptotic computations}

Due to the criticality of the problem we introduce $\log$-polyhomogeneous functions singular at the origin, and we compute explicit asymptotic behaviors under the action of the fractional Laplacian (or Riesz potential) in order to obtain fine barriers.

Let $\chi\in C_c^\infty(\R^N,[0,1])$ be a radial cut-off function with $\chi=1$ in $B_{1/4}$ and  $\chi=0$ in $\R^N\setminus B_{1/2}$. For $\mu,\nu,\vartheta\geq0$, define the ($\log$-)polyhomogeneous functions
\begin{equation}\label{eq:logpoly}
\phi_{\mu,\nu}^\vartheta(r;\eps)
=\dfrac{
    \chi( r)
    (\log\log\frac{1}{\eps r})^\vartheta
}{
    r^{\mu}(\log\frac{1}{\eps r})^{\nu}
},
    \qquad \eps\in(0,1], r>0.
\end{equation}
\begin{equation}\label{eq:logpoly-infty}
\tilde{\phi}_{\mu}^{\nu}(r)
=
\left(1-\chi\left(\frac{r}{4}\right)\right)
\dfrac{
    (\log r)^{\nu}
}{
    r^{\mu}
},
    \qquad r>0.
\end{equation}
We emphasize that $\vartheta$ is a parameter as superscript in $\phi_{\mu,\nu}^\vartheta$. When $\vartheta=0$ we simply write $\phi_{\mu,\nu}$.

Define the constants

\begin{equation}\label{eq:Ki}
\kappa_i
=\kappa_i(N,s)
:=\int_0^\infty
    K_{N,s}(\rho)
    \left(\log\frac{1}{\rho}\right)^{i}
\,d\rho. \quad i\in\N.
\end{equation}
Morally, near the origin, the fractional Laplacian also decreases the exponent of the logarithm by one, at the power of the fundamental solution (i.e. $N-2s$). Indeed, by putting $v(r)=(\log\frac{1}{\eps r})^{-\nu}$ in \Cref{prop:K}, a binomial expansion yields
\[\begin{split}
\left(\log\frac{1}{\eps r}\right)^{-\nu}
-\left(
    \log\frac{1}{\eps r}
    +\log\frac{1}{\rho}
\right)^{-\nu}
=\left(\log\frac{1}{\eps r}\right)^{-\nu}
\left(
    1-\left(
        1+\dfrac{
            \log\frac{1}{\rho}
        }{
            \log\frac{1}{\eps r}
        }
    \right)^{-\nu}
\right)
\sim\nu
\left(
    \log\frac{1}{\eps r}
\right)^{-\nu-1}
    \log\frac{1}{\rho}.
\end{split}\]
Thus we expect $\Ds \phi_{N-2s,\nu}\sim \kappa_1\nu\phi_{N,\nu+1}$, in parallel with the local computations
\[
-\Delta \phi_{N-2,\nu}
=(N-2)\nu\phi_{N,\nu+1}
    -\nu(\nu+1)\phi_{N,\nu+2}.
\]
With the $\log\log$-correction, we can do a similar expansion, namely
\begin{equation}
\label{eq:binom-loglog}
\dfrac{
    \log\log\frac{1}{\eps r}
}{
    (\log\frac{1}{\eps r})^\nu
}
-\dfrac{
    \log\log\frac{1}{\eps r\rho}
}{
    (\log\frac{1}{\eps r \rho})^\nu
}
=
\dfrac{
    \log\log\frac{1}{\eps r}
}{
    (\log\frac{1}{\eps r})^\nu
}
-\dfrac{
    \log\log\frac{1}{\eps r}
    +\log\left(
        1+\dfrac{
            \log\frac{1}{\rho}
        }{
            \log\frac{1}{\eps r}
        }
    \right)
}{
    (\log\frac{1}{\eps r})^\nu
    \left(
        1+\dfrac{
            \log\frac{1}{\rho}
        }{
            \log\frac{1}{\eps r}
        }
    \right)^\nu
}
\sim
    \dfrac{
        \nu\log\log\frac{1}{\eps r}-1
    }{
        (\log\frac{1}{\eps r})^{\nu+1}
    }.
\end{equation}
This suggests $\Ds \phi_{N-2s,\nu}^{1}\sim \kappa_1(\nu\log\log\frac{1}{\eps r}-1)\phi_{N,\nu+1}$ and corresponds exactly to the local formula
\[
-\Delta\phi_{N-2,\nu}^1
=(N-2)\left(
    \nu\log\log\frac{1}{\eps r}-1
\right)\phi_{N,\nu+1}
    +O(\phi_{N,\nu+2}^1).
\]

\begin{prop}[Explicit computations at the origin]
\label{prop:loghomo}
Let $s\in(0,1)$. There exists a constant $C>0$ such that for any $\nu>\frac{N-2s}{N}$, $\eps\in(0,1/4]$ and $r>0$,
\begin{multline*}
\abs{
    \Ds \phi_{N-2s,\nu}
    -\left(
        \nu \kappa_1 \phi_{N,\nu+1}
        -\dfrac{\nu(\nu+1)\kappa_2 }{2} \phi_{N,\nu+2}
    \right)
        \oneset{ r<1/8}
}\\
\leq C\left(
        \phi_{N,\nu+3}
            \oneset{r<1/8}
        +\dfrac{1}{r^{N+2s}\left(\log\frac{1}{\eps}\right)^{\nu}}
            \oneset{ r\geq
            1/8}
    \right),
\end{multline*}
\begin{align*}
\abs{
    \Ds \phi_{N-2s,\nu}^1
    -\left(
        \nu \kappa_1 \phi_{N,\nu+1}^1
        -\kappa_1 \phi_{N,\nu+1}
    \right)
        \oneset{r<1/8}
}
\leq C\left(
        \phi_{N,\nu+2}^1
            \oneset{ r<1/8}
        +\dfrac{
        \log\log\frac{1}{\eps}
    }{ r^{N+2s} \left(
      \log\frac{1}{\eps}
    \right)^{\nu}}
            \oneset{ r\geq 1/8}
    \right).
\end{align*}
Here $\phi_{\mu,\nu}=\phi_{\mu,\nu}(r;\eps)$ and $\phi^1_{\mu,\nu}=\phi^1_{\mu,\nu}(r;\eps)$.
\end{prop}

\begin{prop}[Explicit computations at infinity]
\label{prop:loghomo-infty}
Let $s\in(0,1)$. There exists a constant $C>0$ such that for any $\nu>0$, $r>0$,
\[
\Ds \tilde{\phi}_{N-2s}^{\nu}(r)
=\left(
    \nu\kappa_1 \tilde{\phi}_{N}^{\nu-1}(r)
    +O\left(
        \tilde{\phi}_{N}^{\nu-2}(r)
    \right)
\right)
    \oneset{r>4}
    +O\left(
        r^{-N}
    \right)
    \oneset{r\leq 4}.
\]
\end{prop}

Their proofs are tedious and are postponed to \Cref{sec:proof-loghomo}.

\section{Construction of a singular radial solution: Proof of \Cref{thm:rad-s}}
\label{sec:radial}

\subsection{General strategy}
Let
\begin{equation}\label{eq:c0}
c_0=c_0(N,s)
=\left(
    \dfrac{N-2s}{2s}
    \kappa_1(N,s)
\right)^{\frac{N-2s}{2s}}
>0
\end{equation}
and
\begin{equation}\label{eq:c1}
c_1=c_1(N,s)
=-\dfrac{(N-2s)N}{8s^2}
    \dfrac{\kappa_2(N,s)}{\kappa_1(N,s)}c_0(N,s)
<0,
\end{equation}
where $\kappa_i(s)$ are given in \eqref{eq:Ki}. Define the \emph{Ansatz}
\[
\bar{u}_\eps(r)
=c_0\phi_{N-2s,\frac{N-2s}{2s}}(r;\eps)
+c_1\phi_{N-2s,\frac{N}{2s}}^1(r;\eps),
\]
where $\phi_{\mu,\nu}^{\vartheta}$ is defined in \eqref{eq:logpoly}. %
We remark that the second term is crucial to improve the decay of the error in order that a linear theory can be developed by a direct barrier and continuation argument. Moreover, the correction has to involve a $\log\log$-term; otherwise the improvement would cancel.

Since we do not know \emph{a priori} the sign of the perturbation, we consider instead
\begin{equation}\label{eq:main-abs-radial}
\begin{cases}
\Ds u=|u|^{\frac{N}{N-2s}}
    & \text{ in } B_1\setminus\set{0},\\
u=0
    & \text{ in } \R^N\setminus B_1.
\end{cases}
\end{equation}
Indeed, any solution is $s$-superharmonic and thus positive. Looking for a true solution $u_\eps=\bar{u}_\eps+\varphi$, we have
\[\begin{dcases}
\cL_\eps\varphi=-\cE_\eps+\cN[\varphi]
    & \text{ in } B_1\setminus\set{0},\\
\varphi=0
    & \text{ in } \R^N\setminus B_1,
\end{dcases}\]
where we denote the error by
\[
\cE_\eps=\Ds \bar{u}_\eps -\bar{u}_\eps^{\frac{N}{N-2s}},
\]
the linearized operator by
\begin{equation}\label{eq:L1}
\begin{split}
\cL_\eps\varphi
&=
	\Ds\varphi
    -\frac{N}{N-2s}
        (\bar{u}_{\eps})^{\frac{2s}{N-2s}}
        \varphi
=\Ds\varphi
    -\dfrac{N}{2s}\kappa_1
    \dfrac{
        \chi^{\frac{2s}{N-2s}}(r)
    }{
        r^{2s}
        \log\frac{1}{\eps r}
    }\left(
        1+O\left(
            \dfrac{
                \log\log\frac{1}{\eps r}
            }{
                \log\frac{1}{\eps r}
            }
        \right)
    \right)\varphi,
\end{split}\end{equation}
and the nonlinear term by
\begin{equation*}%
\cN[\varphi]
=
    |\bar{u}_\eps+\varphi|^{\frac{N}{N-2s}}
    -(\bar{u}_\eps)^{\frac{N}{N-2s}}
    -\dfrac{N}{N-2s}
        (\bar{u}_\eps)^{\frac{2s}{N-2s}}
        \varphi.
\end{equation*}
By studying the mapping properties of $\cL_\eps$, we will justify the fixed-point formulation
\begin{equation}\label{eq:fixed-point-radial}
\varphi
=\cL_\eps^{-1}(-\cE_\eps
    +\cN[\varphi]),
\end{equation}
and solve it by the implicit function theorem.
More precisely, for $\mu\geq N-2s$, $\nu>\frac{N-2s}{2s}$, $\alpha\in[0,1]$, $\bar{C}_1>0$ (to be chosen sufficiently large), we define the weights
\begin{equation}\label{eq:weight-rad}
w_{\mu,\nu;\alpha}(r)
:=\phi_{\mu,\nu}(r;1/4)
    +(\bar{C}_1)^{(\sign \alpha)_+}
        (1-|x|^2)_+^\alpha.
\end{equation}
Hereafter we denote the positive part of a real number $a\in\R$ as $a_+=\max\set{a,0}$. Note also that $(\sign\alpha)_+=0$ for $\alpha\leq 0$ and $(\sign\alpha)_+=1$ for $\alpha>0$. This weight includes the following cases:
\begin{itemize}
\item $w_{N,\frac{N+3s}{2s};0}$, which represents the size of the error $\cE_\eps$ (\Cref{prop:error});
\item $w_{N-2s,\frac{N+s}{2s};s}$, which represents the order of the perturbation $\varphi$ (see \Cref{prop:a_priori}; the numerology following from \Cref{prop:loghomo}, at least near the singularity);
\item $w_{N-2s,\frac{N+s/2}{2s};\alpha_1}$ (where $\alpha_1\in(0,s)$), a $\cL_\eps$-superharmonic function which is more singular than $w_{N-2s,\frac{N+s}{2s};s}$ (needed for the application of the $L^2$-maximum principle).
\end{itemize}
We define the corresponding norms for $L^{\infty}_{\loc}(B_1\setminus\set{0})$-functions by
\[
\norm[\mu,\nu;\alpha]{\varphi}
:=\norm[L^\infty(B_1 \setminus \set{0})]{w_{\mu,\nu;\alpha}^{-1}\varphi}.
\]
Thus the error $\cE_\eps$ is contained in
\[
Y:=\set{
    f\in L^\infty_{\loc}(B_1 \setminus \set{0}):
    \norm[N,\frac{N+3s}{2s};0]{f}
    \leq 2\bar{C}_2
    |\log\eps|^{
    	-\min(\frac{1}{2}, \frac{N-2s}{2s})
    }
},
\]
for some $\bar{C}_2>0$ (given in \Cref{prop:error}). We will search for $\varphi$ in the Banach space
\begin{equation*}
X:=\set{
    \varphi\in L^\infty_{\loc}(B_1\setminus \set{0}):
    \norm[N-2s,\frac{N+s}{2s};s]{\varphi}
    \leq 2\bar{C}_2\bar{C}_3|\log\eps|^{-\min(\frac{1}{2}, \frac{N-2s}{2s})}
},
\end{equation*}
where $\bar{C}_3>0$ is a large constant to be fixed. (In fact, $\bar{C}_3=(\bar{c}_2)^{-1}$ with $\bar{c}_2$ given in \Cref{prop:supersol}.)  Define also the spaces
\[
\widetilde{Y}:=\set{
    g\in L^\infty_{\loc}(B_1 \setminus \set{0}):
    \norm[N,\frac{N+3s}{2s};0]{g}
    <+\infty
},
\]
\[
\widetilde{X}:=\set{
    v\in L^\infty_{\loc}(B_1\setminus \set{0}):
    \norm[N-2s,\frac{N+s}{2s};s]{v}
    <+\infty
}.
\]
We will show that $\cL_\eps^{-1}:\widetilde{Y} \to \widetilde{X}$ is a (uniformly-in-$\eps$) bounded operator. %
Using $w_{N-2s,\frac{N+s}{2s};s}$ itself as a barrier, we prove an \emph{a priori} estimate (\Cref{prop:a_priori}). Then the existence of the inverse operator follows from the method of continuity (\Cref{prop:exis_uniq}). Finally, we will apply the contraction mapping principle in $X$ (\Cref{prop:G-new}).

Although these function spaces contain non-radial functions, the solutions we construct are indeed radial by the method of moving plane (\Cref{prop:sym}). %

\subsection{Error estimate}

\begin{prop}[Error estimate]
\label{prop:error}
There exists a universal constant $\bar{C}_2$ such that for all $r\in(0,1)$,
\[
|\cE_\eps(r)|
\leq
    \bar{C}_2
    |\log\eps|^{-\min(\frac{1}{2},\frac{N-2s}{2s})}
    w_{N,\frac{N+3s}{2s};0}(r).
\]
\end{prop}
\begin{proof}
	Use \Cref{prop:loghomo} and \eqref{eq:c0}--\eqref{eq:c1}, thus $c_0\frac{(N-2s)N}{8s^2}\kappa_2+c_1\kappa_1=0$ and $c_0^{\frac{N}{N-2s}}=c_0\frac{N-2s}{2s}\kappa_1$.
\end{proof}
\COMMENT{
Therefore,
{\[\begin{split}
\cE_\eps=&\Dh u_\eps -u_\eps^{\frac{N}{N-1}}
\\
=&
r^{-N}\left[K_1\left(c_1\frac{N\log\log\frac{1}{\eps r} -1}{\left(\log\frac{1}{\eps r}\right)^{N+1}} +c_0\frac{(N-1)}{\left(\log\frac{1}{\eps r}\right)^{N}}\right) -	c_0\frac{N(N-1)}{2} \frac{K_2}{\left(\log\frac{1}{\eps r}\right)^{N+1}}
\right.\\
&\left.\quad\quad
	+O\left(\frac{\log\log\frac{1}{\eps r}}{\left(\log\frac{1}{\eps r}\right)^{N+2}
	}\right)	
	\right]\chi_{(0,\frac18]}(r)\\
	&
	+\left[O\left(\frac{\log\log\frac{1}{\eps }}{r^{N+1}\left(\log\frac{1}{\eps }\right)^{N}
	}\right)+O\left(\frac{1}{r^{N+1}\left(\log\frac{1}{\eps }\right)^{N+1}
	}\right)\right]\chi_{(\frac18,\infty)}(r)\\
& -\left(\dfrac{c_0^{\frac{N}{N-1}}\chi^\frac{N}{N-1}_{*}( r)}{r^{N}(\log\frac{1}{\eps r})^{N}}
+\dfrac{c_0^{\frac{1}{N-1}}c_1N\left(\log\log\frac{1}{\eps r}\right)\chi^\frac{N}{N-1}_{*}( r)}{(N-1)r^{N}(\log\frac{1}{\eps r})^{N+1}}+O\left(
    \dfrac{\left(\log\log\frac{1}{\eps r}\right)^2\chi_{*}( r)}{r^N(\log\frac{1}{\eps r})^{N+2}}
\right)\right)\\
=& \left\{\begin{split}O\left(\frac{\left(\log\log\frac{1}{\eps r}\right)^2}{r^N\left(\log\frac{1}{\eps r}\right)^{N+2}
	}\right)%
	, \quad r\in (0,\frac{1}{8}],\\
O\left(\frac{\log\log\frac{1}{\eps }}{r^{N+1}\left(\log\frac{1}{\eps }\right)^{N}
	}\right), \quad r\in (\frac{1}{8},\infty).\end{split}\right.
\end{split}
\]}}

\subsection{Linear theory}

We consider the linear problem
\begin{equation*}
\begin{dcases}
\cL_\eps \varphi
=f,
    & \text{ in } B_1\setminus\set{0} \\
\varphi=0,
    & \texton \R^N\setminus B_1.
\end{dcases}
\end{equation*}
where $\cL_\eps\varphi$ is the linearized operator as defined in \eqref{eq:L1}.

\begin{lem}
\label{lem:Ds-torsion}
Let $N\geq 2$, $s\in(0,1)$ and $\alpha\in(0,s]\setminus\set{2s-1}$, then there exists a constant $c=c(N,s,\alpha)>0$ %
such that
\[
\Ds (1-|x|^2)_+^{\alpha}
\geq
    c
    +c(s-\alpha)(1-|x|^2)^{-(2s-\alpha)},
        \qquad \forall x\in B_1.
\]
\end{lem}

\begin{proof}
The left hand side can be written \cite[Theorem 1]{Dyda} as $\Hyperg\left(
        \tfrac{N+2s}{2},s-\alpha,\tfrac{N}{2};|x|^2
    \right)\geq 1$, up to a positive multiple. The asymptotic growth as $|x|\to 1^-$ can be seen from Euler's transform.
\end{proof}

\begin{prop}[Super-solutions]
\label{prop:supersol}
Consider the weights $w_{\mu,\nu;\alpha}>0$ defined in \eqref{eq:weight-rad}. Fix any $\alpha_1\in(0,s)\setminus\set{2s-1}$. Then there exist universal constants $\bar{C}_1=\bar{C}_1(N,s)\gg 1$ and $\bar{c}_2=\bar{c}_2(N,s,\bar{C}_1)\ll 1$ such that for any $\eps>0$ small enough and $r\in(0,1)$,
\[\begin{split}
\cL_\eps w_{N-2s,\frac{N+s}{2s};s}
&\geq \bar{c}_2 w_{N,\frac{N+3s}{2s};0},\\
\cL_\eps w_{N-2s,\frac{N+s/2}{2s};\alpha_1}
&\geq \bar{c}_2 w_{N,\frac{N+5s/2}{2s};-(2s-\alpha_1)}.
\end{split}\]
\end{prop}

\begin{proof}
Let $\nu>\frac{N}{2s}$. Using \Cref{prop:loghomo} and \Cref{lem:Ds-torsion}, we have
\[\begin{split}
\cL_\eps w_{N-2s,\nu;s}
&\geq
    \Bigl(
        \nu\kappa_1\phi_{N,\nu+1}
        -C\phi_{N,\nu+2}\oneset{r<1/8}
        -C\oneset{r\geq 1/8}
    \Bigr)
    +\bar{C}_1c,\\
\cL_\eps w_{N-2s,\nu;\alpha_1}
&\geq
    \Bigl(
        \nu\kappa_1\phi_{N,\nu+1}
        -C\phi_{N,\nu+2}\oneset{r<1/8}
        -C\oneset{r\geq 1/8}
    \Bigr)
    +\bar{C}_1c(1-r^2)^{-(2s-\alpha_1)},
\end{split}\]
for any $r\in(0,1)$. For $\bar{C}_1$ large, the negative terms are absorbed. This yields the result.
\end{proof}

\begin{prop}[A priori estimate]
\label{prop:a_priori}
Let $\varphi\in \widetilde{X}$ be a solution \begin{equation*}
\begin{cases}
\cL_\eps\varphi=f
    & \text{ in } B_1\setminus\set{0},\\
\varphi=0
    & \text{ in } \R^N\setminus B_1,
\end{cases}\end{equation*}
for $f\in \widetilde{Y}$. %
Then, for $\bar{c}_2$ is as in \Cref{prop:supersol},
\[
\norm[N-2s,\frac{N+s}{2s};s]{\varphi}
\leq
    (\bar{c}_2)^{-1}\norm[N,\frac{N+3s}{2s};0]{f}.
\]
\end{prop}

\begin{proof}
For any $\delta>0$, define
\[
\tilde{w}^\delta
:=
    (\bar{c}_2)^{-1}
    \left(
        \norm[N,\frac{N+3s}{2s};0]{f}
            w_{N-2s,\frac{N+s}{2s};s}
        +\delta
            w_{N-2s,\frac{N+s/2}{2s};\alpha_1}
    \right),
\qquad \text{ in } B_1\setminus\set{0}.
\]
Then by \Cref{prop:supersol}, we have in $B_1\setminus \set{0}$,
\[\begin{split}
\cL_{\eps} \tilde{w}^\delta
&\geq
    \norm[N,\frac{N+3s}{2s};0]{f}
        w_{N,\frac{N+3s}{2s};0}
    +\delta
        w_{N,\frac{N+5s/2}{2s};-(2s-\alpha_1)}
\geq
    -|f|
    +\delta
        w_{N,\frac{N+5s/2}{2s};-(2s-\alpha_1)}.
\end{split}\]
Thus,
\[
\cL_\eps(\tilde{w}^\delta\pm \varphi)>0
    \qquad \text{ in } B_1\setminus\set{0}.
\]
Moreover, since $\norm[N-2s,\frac{N+s}{2s};s]{\varphi}<+\infty$, $\varphi$ is asymptotically controlled by $w_{N-2s,\frac{N+s/2}{2s};\alpha_1}$ as $r\to 0^+$ and as $r\to 1^-$. In other words, for each $\delta>0$ there exists $\tilde{r}=\tilde{r}\bigl(\delta,\norm[N-2s,\frac{N+s}{2s};s]{\varphi}\bigr)\in(0,\frac12)$ such that
\[
\tilde{w}^\delta\pm\varphi>0
    \quad \text{ in }
    (B_{\tilde r}\setminus\{0\})
    \cup
    (B_1\setminus B_{1-\tilde r}).%
\]
By the maximum principle (\Cref{HsMP}), we conclude that
\[
\tilde{w}^\delta\pm\varphi\geq 0
    \qquad \text{ in } B_1\setminus\set{0}.
\]
Taking $\delta\to 0^+$, we have
\[
\norm[N-2s,\frac{N+s}{2s};s]{\varphi}
\leq
    (\bar{c}_2)^{-1}\norm[N,\frac{N+3s}{2s};0]{f},
\]
as desired.
\end{proof}

Recalling the Green function in \eqref{eq:green_B1},
\[
\bG_{B_1}(x,y)
\asymp
\dfrac{1}{|x-y|^{N-2s}}
\min\left(
    1,
    \dfrac{
        (1-|x|)^{s}
    }{
        |x-y|^{s}
    }
\right)
\min\left(
    1,
    \dfrac{
        (1-|y|)^{s}
    }{
        |x-y|^{s}
    }
\right)
\lesssim
\min\left(
    \dfrac{1}{|x-y|^{N-2s}},
    \dfrac{
        (1-|x|)^{s}
    }{
        |x-y|^{N-s}
    }
\right).
\]
\begin{prop}\label{prop:Green}
The Green operator restricted to
$
\cG:
\widetilde{Y} \to \widetilde{X}
$
is bounded.
\end{prop}
\begin{proof}
It suffices to show $\cG\left(w_{N,\frac{N+3s}{2s};0}\right)\lesssim w_{N-2s,\frac{N+s}{2s};s}$. By the definition of $w_{N,\frac{N+3s}{2s};0}$ in \eqref{eq:weight-rad},
\[\begin{split}
\cG\left(
    w_{N,\frac{N+3s}{2s};0}
\right)(x)
&\lesssim
\int_{B_1}
    \min\left(\dfrac{1}{|x-y|^{N-2s}},
    \dfrac{
        (1-|x|)^{s}
    }{
        |x-y|^{N-s}
    }
\right)
        \left(
            \phi_{N,\frac{N+3s}{2s}}(|y|;\tfrac14)
            +1
        \right)
    \,dy.
\end{split}\]
Then, a direct computation using \Cref{prop:K} shows that for $r<1$,
\[\begin{split}
&%
\Ints\left(
    \phi_{N,\frac{N+3s}{2s}}(\cdot;\tfrac14)
    +\mathbf{1}_{B_1}
\right)(r)
\lesssim
    \dfrac{1}{r^{N-2s}}
	\int_{0}^{\infty}
		K_{N,-s}(\rho)
        \left(
            \dfrac{1}{
                (\log\frac{1}{r\rho})^{\frac{N+3s}{2s}}
            }
            +(r\rho)^N
        \right)
        \oneset{r\rho<1}
	\dfrac{d\rho}{\rho}
	\\
&\qquad
\lesssim
    \dfrac{1}{r^{N-2s}}
    \left(
        \int_{0}^{\frac12}
            \dfrac{1}{
                (\log\frac1r
                +\log\frac1\rho)^{\frac{N+3s}{2s}}
            }
        \,{d\left(
            -\log\frac{1}{\rho}
        \right)}
        +1
    \right)
\lesssim
    \phi_{N-2s,\frac{N+s}{2s}}+1.
\end{split}\]
Similarly
\[
(-\Delta)^{-\frac{s}{2}}\left(
    \phi_{N,\frac{N+3s}{2s}}(\cdot;\tfrac14)
    +\mathbf{1}_{B_1}
\right)(r)
\lesssim
    r^{-(N-s)}+1.
\]
Hence the following two inequalities hold,
\[\begin{split}
\cG\left(
    w_{N,\frac{N+3s}{2s};0}
\right)(x)
&\lesssim
    \int_{B_1}
        \dfrac{1}{|x-y|^{N-2s}}
        \left(
            \phi_{N,\frac{N+3s}{2s}}(|y|;\tfrac14)
            +1
        \right)
    \,dy
\lesssim
    \phi_{N-2s,\frac{N+s}{2s}}(|x|;\tfrac14)
    +1
\end{split}\]
\[\begin{split}
\cG\left(
    w_{N,\frac{N+3s}{2s};0}
\right)(x)
&\lesssim
    (1-|x|)^{\frac12}
    \int_{B_1}
        \dfrac{1}{|x-y|^{N-s}}
        \left(
            \phi_{N,\frac{N+3s}{2s}}(|y|;\tfrac14)
            +1
        \right)
    \,dy
\lesssim
    (1-|x|)^{s}
    \left(
        |x|^{-(N-s)}
        +1
    \right).
\end{split}\]
The first one is sharp near the origin and the second near the boundary. Taking the minimum  and recalling \eqref{eq:weight-rad}, the result follows. 
\end{proof}

\begin{prop}[Existence and uniqueness]
\label{prop:exis_uniq}
For any $f\in \widetilde{Y}$, there exists a unique solution $\varphi\in \widetilde{X}$ of
\[\begin{cases}%
\cL_\eps\varphi=f
    & \text{ in } B_1\setminus\set{0},\\
\varphi=0
    & \text{ in } \R^N\setminus B_1,
\end{cases}\]
with
$$
\norm[N-2s,\frac{N+s}{2s};s]{\varphi}
\leq (\bar{c}_2)^{-1}\norm[N,\frac{N+3s}{2s};0]{f}.
$$
In particular, the operator $\cL_{\eps}: Y\to X$ has a uniformly bounded inverse
with its operator norm bounded by
$$
\|\cL_{\eps}^{-1}\|\leq (\bar{c}_2)^{-1}, \quad \forall\eps\in(0,1),
$$
$\bar{c}_2$ is the constant given in \Cref{prop:supersol}.
\end{prop}

\begin{proof}
We prove the existence and uniqueness of solution using the method of continuity \cite[Proposition 3.11]{Chan_DelaTorre1} (for a general setting see \cite[Theorem 5.2]{GT}). Indeed, we interpolate between $(-\Delta)^{s}$ and $\cL_\eps$ linearly:
\[
\cL_{\eps}^\lambda
:=\Ds
    -\lambda
    \dfrac{N}{N-2s}
    (\bar{u}_\eps)^{\frac{2s}{N-2s}},
\qquad \lambda\in[0,1].
\]
Inductively, $\cL_\eps^\lambda$ has a bounded inverse $\widetilde{Y}\to\widetilde{X}$ for all $\lambda\in[0,1]$, using \Cref{prop:Green} for $\lambda=0$ and \Cref{prop:a_priori} for $\lambda\in(0,1]$.
\end{proof}

\subsection{The nonlinear equation}

We are now in a position to solve the nonlinear equation \eqref{eq:fixed-point-radial}, 
where the inverse linear operator
$\cL^{-1}_{\eps}:Y\to X$ exists and is uniformly bounded in view of \Cref{prop:exis_uniq}.

\begin{prop}[Contraction]
\label{prop:G-new}
For $0<\eps\ll 1$, %
$G_\eps:X\to X$ is a contraction.%
\end{prop}

\begin{proof}
We compute, using mean value theorem twice, %
\[\begin{split}
&\quad\;
\bigl|
    \cN[\varphi]-\cN[\tilde\varphi]
\bigr|
=
\abs{
    |\bar{u}_\eps+\varphi|^{\frac{N}{N-2s}}
    -|\bar{u}_\eps+\tilde\varphi|^{\frac{N}{N-2s}}
    -\frac{N}{N-2s}(\bar{u}_{\eps})^{\frac{2s}{N-2s}}(\varphi-\tilde{\varphi})
}\\
&\leq
    \dfrac{2sN}{(N-2s)^2}
    \int_{0}^{1}\int_{0}^{1}
        \abs{
            \bar{u}_\eps
            +\tau(1-t)\varphi
            -\tau t\tilde{\varphi}
        }^{\frac{N}{N-2s}-2}
    \,dt\,d\tau
    \cdot
    \left(
        |\varphi|
        +|\tilde{\varphi}|
    \right)
    \abs{
        \varphi-\tilde{\varphi}
    }\\
&\leq
    C\left(
        w_{N-2s,\frac{N-2s}{2s};s}
    \right)^{\frac{N}{N-2s}-2}
    \left(
        2\bar{C}_2\bar{C}_3
        |\log\eps|^{-\min(\frac12,\frac{N-2s}{2s})}
    \right)
    \left(
        w_{N-2s,\frac{N+s}{2s};s}
    \right)^2
    \norm[N-2s,\frac{N+s}{2s};s]{
        \varphi-\tilde{\varphi}
    }\\
&\leq
    C\bigl(
        N,s,\bar{C}_2,\bar{C}_3
    \bigr)
    |\log\eps|^{-\min(\frac12,\frac{N-2s}{2s})}
    w_{N,\frac{N+6s}{2s};\frac{N}{N-2s}s}
    \norm[N-2s,\frac{N+s}{2s};s]{\varphi-\tilde\varphi}\\
&\leq
    C\bigl(
        N,s,\bar{C}_1,\bar{C}_2,\bar{C}_3
    \bigr)
    |\log\eps|^{-\min(\frac12,\frac{N-2s}{2s})-\frac32}
    w_{N,\frac{N+3s}{2s};0}
    \norm[N-2s,\frac{N+s}{2s};s]{\varphi-\tilde\varphi}.
\end{split}\]
As a result of \Cref{prop:exis_uniq},
\[
\bigl|
    G_\eps[\varphi]-G_\eps[\tilde\varphi]
\bigr|
\leq
    C\bigl(
        N,s,\bar{C}_1,\bar{C}_2,\bar{C}_3
    \bigr)
    |\log\eps|^{-\min(\frac12,\frac{N-2s}{2s})-\frac32}
    w_{N-2s,\frac{N+s}{2s};s}
    \norm[N-2s,\frac{N+s}{2s};s]{\varphi-\tilde\varphi},
\]
showing that $G_\eps$ is contractive for $\eps\ll 1$. Putting $\tilde\varphi=0$ also shows $G_\eps$ maps $X$ to $X$, as desired.
\end{proof}

\subsection{Local regularity}

We have the following \emph{a priori} estimate.

\begin{lem}%
\label{lem:reg}
Suppose $u\in C^2(B_1\setminus \set{0}) \cap C^{1/2}(\R^N \setminus \set{0})$ is a solution to
\[\begin{cases}
\Ds u=|u|^{\frac{N}{N-2s}}
    & \text{ in } B_1 \setminus\set{0},\\
u=0
    & \text{ in } \R^N \setminus B_1,
\end{cases}\]
such that
\[
u(x) \leq C_0\phi_{N-2s,\frac{N-2s}{2s}}(|x|;1/2)
    \quad \text{ in } B_{1}.
\]
Then $u\in C^\infty(B_1\setminus \set{0})$. Moreover, for any $\beta>0$ and $x\in B_{1/8} \setminus\set{0}$, there is a constant $C>0$ depending only on $N,s,\beta$ such that
\[
\seminorm[C^{\beta}(B_{|x|/4}(x))]{u}
\leq C\phi_{N-2s+\beta,\frac{N-2s}{2s}}(x).
\]
\end{lem}

\begin{proof}
Fix $x\in B_{1/8} \setminus \set{0}$. Then
$
v(y)=
\left(\frac{|x|}{2}\right)^{N-2s}
    u\left(x+\frac{|x|}{2}y\right)
$
satisfies the equation
$\Ds v=|v|^{\frac{N}{N-2s}}$ in $B_1$ and the bound
$\norm[L^\infty(B_1)]{v} \leq C(\log\frac{1}{|x|})^{-\frac{N-2s}{2s}}$. Moreover,
by the change of variable $z=x+\frac{|x|}{2}y$ and the inequality $|x|+|z-x|\geq \frac{1}{4}(|x|+|z|)$,
\begin{align*}
\norm[L^1_{2s}(\R^N)]{v}
&=\int_{\R^N}
	\frac{1}{(1+|y|)^{N+2s}}
	|x|^{N-2s}
	u\left(
		x+\frac{|x|}{2}y
	\right)
\,dy
=\int_{B_1}
	\frac{|x|^{N+2s}}{(|x|+2|z-x|)^{N+2s}}
	|x|^{-2s}
	u(z)
\,dz\\
&\leq 
C\left(
\int_{0}^{|x|}
	|x|^{-2s}
	\frac{\rho^{2s-1}}{(\log\frac{2}{\rho})^{\frac{N-2s}{2s}}}
\,d\rho
+\int_{|x|}^{1}
	\frac{|x|^N}{\rho^{N+2s}}
	\frac{\rho^{2s-1}}{(\log\frac{2}{\rho})^{\frac{N-2s}{2s}}}
\,d\rho
\right)
\leq
	\frac{C}{(\log\frac{1}{|x|})^{\frac{N-2s}{2s}}}.
\end{align*}
Bootstrapping using \cite[Corollary 2.4 and 2.5]{RS1} and a standard covering argument, the result follows.
\end{proof}

\subsection{Proof of \Cref{thm:rad-s}} %
By \Cref{prop:G-new}, \eqref{eq:fixed-point-radial} has a fixed point $\varphi$, giving a singular solution $u_\eps=\bar{u}_\eps+\varphi$ to \eqref{eq:main-abs-radial}, which is regular by \Cref{lem:reg} and positive near the origin. By $s$-superharmonicity, it is also positive in $B_1$, as desired.

\subsection{Integral identities}

By equating the constants from \Cref{thm:behavior-rad} and \Cref{thm:rad-s}, we have:

\begin{lem}\label{lem:const-int}
Let $N\geq 2$ and $s\in(0,1)$. Then,
\begin{align*}
\kappa_1
=\int_0^{\infty}
	K_{N,s}(\rho)
	\left(\log\frac{1}{\rho}\right)
\,d\rho
&=\frac{1}{C_{N,-s}\cH^{N-1}(\bS^{N-1})},
\\
\kappa_2
=\int_0^{\infty}
	K_{N,s}(\rho)
	\left(\log\frac{1}{\rho}\right)^2
\,d\rho
&=\frac{
	2s\cH^{N-2}(\bS^{N-2})
}{
	C_{N,-s}(\cH^{N-1}(\bS^{N-1}))^2
}
\int_{1}^{\infty}
\int_{0}^{\pi}
	\frac{
		\sin^{N-2}\theta\,d\theta
	}{
		(1+\sigma^2-2\sigma\cos\theta)^{\frac{N-2s}{2}}
	}
\,\dfrac{d\sigma}{\sigma}.
\end{align*}
\end{lem}

\begin{proof}
From \eqref{eq:c0-explicit}, \Cref{rmk:c1} and \eqref{eq:c0}--\eqref{eq:c1}, we have
\begin{align*}
c_0
&=\left(
	\frac{N-2s}{2s\kappa}
\right)^{\frac{N-2s}{2s}}
=\left(
	\frac{N-2s}{2s}
	\kappa_1
\right)^{\frac{N-2s}{2s}}
	\quad \implies \quad
\kappa_1=\frac{1}{\kappa},\\
c_1
&=-\frac{(N-2s)N}{8s^2}\frac{\kappa_2}{\kappa_1}c_0
=-\frac{N}{2s}C_3c_0^{\frac{N}{N-2s}}
	\quad \implies \quad
\kappa_2=\frac{2sC_3}{\kappa^2}.
\end{align*}
It suffices to recall the definitions of $\kappa$ from \Cref{lem:eq-v-4-bar} and $C_3$ from \eqref{eq:C3} ($C_1=C_{N,-s}\cH^{N-2}(\bS^{N-2})$).
\end{proof}

\section{Singular Yamabe metrics: Proof of \Cref{th:Yamabe}}
\label{sec:Yamabe}

\subsection{Fermi coordinates}
Let us assume, without lose of generality, that our singularity satisfies $\Sigma\subset B^n_1(0)$. We write the ambient dimension as $n=k+N$, where $k$ and $N$ are respectively the dimensions of the submanifold $\Sigma$
and of the normal space $N_y\Sigma$ at any point $y\in\Sigma$. The Fermi coordinates are well-defined on some tubular neighborhood $\cT_{4\tau}$ of $\Sigma^k\subset\R^n$ of universally small width $4\tau\ll 1$ (see \Cref{prop:Ints_nN}). In fact, any point $z\in\R^n$ with $\dist(z,\Sigma)<4\tau$ can be written as
\begin{equation}\label{eq:Fermi-1}
z=y+\sum_{j=1}^{N}x_j\nu_j(y),
\end{equation}
where $y\in\Sigma^k$ and $(\nu_1(y),\dots,\nu_j(y))$ is a basis for the normal space $N_y\Sigma$ at $y$, and $x=(x_1,\dots,x_N)\in\R^N$ are the coordinates on $N_y\Sigma$. Using polar coordinates in $\R^N$, we set
\begin{equation}\label{eq:Fermi-2}
r=|x|\in[0,4\tau)
	\quad \textand \quad
\omega=\frac{x}{|x|}\in\bS^{N-1}.
\end{equation}
Thus \eqref{eq:Fermi-1}--\eqref{eq:Fermi-2} define a diffeomorphism
\begin{equation*}
\begin{split}
\Phi:(0,4\tau)\times \bS^{N-1} \times \Sigma^k
	& \to
\cT_{4\tau}\setminus\Sigma\subset\R^n \\
\Phi(r,\omega,y)
	&
=y+\sum_{j=1}^{N}r\omega_j\nu_j(y).
\end{split}
\end{equation*}
The associated metric $g(r,\omega,y)$ is well-known (see \cite{FmO,MS,MP}), given by
\[\begin{split}
(g_{ij})
&=\begin{pmatrix}
1 & 0 & O(r) \\
0 & r^2 g_{\bS^{N-1},i'j'}(\omega)+O(r^4) & O(r^2) \\
O(r) & O(r^2) & g_{\Sigma,i''j''}(y)+O(r).
\end{pmatrix},
\end{split}\]
where $O(r^\ell),\ \ell=1,2,4$ are uniformly small as $r\searrow 0$, together with all derivatives with respect to the vector fields $r\p_{r}$, $\p_{\omega_{i'}}$, $\p_{y_{i''}}$. (Here $i,j=1,\dots,n$, $i',j'=1,\dots,N-1$, $i'',j''=1,\dots,k$.)

\subsection{General strategy} %

We take $u_{\eps}(r)=
		\dfrac{c_0+o(1)}{r^{N-1}(\log\frac{1}{\eps r})^{N-1}}$
as given in \Cref{thm:rad-s} with $s=\frac12$ to form the \emph{Ansatz} in $\R^n\setminus\Sigma$ by simply gluing it to zero away from $\Sigma$, namely
\begin{equation}\label{eq:v_eps}
\bar{v}_{\eps}(r,\omega,y)
:=\bar{v}_{\eps}(r)
=u_{\eps}(r)\chi_{\tau}(r),
\quad
\chi_{\tau}(r)=
\begin{dcases}
1, &  \textif {r\leq \tau},\\
0, &  \textif {r\geq 2\tau}.
\end{dcases}%
\end{equation}
with $\tau |D\chi_\tau| + \tau^2 |D^2\chi_\tau| \leq C$, where $r\in(0,4\tau),\  \omega\in\bS^{N-1}, \ y\in \Sigma^k$. We look for a perturbation $\psi$ so that $v_\eps=\bar{v}_\eps+\psi$ solves
\begin{equation}\label{eq:th2}
\Dhn v=|v|^{\frac{n+1}{n-1}}
    \qquad \text{ in } \R^n\setminus\Sigma.
\end{equation}
This means that $\psi$ solves the linearized equation (in fixed-point form)
\begin{equation}\label{eq:NL2}
\psi
=\sG_\eps[\psi]
:=\sL_\eps^{-1}(-\sE_\eps+\sN[\psi])
    \quad \text{ in } \R^n\setminus\Sigma,
\end{equation}
where
\begin{equation}\label{eq:L2}
\mathscr L_\eps\psi:=\Dhn\psi
    -\frac{N}{N-1}
        (\bar{v}_{\eps})^{\frac{1}{N-1}}
        \psi,
\end{equation}
is shown to be invertible in \Cref{prop:lin-Yamabe}, and
\begin{equation}\label{eq:E2N2}
\mathscr E_{\eps}
:=\Dhn \bar{v}_\eps -\bar{v}_\eps^{\frac{N}{N-1}},
	\qquad
	\mathscr N[\psi]:=
    |\bar{v}_\eps+\psi|^{\frac{N}{N-1}}
	-(\bar{v}_\eps)^{\frac{N}{N-1}}
	-\dfrac{N}{N-1}
	(\bar{v}_\eps)^{\frac{1}{N-1}}
	\psi.
\end{equation}
Different from \eqref{eq:fixed-point-radial}, we emphasize that \eqref{eq:NL2} is to be solved in the \emph{whole space}. %
For $\mu\in[N-2,N)$ and $\nu\in[n-\frac32,n+1]$, define
\begin{equation*}%
\tilde{\omega}_{\mu;\nu}(z)
:=d_\Sigma(z)^{-\mu}
    \oneset{z\in\cT_{3\tau}\setminus\Sigma}
+d_\Sigma(z)^{-\nu}
    \oneset{z\in\R^n\setminus\cT_{3\tau}},
\qquad z\in \R^n\setminus\Sigma.
\end{equation*}

Up to a translation, we assume that $0\in \Sigma$. This implies that, in $\R^n\setminus\cT_{3\tau}$, $d_\Sigma(z)$ is comparable to $|z|$ (see \Cref{lem:distances}), and $\tilde{\omega}_{\mu;\nu}(z)$ is comparable to the more convenient
\begin{equation}\label{eq:weight-Yamabe}
\omega_{\mu;\nu}(z)
:=d_\Sigma(z)^{-\mu}
    {\bf 1}_{\cT_{3\tau}\setminus\Sigma}(z)
+|z|^{-\nu}
    {\bf 1}_{\R^n\setminus\cT_{3\tau}}(z),
\qquad z\in \R^n\setminus\Sigma.
\end{equation}
We collect the important choices of the parameters:
\begin{itemize}
\item $\omega_{N-1;n+1}$ controls the error $\sE_\eps$ (\Cref{lem:Error estimates});
\item $\omega_{N-2;n-1}$ represents the order of the perturbation $\psi$ (\Cref{prop:lin-Yamabe});
\item $\omega_{N-\frac32;n-\frac32}$ is an $\sL_\eps$-superharmonic function which is more singular than $\omega_{N-2;n-1}$ near the singularity, as required by the maximum principle (\Cref{lem:apriori-yamabe}). %
\end{itemize}
We denote the corresponding norms for $L^\infty_\loc(\R^n\setminus\Sigma)$-functions by
\[
\norm[\mu;\nu]{v}
:=\norm[L^\infty(\R^n\setminus\Sigma)]{
    \omega_{\mu;\nu}^{-1}v
}.
\]
Then the error $\sE_\eps$ lies in the space
\[
\sY:=\set{
    g\in L^\infty_\loc(\R^n\setminus\Sigma):
    \norm[N-1;n+1]{g}
    \leq 2\tilde{C}_1|\log\eps|^{-(N-1)}
},
\]
for some $\tilde{C}_1>0$ (given in \Cref{lem:Error estimates}). We look for $\psi$ in the Banach space
\begin{equation*}
\sX:=\set{
    \psi\in L^\infty_\loc(\R^n\setminus\Sigma):
    \norm[N-2;n-1]{\psi}
    \leq 2\tilde{C}_1\tilde{C}_2|\log\eps|^{-(N-1)}
},
\end{equation*}
where $\tilde{C}_2>0$ is determined by the %
operator norm of $\sL_\eps^{-1}: \widetilde{\sY} \to \widetilde{\sX}$ (see  \Cref{prop:lin-Yamabe}), where
\[
\widetilde{\sY}
:=\set{
    g\in L^\infty_\loc(\R^n\setminus\Sigma):
    \norm[N-1;n+1]{g}
    <+\infty
},
\]
\[
\widetilde{\sX}
:=\set{
    \psi\in L^\infty_\loc(\R^n\setminus\Sigma):
    \norm[N-2;n-1]{\psi}
    <+\infty
}.
\]
Once the linear theory is established, we apply the contraction mapping principle in $\sX$ (\Cref{prop:G-Yamabe}).

\subsection{Nonlocal computations}

\begin{lem}[Comparison of distances]
\label{lem:distances}
Let $z\in \R^n\setminus \cT_{3\tau}$. Then
\[
\min\left(\frac12, \frac{3\tau}{2\diam\Sigma}\right) |z| \leq d_\Sigma(z)\leq |z|,
\qquad \text{ and } \quad
3\tau \leq |z| \leq 2\diam\Sigma,
\text{ if $z\in \cT_{4\tau}\setminus \cT_{3\tau}$}.
\]
\end{lem}

\begin{prop}\label{prop:Dh_nN}
Suppose $v\in C^{1,1}_{\loc}(B^N_{2\tau}\setminus \set{0})$ is supported on $B^N_{2\tau}\setminus \set{0}$. In Fermi coordinates \eqref{eq:Fermi-1},
consider a $C^{1,1}_{\loc}$ function $\bar{v}:\R^n\setminus\Sigma\to\R$ defined by
\[
\bar{v}(z)
=\begin{cases}
v(x)
    & \textfor z \in \cT_{2\tau}\setminus\Sigma,\\
0
    & \textfor z \in \R^n\setminus \cT_{2\tau}.
\end{cases}
\]
Then for any $z=z(x,y)\in\cT_{3\tau}$,
\[\begin{split}
\Dsn \bar{v}(z)
&=(1+O(|x|))\DsN v(x)
	+O(|v(x)|)
	+O\left(
		|x|^{2-2s}\norm[L^\infty(B^N_{|x|/2}(x))]{Dv}
	\right)
\\
&\quad\;
	+O\left(
		|x|^{-N+1-2s}\norm[L^1(B^N_{2|x|})]{v}
	\right)
	+O\left(
		\norm[L^1(B^N_{2\tau}\setminus B^N_{2|x|})]{|\cdot|^{-(N-1+2s)}v}
	\right),
\end{split}\]
where the implicit constants depend only on $n,N,s,\tau$, and for $z\in \R^n\setminus \cT_{3\tau}$,
\[  | \Dsn \bar{v}(z)|
\leq \frac{C_{n,s}\norm[L^1(\cT_{2\tau}\setminus\Sigma)]{v}}{(\mathrm{dist}_{\Sigma}(z))^{n+2s}}.
\]
\end{prop}

\begin{proof}
We denote the dummy variables with a bar. %
We fix coordinates centered at $y\in\Sigma$. Taking $
\rho=\sqrt{|x-\bar{x}|^2+|\bar{y}|^2}
$, we have for $\rho<\tau$,
\begin{equation}\label{eq:jac}
\dfrac{1}{
    |z-\bar{z}|^{n+2s}
}
\,d\bar{z}
=\dfrac{
    1+O(|\bar{x}|)+O(|x-\bar{x}|)+O(|\bar{y}|)
}{
    \rho^{n+2s}
}
\,d\bar{x}\,d\bar{y}
=\dfrac{
    1+O(|x|)+O(\rho)
}{
    \rho^{n+2s}
}
\,d\bar{x}\,d\bar{y}.
\end{equation}
Then, using the facts that
\[
C_{n,s}\int_{\R^{n-N}}
    \dfrac{
        d\bar{y}
    }{
        \left(
            |x-\bar{x}|^2+|\bar{y}|^2
        \right)^{\frac{n+2s}{2}}
    }
=
    \dfrac{
        C_{N,s}
    }{
        |x-\bar{x}|^{N+2s}
    },
\]
\[
C_{n,s}\int_{\R^{n-N}}
    \dfrac{
        d\bar{y}
    }{
        \left(
            |x-\bar{x}|^2+|\bar{y}|^2
        \right)^{\frac{n-1+2s}{2}}
    }
=
\dfrac{
    C(n,N,s)
}{
    |x-\bar{x}|^{N-1+2s}
},
\]
we compute for $z=(x,0)\in\cT_{3\tau}$, i.e. when $|x|=r<3\tau$, %
\[\begin{split}
\Dsn \bar{v}(z)
&=
    C_{n,s}\iint_{\substack{
        |\bar{x}|\leq 4\tau\\
        |\bar{y}|\leq \tau
    }}
        \dfrac{
            v(x)-v(\bar{x})
        }{
            \rho^{n+2s}
        }
        \left(
            1+O(|x|)+O(\rho)
        \right)
    \,d\bar{x}\,d\bar{y}\\
&\quad\;
    +C_{n,s}\iint_{\substack{
        |\bar{x}|\leq 4\tau\\
        \tau<|\bar{y}|\leq \diam \Sigma
    }}
        \dfrac{
            v(x)-v(\bar{x})
        }{
            \rho^{n+2s}
        }
        \left(
            1+O(|x|)+O(\rho)
        \right)
    \,d\bar{x}\,d\bar{y}
    +C_{n,s}\int_{\bar{z}\notin \cT_{4\tau}}
        \dfrac{
            v(z)
        }{
            |z-\bar{z}|^{n+2s}
        }
    \,d\bar{z}\\
&=
    (1+O(|x|))
    C_{N,s}\int_{|\bar{x}|\leq 4\tau}
        \dfrac{
            v(x)-v(\bar{x})
        }{
            |x-\bar{x}|^{N+2s}
        }
    \,d\bar{x}
    +O(1)\int_{|\bar{x}|\leq 4\tau}
        \dfrac{
            |v(x)-v(\bar{x})|
        }{
            |x-\bar{x}|^{N-1+2s}
        }
    \,d\bar{x}
    \\
&\quad\;
    +O(\tau^{-(n+2s)})\int_{|\bar{x}|\leq 4\tau}
        \left(
            |v(x)|+|v(\bar{x})|
        \right)
    \,d\bar{x}
    +O(\tau^{-(n+2s)})|v(x)|\\
&=
    (1+O(|x|))
    C_{N,s}\int_{\R^N}
        \dfrac{
            v(x)-v(\bar{x})
        }{
            |x-\bar{x}|^{N+2s}
        }
    \,d\bar{x}
    +O(1)I
    +O(\tau^{-(n+2s)})
    \left(
    	|v(x)|+\norm[L^1(B_{4\tau})]{v}
    \right).
\end{split}\]
We bound $I$ as follows,
\begin{align*}
I&=\int_{B^N_{|x|/2}(x)}
        \dfrac{
            |v(x)-v(\bar{x})|
        }{
            |x-\bar{x}|^{N-1+2s}
        }
    \,d\bar{x}
   	+\int_{B^N_{2|x|}\setminus B^N_{|x|/2}(x)}
        \dfrac{
            |v(x)-v(\bar{x})|
        }{
            |x-\bar{x}|^{N-1+2s}
        }
    \,d\bar{x}
    +\int_{B^N_{4\tau} \setminus B^N_{2|x|}}
        \dfrac{
            |v(x)-v(\bar{x})|
        }{
            |x-\bar{x}|^{N-1+2s}
        }
    \,d\bar{x}\\
&\lesssim
	|x|^{2-2s}\norm[L^\infty(B^N_{|x|/2}(x))]{Dv}
	+|v(x)|
	+|x|^{-N+1-2s}\norm[L^1(B^N_{2|x|})]{v}
	+\norm[L^1(B^N_{4\tau}\setminus B^N_{2|x|})]{|\cdot|^{-(N-1+2s)}v}.
\end{align*}
Finally, we conclude by showing that if $z\in \R^n\setminus \cT_{3\tau}$, then

 \[  | \Dsn \bar{v}(z)|
\leq \frac{C}{({\rm dist}_{\Sigma}(z))^{n+2s}}\int_{\cT_{2\tau}}\left|
           -v(\bar{z})
   \right| \,d\bar{z}
\leq
    \frac{
        C\norm[L^1(\cT_{2\tau}\setminus\Sigma)]{v}
    }{
        ({\rm dist}_{\Sigma}(z))^{n+2s}
    }.
\qedhere
\]
\end{proof}

\begin{prop}\label{prop:Ints_nN}
Suppose $v\in C^{1,1}_{\loc}(B^N_{3\tau}\setminus \set{0})$ is non-negative and supported on $B^N_{3\tau}\setminus \set{0}$. In the Fermi coordinates \eqref{eq:Fermi-1}, 
consider {the $L^\infty_{\loc}$ function $\bar{v}:\R^n\setminus\Sigma\to\R$ defined by}
\[
\bar{v}(z)
=\begin{cases}
v(x)
    & \textfor z \in \cT_{3\tau}\setminus\Sigma,\\
0
    & \textfor z \in \R^n\setminus \cT_{3\tau}.
\end{cases}
\]
Assume that $\tau>0$ is so small that
\begin{equation*}
C(N,k,s,\Sigma)\tau \leq \frac{\tilde{c}_1}{2}
    \quad \textand \quad
\tau\leq\frac{1}{4}\diam\Sigma.
\end{equation*}
where $C(N,k,s,\Sigma)$ is a universal constant that controls the error in \eqref{eq:jac} and $\tilde{c}_1$ is given in \eqref{eq:c-tilde-1}.
Then there exist universal constants $C>c>0$ (independent of $\tau$) such that for any $z\in\cT_{4\tau}$,
\[\begin{split}
    c\IntsN v(x)
\leq
    \Intsn \bar{v}(z)
&\leq
    C\IntsN v(x)
    +C\tau^{-(n-2s)}
        \norm[L^1(B_{3\tau}\setminus\Sigma)]{v},
\end{split}\]
and for $z\in\R^n \setminus \cT_{4\tau}$,
\[
c\frac{
    \norm[L^1(\cT_{3\tau}\setminus\Sigma)]{v}
}{
    (\mathrm{dist}_{\Sigma}(z))^{n-2s}
}
\leq
    \Intsn \bar{v}(z)
\leq
C\frac{
    \norm[L^1(\cT_{3\tau}\setminus\Sigma)]{v}
}{
    (\mathrm{dist}_{\Sigma}(z))^{n-2s}
}.
\qedhere
\]
\end{prop}

\begin{proof}
We proceed similarly as in \Cref{prop:Dh_nN}. In fact, the decay in $\R^n\setminus\cT_{4\tau}$ follows by exactly the same argument. Suppose $z\in \cT_{4\tau}$. In the same notation (with $s$ replaced by $-s$), using the fact that $n-2s=k+(N-2s)>k$ for $N\geq 2$, we have for $|\bar{x}|\leq 3\tau$,
\[
\int_{|\bar{y}|\leq \tau}
    \dfrac{
        1
    }{
        (|x-\bar{x}|^2+|\bar{y}|^2)^{\frac{n-2s}{2}}
    }
\,d\bar{y}
=
    \dfrac{1}{|x-\bar{x}|^{N-2s}}
    \int_{|\tilde{y}|\leq\frac{\tau}{|x-\bar{x}|}}
        \dfrac{1}{(1+|\tilde{y}|^2)^{\frac{n-2s}{2}}}
    \,d\tilde{y}
\geq
    \dfrac{\tilde{c}_1}{|x-\bar{x}|^{N-2s}},
\]
where the constant $\tilde{c}_1=\tilde{c}_1(N,k,s)>0$ does not depend on $\tau$ and can be taken universally as
\begin{equation}\label{eq:c-tilde-1}
\tilde{c}_1=
    \int_{B_{1/7}^k}
        \dfrac{1}{(1+|\tilde{y}|^2)^{\frac{n-2s}{2}}}
    \,d\tilde{y},
\end{equation}
since $|x-\bar{x}| \leq 7\tau$. Now we compute
\[\begin{split}
\Intsn \bar{v}(z)
&=
    C_{n,-s}\iint_{\substack{
        |\bar{x}|\leq 3\tau\\
        |\bar{y}|\leq \tau
    }}
        \dfrac{
            v(\bar{x})
        }{
            \rho^{n-2s}
        }
        \left(
            1+O(|x|)+O(\rho)
        \right)
    \,d\bar{x}\,d\bar{y}\\
&\quad\;
    +C_{n,-s}\iint_{\substack{
        |\bar{x}|\leq 3\tau\\
        \tau<|\bar{y}|\leq \diam \Sigma
    }}
        \dfrac{
            v(\bar{x})
        }{
            \rho^{n-2s}
        }
        \left(
            1+O(|x|)+O(\rho)
        \right)
    \,d\bar{x}\,d\bar{y}\\
&=
    \int_{
        |\bar{x}|\leq 3\tau
    }
        \dfrac{
            v(\bar{x})
        }{
            |x-\bar{x}|^{N-2s}
        }
        \left[
            C_{N,-s}
            \int_{|\tilde{y}|\leq \frac{\tau}{|x-\bar{x}|}}
            \dfrac{
                d\tilde{y}
            }{
                (1+|\tilde{y}|^2)^{\frac{n-2s}{2}}
            }
            +O(\tau)
        \right]
    \,d\bar{x}
  \\
&\quad\;
    +O(\tau^{-(n-2s)})
    \int_{|\bar{x}|\leq 3\tau}
        v(\bar{x})
    \,d\bar{x}.\\
\end{split}\]
Since the square bracket is positive for small $\tau$, and the second term is non-negative, the proof is complete.
\end{proof}

\subsection{Error estimates}
\begin{lem}\label{lem:Error estimates}
The error in \eqref{eq:E2N2} made by approximating the solution to \eqref{eq:th2} with \eqref{eq:v_eps} satisfies
\[
|\mathscr E_{\eps}(z)|
\lesssim_\tau
    |\log\eps|^{-(N-1)}
    \left(
        \frac{1}{(\mathrm{dist}_{\Sigma}(z))^{N-1}}
            \oneset{ z\in\cT_{3\tau}\setminus \Sigma}
        +\frac{1}{(\mathrm{dist}_{\Sigma}(z))^{n+1}}
            \oneset{ z\in \R^n\setminus\cT_{3\tau}}
    \right).
\]
Equivalently, there exists a constant $\tilde{C}_1=\tilde{C}_1(N,s,k,\Sigma,\tau)>0$ %
such that
\[
|\mathscr E_{\eps}(z)|
\leq
    \tilde{C}_1
    |\log\eps|^{-(N-1)}
    \omega_{N-1;n+1}(z),
\]
where $\omega_{N-1;n+1}$ is defined in \eqref{eq:weight-Yamabe}.
\end{lem}

\begin{proof}
For $\bar{v}_{\eps}$ defined as in \eqref{eq:v_eps}, thanks to \Cref{prop:Dh_nN} and the product rule,
we know that
for any $z=z(r,\omega,y)\in \cT_{3\tau}\setminus\Sigma$,
\[\begin{split}
\sE_{\eps}(z)
&=\Dhn \bar{v}_\eps(r,\omega,y) -\bar{v}_\eps(r,\omega,y)^{\frac{N}{N-1}}\\
&=(1+O(r))\DhN \bar{v}_\eps(r)
	-\bar{v}_\eps(r)^{\frac{N}{N-1}}
	+O(|\bar{v}_\eps(r)|)
	+O\left(
		r		\norm[L^\infty(B^N_{r/2}(re_1))]{D\bar{v}_\eps}
	\right)
\\
&\quad\;
	+O\left(
		r^{-N}
		\norm[L^1(B^N_{2r})]{\bar{v}_\eps}
	\right)
	+O\left(
		\norm[L^1(B^N_{2\tau}\setminus B^N_{2r})]{|\cdot|^{-N}\bar{v}_\eps}
	\right)
\\
&=(1+O(r))
    \left(
        \chi_{\tau}(r) \DhN u_{\eps}(r)
        +u_{\eps}(r) \DhN \chi_{\tau}(r)
        -2\angles{u_\eps,\chi_\tau}(r)
    \right)
    -\left(u_{\eps}(r)\chi_{\tau}(r)\right)^{\frac{N}{N-1}}\\
&\quad\;
	+O(u_\eps(r)\chi_\tau(r))
	+O\left(
		r
		\norm[L^\infty(B^N_{r/2}(re_1))]{\chi_\tau Du_\eps}
	\right)
	+O\left(
		r
		\norm[L^\infty(B^N_{r/2}(re_1))]{u_\eps D\chi_\tau}
	\right)
\\
&\quad\;
	+O\left(
		r^{-N}
		\norm[L^1(B^N_{2r}\cap B^N_{2\tau})]{u_\eps}
	\right)
	+O\left(
		\norm[L^1(B^N_{2\tau}\setminus B^N_{2r})]{|\cdot|^{-N}u_\eps}
	\right),
\end{split}\]
where
\[
\angles{u_\eps,\chi_\tau}(r)
=\frac{C_{N,s}}{2}
\int_{\R^N}
\frac{
    (u_\eps(x)-u_\eps(\bar{x}))
    (\chi_\tau(x)-\chi_\tau(\bar x))
}{
    |x-\bar{x}|^{N+2s}
}
\,d\bar{x}.
\]
For $r\in(0,3\tau)$, we have that $|\DhN \chi_\tau(r)|\lesssim \tau^{-1}$ and, using \Cref{lem:reg},
\[\begin{split}
\abs{\angles{u_\eps,\chi_\tau}(r)}
&\lesssim
\begin{dcases}
    \int_{\frac{\tau}{2}}^{\infty}
        \dfrac{
            (u_\eps(r)+u_\eps(\bar{r}))
            (1-\chi_\tau(\bar{r}))
        }{
            \bar{r}^{N+1}
        }
    \bar{r}^{N-1}
    \,d\bar{r},
            & \textfor r\in(0,\tfrac{\tau}{4}],\\
    \int_{0}^{\tau}
        \dfrac{
            (u_\eps(r)+u_\eps(\bar{r}))
            |0-\chi_\tau(\bar{r})|
        }{
            r^{N+1}
        }
    \bar{r}^{N-1}
    \,d\bar{r},
            & \textfor r\in[2\tau,3\tau),\\
    \int_{\frac{\tau}{8}<|\bar{x}|<4\tau}
        \dfrac{
            \norm[L^\infty(B_{|x-\bar{x}|}(x))]{
                Du_\eps
            }
            \norm[L^\infty(B_{|x-\bar{x}|}(x))]{
                D\chi_\tau
            }
        }{
            |x-\bar{x}|^{N-1}
        }
    \,d\bar{x}\\
    \qquad
    +\tau^{-2}
        |\log\eps|^{-(N-1)}
        r^{-(N-1)}
    +\tau^{-N}
        |\log\eps|^{-(N-1)},
            & \textfor r\in(\tfrac{\tau}{4},2\tau).
\end{dcases}
\\
&\lesssim_\tau
|\log\eps|^{-(N-1)}r^{-(N-1)}.
\end{split}\]
Thus the nonlocal cut-off errors are controlled by
\[
|u_\eps(r)\DhN \chi_\tau(r)|
+2\abs{\angles{u_\eps,\chi_\tau}(r)}
\lesssim_\tau
    |\log\eps|^{-(N-1)}
    r^{-(N-1)},
        \qquad \textfor r\in(0,3\tau).
\]
The local errors are controlled (using \Cref{lem:reg} again) using
\[
u_\eps(r)+r|Du_\eps(r)|
\lesssim |\log\eps|^{-(N-1)} r^{-(N-1)},
	\qquad 
\text{ for } r\in(0,3\tau).
\]
The integral errors are estimated by
\begin{align*}
	r^{-N}
		\norm[L^1(B^N_{2r}\cap B^N_{2\tau})]{u_\eps}
		+\norm[L^1(B^N_{2\tau}\setminus B^N_{2r})]{|\cdot|^{-N}u_\eps}
\lesssim |\log\eps|^{-(N-1)}r^{-(N-1)},
	\qquad \text{ for } r\in(0,3\tau).
\end{align*}
We conclude that
\[\begin{split}
|\sE_\eps(z)|
&\lesssim_\tau
    (1+O(r))
    \left(
        \chi_\tau(r)
        -\chi_\tau(r)^{\frac{N}{N-1}}
    \right)
    |\log\eps|^{-N}r^{-N}
    +|\log\eps|^{-(N-1)}r^{-(N-1)}
    \\
&\lesssim_\tau
    |\log\eps|^{-(N-1)}
    r^{-(N-1)},
\end{split}\]
for $z\in\cT_{3\tau}\setminus\Sigma$. On the other hand, for $z\in\R^n\setminus\cT_{3\tau}$, it is immediate from \Cref{prop:Dh_nN} that
\[\begin{split}
|\sE_\eps(z)|
&\lesssim_\tau
    |\log\eps|^{-(N-1)}
    {\rm dist}_\Sigma(z)^{-(n+1)}.
\qedhere
\end{split}\]
\end{proof}

\subsection{Linear theory}

The error given by \Cref{lem:Error estimates} suggests spaces weighted by subcritical powers (involving no logarithmic corrections), so that the linearized operator \eqref{eq:L2} is globally well approximated by the fractional Laplacian. Indeed, by \Cref{prop:loghomo},
\begin{equation}\label{eq:L2-2}
\sL_\eps
=\Dhn
    -\frac{N}{N-1}
        (\bar{v}_{\eps})^{\frac{1}{N-1}}
=\Dhn
    -(N\kappa_1+o(1))
    \dfrac{
        \chi_\tau(r)^{\frac{1}{N-1}}
    }{
        r
        \log\frac{1}{\eps r}
    }.
\end{equation}
In fact, the spatial cut-off entails that $\sL_\eps=\Dhn$ away from $\Sigma$. %
We expect that its inverse is almost the Riesz potential in $\R^n$, which is indeed the case. 

Since the linearized operator satisfies maximum principle, it suffices to construct suitable barriers. We are interested in functions of the form $\Inthn \omega_{\mu;\nu}$ where $\omega_{\mu;\nu} \in L^p(\R^n)$ for some $p\in(1,n)$. By \Cref{prop:K} (with $N$ replaced by $n$), this is equivalent to $\frac{n}{\nu}<p<\frac{N}{\mu}$. In particular, this is true for the two pairs $(\mu,\nu)=(N-1,n+1),(N-\frac{1}{2};n-\frac{1}{2})$.

\begin{lem}\label{lem:weights-Yamabe}
We have
\[
\omega^{(1)}(z)
:=
    \Inthn \omega_{N-1;n+1}(z)
\asymp_{\tau}
    \omega_{N-2;n-1}(z),
\]
\[
\omega^{(2)}(z)
:=
    \Inthn \omega_{N-\frac{1}{2};n-\frac{1}{2}}(z)
\asymp_\tau
    \omega_{N-\frac32;n-\frac32}(z),
\]
Here the weights $\omega_{\mu;\nu}$ are defined in \eqref{eq:weight-Yamabe}.
\end{lem}

The shift of the exponents $\mu$ and $\nu$ by $1$ is due to the fractional integration of order $1$, %
and the resulting parameter in the $\nu$-slot is at most $n-1$ because of the convolution with the fundamental solution.

\begin{proof}
Let $\mu\in[N-1,N-\frac12]$, $\nu\in[n-\frac12,n+1]$. 
We compute the two terms separately in
\[
\Inthn \omega_{\mu;\nu}(z)
=\Inthn\left(
    d_\Sigma^{-\mu}
    {\bf 1}_{\cT_{3\tau}\setminus\Sigma}
\right)(z)
+\Inthn\left(
    |\cdot|^{-\nu}
    {\bf 1}_{\R^n\setminus\cT_{3\tau}}
\right)(z).
\]
By \Cref{prop:Ints_nN}, \Cref{prop:K} and \Cref{lem:distances}, since $\big\||x|^{-\mu}\big\|_{L^1(B^N_{3\tau})}\leq C(N,\tau)$,
\[\begin{split}
\Inthn\left(
    d_\Sigma^{-\mu}
    {\bf 1}_{\cT_{3\tau}\setminus\Sigma}
\right)(z)
&\asymp_\tau
\begin{cases}
    \InthN\left(
        |x|^{-\mu}
        {\bf 1}_{B^N_{3\tau}}
    \right)
    +O\left(\norm[L^1(B^N_{3\tau})]{|x|^{-\mu}}\right),
         & \textfor
          z\in \cT_{4\tau\setminus\Sigma},\\
    \norm[L^1(B^N_{3\tau})]{|x|^{-\mu}}d_\Sigma(z)^{-(n-1)},
        & \textfor
         z\in \R^n\setminus\cT_{4\tau}
\end{cases}\\
&\asymp_{\tau}
\begin{cases}
    r^{-(\mu-1)}
    +O(1),
        & \textfor z\in \cT_{4\tau\setminus\Sigma},\\
    |z|^{-(n-1)},
        & \textfor z\in \R^n\setminus\cT_{4\tau}.
\end{cases}\\
&\asymp_{\tau}
    \omega_{\mu-1;n-1}(z).
\end{split}\]
Next $
    |z|^{-\nu}
    \oneset{z\in\R^n\setminus\cT_{3\tau}}
\in L^1(\R^n)$ so it has a bounded Riesz potential in $\R^n$, and using \Cref{prop:K} again, we have for  $\nu\in[n-\frac12,n+1]\setminus\set{n}$ and $|z|\gg_{\Sigma,\tau} 1$,
\begin{align*}
\Inthn\left(
    |\cdot|^{-\nu}
    {\bf 1}_{\R^n\setminus\cT_{3\tau}}
\right)(z)
&\asymp_{\tau}
	|z|^{-(n-1)}
    \int_{0}^{\infty}
        K_{n,-\frac12}(\rho)
        (|z|\rho)^{n-\nu}
        \oneset{|z|\rho\geq c(\Sigma,\tau)}
    \,d\rho\\
&\asymp_{\tau}
	|z|^{-(\nu-1)}
\left(
    \int_{\frac{c(\Sigma,\tau)}{|z|}}^{\frac12}
    +\int_{\frac12}^{\infty}
\right)
        K_{n,-\frac12}(\rho)
        \rho^{n-\nu}
    \,d\rho
\\
&\asymp_{\tau,\Sigma,\nu}
	|z|^{-\min(\nu-1,n-1)}
\asymp_{\tau,\Sigma,\nu} \omega_{\mu-1; \min(\nu-1,n-1)}(z).
\qedhere
\end{align*}
\end{proof}

Next we show the $\sL_\eps$-superharmonicity of the above two weights. 

\begin{lem}[Barriers]
\label{lem:barriers-Yamabe}
Let $\sL_\eps$ be as in \eqref{eq:L2-2}, $\omega^{(i)}$ ($i=1,2$) be as in  \Cref{lem:weights-Yamabe}. Then %
\[
\sL_\eps \omega^{(1)}
\geq
    \frac12
    \omega_{N-1;n+1},
	\qquad
\sL_\eps \omega^{(2)}
\geq \frac12 \omega_{N-\frac{1}{2};n-\frac{1}{2}}.
\]
\end{lem}

\begin{proof}
Since $\omega^{(i)}\in L^p(\R^n)$ for some $p\sim 1^+$ ($i=1,2$), $\Inthn$ uniquely inverts $\Dhn$. Thus
\[
\sL_\eps \omega^{(1)}
\geq
    \omega_{N-1;n+1}
    -\dfrac{
        C\chi_\tau(r)^{\frac{1}{N-1}}
    }{
        r\log\frac{1}{\eps r}
    }
    \omega_{N-2;n-1}
\geq
    \frac12
    \omega_{N-1;n+1},
\]
as the coefficient of $\omega_{N-2;n-1}$ is supported in $\cT_{2\tau}$ and is $o(r^{-1})$ as $\eps r\to 0^+$. For $\omega^{(2)}$ it is analogous.
\end{proof}

\begin{lem}[\emph{A priori} estimates]
\label{lem:apriori-yamabe}
Let $g\in \widetilde{\sY}$. If $\psi\in \widetilde{\sX}$ solves
\[
\sL_\eps \psi = g
    \qquad \text{ in } \R^n\setminus\Sigma,
\]
then there exists $\tilde{C}_2=\tilde{C}_2(N,s,k,\Sigma,\tau)>0$ such that
\begin{equation}\label{eq:apriori-Yamabe}
\norm[N-2;n-1]{\psi}
\leq
    \tilde{C}_2
    \norm[N-1;n+1]{g}.
\end{equation}
\end{lem}

\begin{proof}
There exist $\delta_j\to 0^+$ and exhaustive compact subsets $\Omega_j \subset \R^n\setminus\Sigma$ such that by \Cref{lem:barriers-Yamabe},
\[
\sL_\eps\left(
    \tilde{C}_2
    \norm[N-1;n+1]{g}
        \omega^{(1)}
    +\delta_j\omega^{(2)}
    \pm\psi
\right)
\geq 0
	\qquad \text{ in } \Omega_j,
\]
and the function in the bracket is positive (due to $\omega^{(2)}$) on $\partial \Omega_j$.
Thus \Cref{HsMP} applies for each $j$.
\end{proof}

\begin{prop}[Linear theory]
\label{prop:lin-Yamabe}
If $g\in\widetilde{\sY}$, then there exists a unique $\psi\in\widetilde{\sX}$ such that
\[
\sL_\eps \psi=g
    \qquad \text{ in } \R^n\setminus\Sigma.
\]
Moreover, the estimate \eqref{eq:apriori-Yamabe} holds. In particular, $\sL_\eps^{-1}$ maps $\sY$ to $\sX$.
\end{prop}

\begin{proof}
Note that $\Inthn:\widetilde{\sY}\to \widetilde{\sX}$ by \Cref{lem:weights-Yamabe}. By the standard method of continuity using \Cref{lem:apriori-yamabe}, for $\lambda\in[0,1]$ the inverse operator of
\[
\sL_{\eps,\lambda}
:=
    \Dhn
    -\lambda\frac{N}{N-1}
        (\bar{v}_\eps)^{\frac{1}{N-1}}
\]
is well defined from $\widetilde{\sY}$ to $\widetilde{\sX}$,
since the potential is also a continuous linear operator from $\widetilde{\sY}$ to $\widetilde{\sX}$.
\end{proof}

\subsection{The nonlinear fixed point argument}

\begin{prop}[Contraction]
\label{prop:G-Yamabe}
For $\eps\ll 1$, $\sG_\eps$ as given in \eqref{eq:NL2} is a contraction on $\sX$.
\end{prop}

\begin{proof}
By \Cref{lem:Error estimates} and \Cref{prop:lin-Yamabe},
\[
\norm[N-2;n-1]{
    \sL_\eps^{-1}(\sE_\eps)
}
\leq
    \tilde{C}_1
    \tilde{C}_2
    |\log\eps|^{-(N-1)}.
\]
Given $\psi,\tilde{\psi}\in\sX$, we compute
\[\begin{split}
\abs{
    \sN[\psi]-\sN[\tilde{\psi}]
}
&=\abs{
    |\bar{v}_\eps+\psi|^{\frac{n+1}{n-1}}
    -|\bar{v}_\eps+\tilde\psi|^{\frac{n+1}{n-1}}
    -\frac{n+1}{n-1}(\bar{v}_{\eps})^{\frac{2}{n-1}}(\psi-\tilde{\psi})
}\\
&\leq
    \dfrac{2(n+1)}{(n-1)^2}
    \int_{0}^{1}\int_{0}^{1}
        \abs{
            \bar{v}_\eps
            +\tau(1-t)\psi
            -\tau t\tilde{\psi}
        }^{\frac{n+1}{n-1}-2}
    \,dt\,d\tau
    \cdot
    \left(
        |\psi|
        +|\tilde{\psi}|
    \right)
    \abs{
        \psi-\tilde{\psi}
    }\\
&\leq
    C(N,s,k,\Sigma,\tau,\tilde{C}_1,\tilde{C}_2)
    \left(
        |\log\eps|^{-(N-1)}
        \omega_{N-1;n-1}
    \right)^{\frac{n+1}{n-1}-2}\\&
    \hspace{4cm}
    \cdot
    \left(
        |\log\eps|^{-(N-1)}
        \omega_{N-2;n-1}
    \right)^2
    \norm[N-2;n-1]{
        \psi-\tilde{\psi}
    }\\
&\leq
    C(N,s,k,\Sigma,\tau,\tilde{C}_1,\tilde{C}_2)    |\log\eps|^{-N}
    \omega_{N-2;n+1}
    \norm[N-2;n-1]{
        \psi-\tilde{\psi}
    }.
\end{split}\]
Since $\omega_{N-2;n+1} \leq \omega_{N-1;n+1}$, \Cref{prop:lin-Yamabe} yields
\[
\abs{
    \sG_\eps[\psi]
    -\sG_\eps[\tilde{\psi}]
}
\leq
    C(N,s,k,\Sigma,\tau,\tilde{C}_1,\tilde{C}_2)    |\log\eps|^{-N}
    \omega_{N-2;n-1}
    \norm[N-2;n-1]{
        \psi-\tilde{\psi}
    }.
\]
Therefore, %
$\cG_\eps$ is a contraction on $\sX$.
\end{proof}

\subsection{Proof of \Cref{th:Yamabe}}

\begin{proof}
Using \Cref{prop:G-Yamabe}, there exists a unique solution $\psi\in\sX$ of \eqref{eq:NL2}, i.e. a solution $v=\bar{v}_\eps+\psi$ of \eqref{eq:th2} which is positive near $\Sigma$, $\frac12$-superharmonic on $\R^n\setminus\Sigma$ and vanishes at infinity. %
 {Thus,  global non-positive minima do not exist, and so $v>0$.} This concludes the proof.
\end{proof}

\appendix

\section{Functional framework}
\label{app:functional}

\subsection{Standard function spaces}

Let $\Omega\subseteq\R^N$ be a smooth domain. We denote the standard Lebesgue spaces by $L^p(\Omega)$ for $1\leq p\leq \infty$, and those weighted algebraically at infinity by
\[
L^p_{\alpha}(\R^N)
=\set{
    v\in L^p_{\loc}(\R^N):
    \int_{\R^N}
        \dfrac{
            |v(x)|^p
        }{
            \angles{x}^{N+\alpha}
        }
    \,dx
    <+\infty
},
    \qquad \alpha\in \R.
\]
We mainly take $(p,\alpha)=(\frac{N}{N-2s},-2s)$  or $(1,2s)$. To accommodate the Poisson formula in $B_1$, we set %
\[
\tilde{L}^1_{2s}(\R^N\setminus B_1)
=\set{
    v\in L^1_{\loc}(\R^N):
    \int_{\R^N\setminus B_1}
        \dfrac{
            |v(x)|
        }{
            |x|^{N}
            (|x|^2-1)^{s}
        }
    \,dx
    <+\infty
}.
\]
The space $C^\alpha(\Omega)$ contains standard H\"{o}lder continuous functions when $\alpha>0$ is a non-integer and $\alpha$-fold continuously differentiable functions if $\alpha\in\mathbb{N}\cup\set{0}$. We write
\[
C^{2s+}(\Omega)=\bigcup_{\alpha>0}C^{2s+\alpha}(\Omega).
\]
For $s\in(0,1)$, the Sobolev--Slobodeckij space $H^s(\Omega)$ and the Lions--Magenes space $H^s_{00}(\Omega)$ are endowed with the norms
\[
    \norm[H^s(\Omega)]{v}^2
    :=\int_{\Omega}
        v^2
    \,dx
    +\iint_{\Omega\times \Omega}
        \dfrac{
            (v(x)-v(y))^2
        }{
            |x-y|^{N+2s}
        }
    \,dx\,dy
    <+\infty,
\]
\[
    \norm[H_{00}^s(\Omega)]{v}^2
    :=\int_{\Omega}
        v^2
    \,dx
    +\int_{\Omega}
        \dfrac{
            v^2
        }{
            \dist(x,\p\Omega)^{2s}
        }
    \,dx
    +\iint_{\Omega\times \Omega}
        \dfrac{
            (v(x)-v(y))^2
        }{
            |x-y|^{N+2s}
        }
    \,dx\,dy
    <+\infty.
\]
Write $H^s_0(\Omega)$ as the closure of $C_c^\infty(\Omega)$ with respect to the norm $\norm[H^s(\Omega)]{\cdot}$. In a bounded smooth domain we have the inclusions
(see for instance \cite[Section 8.10]{Bhattacharyya-book})
\[\begin{cases}
H^s_{00}(\Omega)=H^s_0(\Omega)=H^s(\Omega)
    & \textfor s\in(0,\frac12),\\
H^{\frac12}_{00}(\Omega)\subsetneq H_0^{\frac12}(\Omega)=H^{\frac12}(\Omega)
    & \textfor s=\frac12,\\
H^s_{00}(\Omega)=H^s_0(\Omega)\subsetneq H^s(\Omega)
    & \textfor s\in(\frac12,1).
\end{cases}\]
The two strict inclusions are exemplified by the constant function. Observe also that $H^s_{00}(\Omega)$ contain precisely those functions whose zero extension lies in $H^s(\R^N)$.

\subsection{Green and Poisson formulae}

Let $\Omega\subset\R^N$. The solution to the Dirichlet problem
\[\begin{dcases}
\Ds u=f
    & \text{ in } \Omega,\\
u=g
    & \text{ in } \R^N\setminus\Omega,
\end{dcases}\]
is given by
\[
u(x)
=
	\cG_{\Omega}[f](x)
	+\cP_{\Omega}[g](x)
:=
    \int_{\Omega}
        \bG_{\Omega}(x,y)
        f(y)
    \,dy
    +\int_{\R^N\setminus\Omega}
        \bP_{\Omega}(x,y)
        g(y)
    \,dy,
\]
where $\bG_\Omega$ and $\bP_\Omega$ are the Green and Poisson kernels in $\Omega$ and $f,g$ are suitable data such that the right hand side is well-defined. For the unit ball and its exterior, these kernels are given explicitly. When $\Omega=B_1$, the Green and Poisson kernels are given respectively in \cite{Bucur}:
\begin{equation}
\label{eq:green_B1}
\bG_{B_1}(x,y)
=\dfrac{
    \Gamma(\frac{N}{2})
}{
    2^{2s}
    \pi^{\frac{N}{2}}
    \Gamma(s)^2
}
    \dfrac{
        1
    }{
        |x-y|^{N-2s}
    }
    \int_{0}^{\frac{(1-|x|^2)(1-|y|^2)}{|x-y|^2}}
        \dfrac{
            \tau^{s-1}
        }{
            (\tau+1)^{\frac{n}{2}}
        }
    \,d\tau,
\qquad
x,y\in B_1,
\end{equation}
\begin{equation*}\label{eq:poisson_B1}
\bP_{B_1}(x,y)
=\frac{
    \Gamma(\frac{N}{2})
    \sin(\pi s)
}{
    \pi^{\frac{N}{2}+1}
}
    \left(
        \dfrac{
            1-|x|^2
        }{
            |y|^2-1
        }
    \right)^s
    \dfrac{1}{
        |x-y|^N
    },
\qquad
    x\in B_1,\,y\in \R^N\setminus B_1.
\end{equation*}
An immediately pointwise upper bound for the Poisson integral is given by
\begin{equation}
\label{eq:Poisson-bound}
|\cP_{B_1}[g](x)|
\leq C(1-|x|)^{-(N-s)}
	\norm[\tilde{L}^1_{2s}(\R^N\setminus B_1)]{g},
		\qquad x\in B_1.
\end{equation}
When $\Omega=B_1^c=\R^N\setminus B_1$, one obtains the corresponding kernels by taking the Kelvin transformation \cite{ALW}. In fact,
\[
\bG_{B_1^c}(x,y)
=\dfrac{
    \Gamma(\frac{N}{2})
}{
    2^{2s}
    \pi^{\frac{N}{2}}
    \Gamma(s)^2
}
    \dfrac{
        1
    }{
        |x-y|^{N-2s}
    }
    \int_{0}^{\frac{(|x|^2-1)(|y|^2-1)}{|x-y|^2}}
        \dfrac{
            \tau^{s-1}
        }{
            (\tau+1)^{\frac{n}{2}}
        }
    \,d\tau,
\quad
x,y\in \R^N\setminus B_1,
\]
\[
\bP_{B_1^c}(x,y)
=\frac{
    \Gamma(\frac{N}{2})
    \sin(\pi s)
}{
    \pi^{\frac{N}{2}+1}
}
\left(\frac{|x|^2-1}{1-|y|^2}\right)^s
\frac{1}{|x-y|^{N}},
    \qquad
x\in \R^N\setminus B_1,\, y\in B_1.
\]
Note that the Green kernel in a general smooth bounded domain satisfies the two-sided estimate %
\[
\bG_{\Omega}(x,y)
\asymp
    \dfrac{1}{|x-y|^{N-2s}}
    \min\left(
        1,
        \dfrac{
            {\rm dist}(x,\p\Omega)
            {\rm dist}(y,\p\Omega)
        }{
            |x-y|^2
        }
    \right)^s.
\]

\subsection{Notions of solution}

We discuss several notions of solution for the fractional Lane--Emden--Serrin equation
\begin{equation}\label{eq:main-notion}
\Ds u=u^{\frac{N}{N-2s}}
    \qquad \text{ in } B_1\setminus\set{0}.
\end{equation}

\begin{defn}\label{defn:sol}
We say that:
\begin{enumerate}
\item $u\in C^{2s+}(B_1\setminus\set{0})\cap L^1_{2s}(\R^N)$ is a \emph{classical solution} if \eqref{eq:main-notion} is satisfied everywhere in $B_1\setminus\set{0}$.
\item %
    $u\in H^s(\R^N)$ is a \emph{variational solution} (or \emph{weak solution}) of \eqref{eq:main-notion} if
    \[
    \dfrac{C_{N,s}}{2}
    \iint_{(\R^N\times \R^N)\setminus (B_1^c \times B_1^c)}
        \dfrac{
            (u(x)-u(y))
            (\zeta(x)-\zeta(y))
        }{
            |x-y|^{n+2s}
        }
    \,dx\,dy
    =\int_{B_1}
        u^{\frac{N}{N-2s}}
        \zeta
    \,dx,
    \]
    $\forall\zeta\in H^s_{00}(\R^N)$.
\item $u\in L^{\frac{N}{N-2s}}(B_1)\cap L^1_{2s}(\R^N)$ is a \emph{distributional solution} (or \emph{very weak solution}) of \eqref{eq:main-notion} if
    \[
    \int_{\R^N}
        u\Ds \zeta
    \,dx
    =\int_{B_1}
        u^{\frac{N}{N-2s}}
        \zeta
    \,dx,
        \qquad \forall\zeta\in C^{2s+}(B_1),\, \zeta|_{\R^N\setminus B_1}\equiv 0.
    \]
\item $u\in L^{\frac{N}{N-2s}}(B_1)\cap \tilde{L}^1_{2s}(\R^N \setminus B_1)$ is a \emph{weak-dual solution} if
    \[
    \int_{B_1}
        u\psi
    \,dx
    +\int_{\R^N\setminus B_1}
        u\Ds\cG[\psi]
    \,dx
    =\int_{B_1}
        u^{\frac{N}{N-2s}}
        \cG[\psi]
    \,dx,
        \qquad \forall \psi\in L^\infty(B_1).
    \]
\item $u\in L^{\frac{N}{N-2s}}(B_1) \cap \tilde{L}^1_{2s}(\R^N \setminus B_1)$ is a \emph{Green--Poisson solution} if
    \[
    u=\cG_{B_1}[u^{\frac{N}{N-2s}}]
        +\cP_{B_1}[u]
            \qquad \textae \text{ in } B_1,
    \]
    {where $\cG_{B_1}$ and $\cP_{B_1}$ are the Green and Poisson operators defined above.}
\end{enumerate}
\end{defn}

It is not hard to see that these definitions are equivalent: whenever two of them make sense for a solution simultaneously, one definition implies the other. Indeed, going down from (2) to (4) one simply enlarges the space of test functions and the reverse direction holds in view of the density of one function space in another. %

\begin{remark}
Some remarks concerning \Cref{defn:sol} are in order.
\begin{itemize}
\item In (1), the regularity and decay are the minimal requirement for the pointwise evaluation of the fractional Laplacian.
\item In (2), that the solution is satisfied only (weakly) in $B_1$ is seen from the fact that $\zeta|_{\R^N\setminus B_1}\equiv 0$. The left hand side is finite by the Cauchy--Schwarz inequality, while the right hand side is finite since the Serrin exponent is Sobolev-subcritical. Moreover, it is commonly written in the literature (such as \cite{FKV2015, SV2013}) that the test function lies in $H^s(\R^N)$ and is compactly supported in $\Omega$. But these functions are precisely those in the Lions--Magenes space \cite{Lions-Magenes}.
\item In (3), a straightforward computation reveals that $\Ds \zeta \lesssim \angles{x}^{-N-2s}$ for $\zeta\in C_c^{2s+}(\Omega)$. Hence the left hand side is finite.
\item In (4), the integrability of the second term on the left hand side near the unit sphere follows from either the explicit Poisson kernel or the estimate in the general domain \cite[Equation (36)]{Abatangelo2015}, $|\Ds \bG_{B_1}(x,y)|\lesssim \dist(y,\p\Omega)^{-s}$ for $x\in B_1$, $y\in \R^N\setminus B_1$.
\end{itemize}

\end{remark}

\section{Some explicit formulae for radial functions}\label{sec:computations}

\color{red}

\normalcolor

\begin{lem}
The formula \eqref{eq:K-2F1} holds.
\end{lem}

\begin{proof}
Under the change of variable $\sigma=\sin^2\frac{\theta}{2}$,
\begin{align*}
K_{N,s}(\rho)
&=C_{N,s}\cH^{N-2}(\bS^{N-2})\rho^{(2s)_+-1}
\int_{0}^{\pi}
	\frac{
		\sin^{N-2}\theta
	}{
		\bigl(
			(\rho-1)^2+2\rho(1-\cos\theta)
		\bigr)^{\frac{N+2s}{2}}
	}
\,d\theta\\
&=C_{N,s}\cH^{N-2}(\bS^{N-2})
\frac{
	\rho^{(2s)_+-1}
}{
	|\rho-1|^{N+2s}
}
\int_{0}^{\pi}
	\frac{
		2^{N-2}
		\sin^{N-3}\frac{\theta}{2}
		\cos^{N-3}\frac{\theta}{2}
		\cdot
		(\sin\frac{\theta}{2}\cos\frac{\theta}{2}\,d\theta)
	}{
		\bigl(
			1+\frac{4\rho}{|\rho-1|^2}
			\sin^2\frac{\theta}{2}
		\bigr)^{\frac{N+2s}{2}}
	}\\
&=2^{N-2}C_{N,s}\cH^{N-2}(\bS^{N-2})
\frac{
	\rho^{(2s)_+-1}
}{
	|\rho-1|^{N+2s}
}
\int_{0}^{1}
	\sigma^{\frac{N-3}{2}}
	(1-\sigma)^{\frac{N-3}{2}}
	\bigl(
		1+\tfrac{4\rho}{|\rho-1|^2}\sigma
	\bigr)^{-\frac{N+2s}{2}}
\,d\sigma.
\end{align*}
Identifying the last integral using \eqref{prop7} and simplifying the constants, \eqref{eq:K-2F1} follows.
\end{proof}

\begin{cor}
The asymptotic expansions \eqref{eq:K-asymp-1}--\eqref{eq:K-asymp-2} hold. 
\end{cor}

\begin{proof}
See \Cref{hypergeo}. For the case $s=-1/2$, $\rho \to 1$ with $N$ odd, see the remark in \cite[15.3.14]{Abramowitz}.
\end{proof}

\normalcolor

\begin{lem}[Vanishing of the zeroth order term]
\label{cor:K-int-0}
	Let $s\in(0,1)$. For any $f\in C_c^2([0,1))$, %
	\[
	\int_{0}^{1}
		K_{N,s}(\rho)f(\rho)
	\,d\rho
	=\int_{1}^{\infty}
		K_{N,s}(\rho)f(\rho^{-1})\rho^{N-2s}
	\,d\rho.
	\]
	In particular, by taking $f(\rho)=\rho^{N-2s}-1$, one obtains that
	\[
	\PV\int_{0}^{\infty}
		K_{N,s}(\rho)(\rho^{N-2s}-1)
	\,d\rho=0.
	\]
\end{lem}

\begin{proof}
Use the symmetry
$
K_{N,s}\left(\frac{1}{\rho}\right)=\rho^{N+2-2s}K_{N,s}(\rho).
$
\end{proof}

\begin{lem}[Conformal fractional Laplacian in polar coordinates]
\label{lem:emden_change}
If $v(r)$ is a radial function, then at each point where $v$ is $C^{2}$, there holds
\[
r^N\Ds\left(\dfrac{v(r)}{r^{N-2s}}\right)
=
	\PV\int_{0}^{\infty}
		K_{N,s}(\rho)(v(r)-v(r\rho))
	\,d\rho.
\]
The principal value is not needed when $s\in(0,\frac12)$.
\end{lem}

\begin{remark}\label{rmk:PV}
Here $\PV$ denotes the Cauchy principal value in the sense
\[\begin{split}
&\quad\;
\PV\int_{0}^{\infty}
		K_{N,s}(\rho)(v(r)-v(r\rho))
	\,d\rho
	\\
&=\lim_{\eps\searrow0}
\int_{1+\eps}^{\infty}
	\left(
		K_{N,s}(\rho)(v(r)-v(r\rho))
		+\dfrac{1}{\rho^2}
            K_{N,s}\left(\dfrac1\rho\right)
		\left(
			v(r)-v\left(\dfrac{r}{\rho}\right)
		\right)
	\right)
\,d\rho\\
&=\lim_{\eps\searrow0}
\int_{1+\eps}^{\infty}
	K_{N,s}(\rho)
	\left(
		(v(r)-v(r\rho))
		-\rho^{N-2s}
		\left(
			v\left(\dfrac{r}{\rho}\right)-v(r)
		\right)
	\right)
\,d\rho.
\end{split}\]
In particular, when $s\in[\frac12,1)$ and $v\in C^{2}$ on some $[r_-,r_+]\ni r$,
using Taylor expansion close to $\rho=1$, we have
\[\begin{split}
(v(r)-v(r\rho))
	-\rho^{N-2s}
	\left(
		v\left(\dfrac{r}{\rho}\right)-v(r)
	\right)
&=rv'(r)(\rho-1)(\rho^{N-1-2s}-1)
+O\left(
        \seminorm[C^{0,1}({[r_-,r_+]})]{v'}
        r^{2}
        |\rho-1|^{2}
    \right),
\end{split}
\]
and hence the singular integral is finite. %
Moreover, an upper bound is given by
\begin{equation*}\begin{split}
\abs{\PV\int_{0}^{\infty}
    K_{N,s}(\rho)(v(r)-v(r\rho))
\,d\rho}
&\lesssim
    v(r)
    +rv'(r)
    +r^{2}
    \seminorm[{C^{0,1}([\frac{r}{2},2r])}]{v'}
+\int_{0}^{\frac12}
        v(r\rho)
    \,d\rho
    +\int_{2}^{\infty}
        \dfrac{
            v(r\rho)
        }{
            \rho^{N+2s}
        }
    \,d\rho.
\end{split}\end{equation*}
\end{remark}

\begin{proof}[Proof of \Cref{lem:emden_change}]
In the radial variables %
$r=|x|$ and $\bar{r}=|y|$ (where $y$ is the dummy variable in $\R^N$),
\[\begin{split}
&\quad\;
r^N\Ds\left(\dfrac{v(r)}{r^{N-2s}}\right)
\\
&=C_{N,s}
	r^N
	\PV\int_{0}^{\infty}
	\int_{0}^{\pi}
		\dfrac{
			r^{-(N-2s)}v(r)-(\bar{r})^{-(N-2s)}v(\bar{r})
		}{
			(r^2+\bar{r}^2-2r\bar{r}\cos\theta
			)^{\frac{N+2s}{2}}
		}
	\cH^{N-2}(\bS^{N-2})\sin^{N-2}\theta\,d\theta
	\,\bar{r}^{N-1}\,d\bar{r}\\
&=C_{N,s}
	\PV\int_{0}^{\infty}
	\int_{0}^{\pi}
		\dfrac{
			(\frac{\bar{r}}{r})^{N-2s}v(r)-v(\bar{r})
		}{
			\left(
				1+(\frac{\bar{r}}{r})^2
				-2\frac{\bar{r}}{r}\cos\theta
			\right)^{\frac{N+2s}{2}}
		}
	\cH^{N-2}(\bS^{N-2})\sin^{N-2}\theta\,d\theta
\left(\frac{\bar{r}}{r}\right)^{2s}
	\dfrac{d\bar{r}}{\bar{r}}\\
&=
	\PV\int_{0}^{\infty}
		K_{N,s}(\rho)\left(
			\rho^{N-2s}v(r)-v(r\rho)
		\right)
	\,d\rho,
\end{split}\]
where $\bar{r}=r\rho$. Using \Cref{cor:K-int-0}, the result follows.
\end{proof}

\begin{lem}[Riesz potential in polar coordinates]
Suppose $|\cdot|^{-N}f \in L^p(\R^N)$ for some $p\in(1,\frac{N}{2s})$, and
\[
\Ds\left(
    \dfrac{v(r)}{r^{N-2s}}
\right)
=\dfrac{f(r)}{r^N}
    \qquad \textae \text{ in } \R^N.
\]
Then $v$ is uniquely given by the Riesz potential
\[
v(r)
=r^{N-2s}
    \Ints\left(
        \dfrac{f(r)}{r^N}
    \right)
=\int_{0}^{\infty}
    K_{N,-s}(\rho)
    f(r\rho)
\,d\rho.
\]
\end{lem}

\begin{proof}
Since $|\cdot|^{-N}f\in L^p$ with $p>1$, the uniqueness is ensured by \cite[Corollary 1.4]{Fall}. Then we compute
\begin{align*}
r^{N-2s}\Ints\left(\dfrac{f(r)}{r^N}\right)
&=C_{N,-s}
	r^{N-2s}
	\int_{0}^{\infty}
	\int_{0}^{\pi}
		\dfrac{
            \tilde{\rho}^{-N}
            f(\tilde{\rho})
		}{
			(r^2+\tilde{\rho}^2-2r\tilde{\rho}\cos\theta
			)^{\frac{N-2s}{2}}
		}
	\cH^{N-2}(\bS^{N-2})\sin^{N-2}\theta\,d\theta
	\,\tilde{\rho}^{N-1}\,d\tilde{\rho}\\
&=C_{N,-s}
	\int_{0}^{\infty}
	\int_{0}^{\pi}
		\dfrac{
            f(r\rho)
		}{
			\left(
				1+\rho^2
				-2\rho\cos\theta
			\right)^{\frac{N-2s}{2}}
		}
	\cH^{N-2}(\bS^{N-2})\sin^{N-2}\theta\,d\theta
	\,\dfrac{d\rho}{\rho}.
\qedhere
\end{align*}
\end{proof}

\section{Proof of \Cref{prop:loghomo} and \Cref{prop:loghomo-infty}}
\label{sec:proof-loghomo}

\begin{proof}[Proof of \Cref{prop:loghomo}]
Thanks to \Cref{prop:K} we have
\[\begin{split}
	I:=r^N\Ds(\phi_{N-2s,\nu}(r;\eps))
	&=\PV\int_0^\infty
	K_{N,s}(\rho)\left(
	\chi( r)
	\left(\log\frac{1}{\eps r}\right)^{-\nu}
	-\chi(  r\rho)
	\left(\log\frac1{\eps r\rho}\right)^{-\nu}
	\right)
	\,d\rho,
\end{split}\]
\[\begin{split}
	I^{1}:=r^N\Ds(\phi_{N-2s,\nu}^1(r;\eps))
	&=\PV\int_0^\infty
	K_{N,s}(\rho)\left(
	\chi(  r)
	\frac{\log\log\frac{1}{\eps r}}{\left(\log\frac{1}{\eps r}\right)^{\nu}}
	-\chi(  r\rho)
	\frac{\log\log\frac{1}{\eps \rho r}}{\left(\log\frac{1}{\eps \rho r}\right)^{\nu}}
	\right)
	\,d\rho.
\end{split}\]
For both integrals, the main contribution is logarithmic and comes from the interval $[\frac{1}{\sqrt{\eps r}},\sqrt{\eps r}]$, since the tails are polynomially small in view of \Cref{prop:K}. An extra splitting is necessary because of the cut-off function. We proceed to the computation of $I$.

\medskip

\textbf{Case 1: $r\in(0, \frac{1}{ 8}]$.}
Here $\chi(  r)=1$ and $\frac{1}{4 r}\geq 2$. Moreover, since $\chi(  r\rho)=1$ for $\rho \leq \frac{1}{4r}$ and $\chi( r\rho)=0$ for $\rho\geq \frac{1}{ 2r}$, we have
$I=I_1+I_2+I_3+I_4$ where
\[\begin{split}
	|I_1|
	&=\abs{\int_0^{\sqrt{\eps r}}
		K_{N,s}(\rho)\left(
		\left(\log\frac{1}{\eps r}\right)^{-\nu}
		-\left(\log\frac1{\eps r\rho}\right)^{-\nu}
		\right)
		\,d\rho
	}
	\lesssim
	(\eps r)^s
	\left(\log\frac{1}{\eps r}\right)^{-\nu}\\
	I_2
	&=\PV\int_{\sqrt{\eps r}}^{\frac{1}{4r}}
	K_{N,s}(\rho)\left(
	\left(\log\frac{1}{\eps r}\right)^{-\nu}
	-\left(\log\frac{1}{\eps r}+\log\frac1\rho\right)^{-\nu}
	\right)
	\,d\rho\\
	|I_3|
	&=\abs{\int_{\frac{1}{4 r}}^{\frac{1}{2 r}}
		K_{N,s}(\rho)\left(
		\left(\log\frac{1}{\eps r}\right)^{-\nu}
		-\chi(r\rho)
		\left(\log\frac1{\eps r\rho}\right)^{-\nu}
		\right)
		\,d\rho
	}
	\lesssim
	r^N\left(\log\frac{1}{\eps r}\right)^{-\nu}\\
	|I_4|
	&=
	\left(\log\frac{1}{\eps r}\right)^{-\nu}
	\int_{\frac{1}{2 r}}^\infty
	K_{N,s}(\rho)
	\,d\rho
	\lesssim
	r^N\left(\log\frac{1}{\eps r}\right)^{-\nu}.
\end{split}\]
Then, we only need to compute $I_2$. There are two possibilities:
\begin{itemize}
	\item If $r\geq \frac{\eps}{16}$ we have $I_2=\PV\int_{\sqrt{\eps r}}^{\frac{1}{\sqrt{\eps r}}}-\int_{\frac{1}{4 r}}^{\frac{1}{\sqrt{\eps r}}}$, where the second term is bounded as $I_3$ above.
	\item If $r< \frac{\eps}{16}$ we have $I_2=\PV\int_{\sqrt{\eps r}}^{\frac{1}{\sqrt{\eps r}}}+\int_{\frac{1}{\sqrt{\eps r}}}^{\frac{1}{4 r}}$,
	and the second term is bounded by:
	$$
	\abs{
		\int_{\frac{1}{\sqrt{\eps r}}}^{\frac{1}{4 r}}\rho^{-(N+1)}\left(
		\left(\log\frac{1}{\eps r}\right)^{-\nu}
		-\left(\log\frac{1}{\eps r}+\log\frac1\rho\right)^{-\nu}
		\right)
		\,d\rho
	}
	\leq C(\eps r)^{N/2}  \left(\log\frac{1}{\eps r}\right)^{-\nu}.$$
\end{itemize}
Thus, in both cases, it remains to compute the first term.
Since $\sqrt{ \eps r}<\rho<1/\sqrt{\eps r}$ implies $-\frac12 \log\frac{1}{\eps r}<\log\frac1\rho<\frac12\log\frac{1}{\eps r}$, one may use the binomial expansion. Up to a power error (of $\eps r$) that follows from the algebraic decay of $K_{N,s}$ (\Cref{prop:K}), we have
\[\begin{split}
	I_2
	&=\left(\log\frac{1}{\eps r}\right)^{-\nu}
	\int_{\sqrt{\eps r}}^{\frac{1}{\sqrt{\eps r}}}
	K_{N,s}(\rho)\left(
	1-\left(1+
	\dfrac{
		\log\frac1\rho
	}{
		\log\frac{1}{\eps r}
	}
	\right)^{-\nu}
	\right)
	\,d\rho\\
	&=\left(\log\frac{1}{\eps r}\right)^{-\nu}
	\int_{\sqrt{\eps r}}^{\frac{1}{\sqrt{\eps r}}}
	K_{N,s}(\rho)\left(
	\nu
	\dfrac{
		\log\frac1\rho
	}{
		\log\frac{1}{\eps r}
	}- \frac{\nu(\nu+1)}{2}\dfrac{
		(\log\frac1\rho)^2
	}{
		(\log\frac{1}{\eps r})^2
	}
	+O\left(
	\dfrac{
		(\log\frac1\rho)^3
	}{
		(\log\frac{1}{\eps r})^3
	}
	\right)
	\right)
	\,d\rho\\
	&=\nu\kappa_1(s) \left(\log\frac{1}{\eps r}\right)^{-\nu-1}
	-\frac{\nu(\nu+1)}{2}\kappa_2(s)\left(\log\frac{1}{\eps r}\right)^{-\nu-2}+O\left(
	\left(\log\frac{1}{\eps r}\right)^{-\nu-3}
	\right).
\end{split}\]

\textbf{Case 2: $r\in(\frac{1}{ 8},1]$.}
Here we have that %
$w(r)=\chi(r)(\log\frac{1}{\eps r})^{-\nu}$
is smooth, with
\[\begin{split}
	&\quad\;
	w(r)+|rw'(r)|+
	r^2\sup_{[\frac{r}{2},2r]}|w''|
	+\int_{0}^{\frac14}
	w(r\rho)
	\,d\rho
	+\int_{1}^{\infty}
	\dfrac{w(r\rho)}{\rho^{N+1}}
	\,d\rho
	\leq C\left(\log\frac{1}{\eps}\right)^{-\nu}.
\end{split}\]
Then, by \eqref{eq:PV-upper} for $s\in[\frac12,1)$ or a similar consideration for $s\in(0,\frac12)$, %
there exists $C>0$ such that
\[
|I|\leq C\left(\log\frac{1}{\eps}\right)^{-\nu}.
\]

\textbf{Case 3: $r\in(1,\infty)$.} Here $\chi( r)=0$, and $\chi( r\rho)>0$ only when $\rho<\frac{1}{2 r}\leq\frac12$. We have
\[\begin{split}
	I&=-\int_0^{\frac{1}{2r}}
	K_{N,s}(\rho)
	\chi( r\rho)
	\left(\log\frac1{\eps r\rho}\right)^{-\nu}
	\,d\rho\\
\end{split}\]
Then we can estimate, using the change of variable $\tilde{\rho}= r\rho$,
\[\begin{split}
	|I|
	&\leq
	C\int_0^{\frac{1}{2 r}}
	\rho^{2s-1}
	\chi(r\rho)
	\left(\log\frac1{\eps r\rho}\right)^{-\nu}
	\,d\rho\leq
	\dfrac{C}{r^{2s}}
	\int_0^{\frac{1}{2}}
	\tilde{\rho}^{2s-1}
	\left(\log\frac{1}{\eps\tilde{\rho}}\right)^{-\nu}
	\,d\tilde{\rho}
	\leq \frac{C}{r^{2s}}
	\left(
	\log\frac{1}{\eps}
	\right)^{-\nu}.
\end{split}\]
This completes the computations for $I$.

The integral $I^1$ is treated in a similar way. Again, the dominant contribution comes from the interval where the cut-off function equals to $1$ (or can be replaced by $1$) and where the binomial expansion \eqref{eq:binom-loglog} is valid.
Therefore, for $r\in (0,\frac18]$, up to a power error (of $\eps r$) we have
\[\begin{split}
	I^{1}
	&=\PV\int_{\sqrt{\eps r}}^{\frac{1}{\sqrt{\eps r}}}K_{N,s}(\rho)
	\Bigg[
	\nu
	\dfrac{\left(\log\log\frac{1}{\eps r}\right)}{(\log\frac{1}{\eps r})^{\nu}}
	\frac{\log\frac{1}{\rho}}{\log{\frac{1}{\eps r}}}
	-
	\dfrac{1}{(\log\frac{1}{\eps r})^{\nu}}
	\frac{\log\frac{1}{\rho}}{\log{\frac{1}{\eps r}}}
	+O\left(
	\frac{
		\left(
		\log\log\frac{1}{\eps r}
		+1
		\right)
		(\log\frac{1}{\rho})^2
	}{
		(\log\frac{1}{\eps r})^2
	}
	\right)
	\Bigg]
	\,d\rho\\
	&= \kappa_1(s)\dfrac{ \nu\log\log\frac{1}{\eps r}-1}{(\log\frac{1}{\eps r})^{\nu+1}} +O\left(\dfrac{\log\log\frac{1}{\eps r}}{(\log\frac{1}{\eps r})^{\nu+2}}\right).
\end{split}\]
The remaining estimates follow in the same way. 
\end{proof}

\begin{proof}[Proof of \Cref{prop:loghomo-infty}]
If $r\leq 4$, we bound $|r^N\Ds\tilde{\phi}_{N-2s}^{\nu}|$ by a constant, %
 using \eqref{eq:K-asymp-1} and \Cref{rmk:PV}. If $r>4$, by a similar splitting to the one in the proof above, the main contribution comes from $[\frac{1}{\sqrt{r}},\sqrt{r}]$:
\[\begin{split}
r^N\Ds\tilde{\phi}_{N-2s}^{\nu}&=
	(\log r)^{\nu}
	\int_{\frac{1}{\sqrt{r}}}^{\sqrt{r}}
	K_{N,s}(\rho)
	\left(
		1-\left(
			1+\frac{\log\rho}{\log r}
		\right)^{\nu}
		\right)
	\,d\rho
	+O(r^{-s}(\log r)^\nu)\\
	&=(\log r)^{\nu}
	\int_{\frac{1}{\sqrt{r}}}^{\sqrt{r}}
	K_{N,s}(\rho)
	\left(
	-\nu\frac{\log\rho}{\log r}
	+O\left(
	\frac{
		(\log\rho)^2
	}{
		(\log r)^2
	}
	\right)
	\right)
	\,d\rho
	+O(r^{-s}(\log r)^\nu)\\
	&=\nu\kappa_1 (\log r)^{\nu-1}
	+O\left(
	(\log r)^{\nu-2}
	\right).
\qedhere
\end{split}\]
\end{proof}

\section{Maximum principle}

Let $P=\Ds-V(x)$ be a fractional Schr\"{o}dinger operator defined on a bounded domain $\Omega\subset \R^N$. %
We assume that
\[
0\leq V(x) \leq \dfrac{\eta}{|x|^{2s}},
    \quad %
\eta<\Lambda_{N,s}=2^{2s}\dfrac{
    \Gamma\bigl(\frac{n+2s}{4}\bigr)^2
}{
    \Gamma\bigl(\frac{n-2s}{4}\bigr)^2
}.
\]
Recall that $\Lambda_{N,s}$ is the optimal constant for the fractional Hardy inequality. In particular, $P$ is a positive operator in the sense that it has a positive first Dirichlet eigenvalue $\lambda_1=\lambda_1(P)>0$,
\begin{equation*}%
	\int v Pv  \,dx \geq \lambda_1 	\int v^2 \,dx \geq 0,
	\qquad \forall v \in H^s(\Omega),\, v\equiv 0 \text{ in } \R^N\setminus \Omega.
\end{equation*}

\begin{lem}[Maximum principle for variational solutions]
\label{HsMP}%
Let $P$ be as above. Suppose $u\in H^s(\Omega)$ solve
\begin{equation*}%
\begin{cases}
	Pu\geq 0& \text{ in } \Omega,\\
	u\geq 0 & \text{ in } \R^N \setminus \Omega,
\end{cases}\end{equation*}
then $u\geq 0$ in $\Omega$.
\end{lem}

\begin{proof}
We split $u=u_+-u_-$. Observe that $u_-=0$ in $\R^N\setminus\Omega$. By testing the equation against $u_-$,  %
\[
0\leq \int_{\Omega} u_- Pu  \,dx= \int_{\Omega}  u_- (Pu_+-Pu_-)  \,dx\leq -\lambda_1	\int_{\Omega}  u_-^2 \,dx \leq 0,
\]
since $\lambda_1>0$. Thus $u_-\equiv 0$ a.e., so $u\geq 0$ in $\Omega$.
\end{proof}

\section{Some known results on special functions}
\label{Ap:hyper}

\begin{lem}[\!\!\cite{Abramowitz,SlavyanovWolfganglay}]

Let $z\in\mathbb C$. The Gaussian hypergeometric function is defined by the power series
\begin{equation}
\label{hypergeo}
\Hyperg(a,b;c;z) = \sum_{n=0}^\infty \frac{(a)_n (b)_n}{(c)_n} \frac{z^n}{n!}=\frac{\Gamma(c)}{\Gamma(a)\Gamma(b)}\sum_{n=0}^\infty \frac{\Gamma(a+n)\Gamma(b+n)}{\Gamma(c+n)} \frac{z^n}{n!},
	\quad
\text{ for } |z|<1.
\end{equation}
It is undefined (or infinite) if $c$ equals a non-positive integer. It satisfies the following properties.
\begin{itemize}
  \item[i.] When evaluated at $z=0$, we have
  \begin{equation*}%
  \Hyperg(a+j,b-j;c;0)=1; \  j=\pm1,\pm2,...
  \end{equation*}
     \item[ii.] If $|\arg(1-z)|<\pi$, then
\begin{equation*}%
  \begin{split}
  \Hyperg(a,b;c;z)
  &=
                     \frac{\Gamma(c)\Gamma(c-a-b)}{\Gamma(c-a)\Gamma(c-b)}
                     \Hyperg\left(a,b;a+b-c+1;1-z\right)
                      \\
                     &\quad
                     +(1-z)^{c-a-b}\frac{\Gamma(c)\Gamma(a+b-c)}
                     {\Gamma(a)\Gamma(b)}\Hyperg(c-a,c-b;c-a-b+1;1-z).
     \end{split} \end{equation*}
\item[iii.] If $|\arg(-z)|<\pi$, then
\begin{equation*}%
\begin{split}
\Hyperg(a,b;c;z)
&=
\frac{
	\Gamma(c)\Gamma(b-a)
}{
	\Gamma(b)\Gamma(c-a)
}
(-z)^{-a}
\Hyperg(a,a-c+1;a-b+1;z^{-1})\\
&\quad
+\frac{
	\Gamma(c)\Gamma(a-b)
}{
	\Gamma(a)\Gamma(c-b)
}
(-z)^{-b}
\Hyperg(b,b-c+1;b-a+1;z^{-1}).
\end{split}
\end{equation*}
\item[iv.] If $|\arg(-z)|<\pi$, $|z|>1$, $c-a\notin \mathbb{Z}$, then
\begin{equation*}
\begin{split}
&\Hyperg(a,a;c;z)\\
&\quad=\frac{
	\Gamma(c)
}{
	\Gamma(a)\Gamma(c-a)
}
(-z)^{-a}
\sum_{m=0}^{\infty}
z^{-m}
[\log(-z)
+2\psi(m+1)
-\psi(a+m)
-\psi(c-a-m)].
\end{split}\end{equation*}
\item[v.] It is symmetric with respect to first and second arguments, i.e.,
\begin{equation*}%
  \Hyperg(a,b;c;z)= \Hyperg(b,a;c;z).
  \end{equation*}
\item[vi.] Let $m\in \N$. Its $m$-derivative is given by
\begin{equation*}%
 \tfrac{d^m}{dz^m} \left[(1-z)^{a+m-1} \Hyperg(a,b;c;z)\right]= \tfrac{(-1)^m(a)_m(c-b)_m}{(c)_m} (1-z)^{a-1}\Hyperg(a+m,b;c+m;z).
  \end{equation*}

 \item[vii.] If $c > b > 0$, then by using its meromorphic extension, we have for $|z|<1$
\begin{equation}\label{prop7}
\Hyperg(a,b;c;z)
=\dfrac{\Gamma(c)}{\Gamma(b)\Gamma(c-b)}
\int_0^1 t^{b-1}(1-t)^{c-b-1}(1-tz)^{-a} \,dt.
\end{equation}
\end{itemize}
\end{lem}

\begin{lem}
For $x,y\in \mathbb{C}$ with ${\rm Re\,} x, {\rm Re\,} y>0$, the Beta function satisfies
\[
B(x,y)
=\displaystyle\int_0^1
	\sigma^{x-1}(1-\sigma)^{y-1}
\,d\sigma
=\frac{
    \Gamma(x)\Gamma(y)
}{
    \Gamma(x+y)
}
=2\displaystyle\int_0^1
    \sigma^{2x-1}(1-\sigma^2)^{y-1}
\,d\sigma
=2\int_0^\infty
    \dfrac{
        \tau^{2x-1}
    }{
        (1+\tau^2)^{x+y}
    }
\,d\tau.
\]
\end{lem}

\section*{Acknowledgements}

HC has received funding from the European Research Council under Grant Agreement No 721675, from the Spanish Government under Grant CEX2019-000904-S funded by MCIN/AEI/10.13039/501100011033
and PID2020-113596GB-I00, and from the Swiss National Science Foundation under Grant PZ00P2\_202012/1. A. DelaTorre  has been supported by the following grants: Juan de la Cierva incorporaci\'on 2018 with ref. IJC2018-036320-I, PGC2018-096422-B-100, MTM2014-52402-C3-1-P and MTM2017-85757-P by Spanish government; by the FEDER-MINECO Grants PID2021- 122122NB-I00 and PID2020-113596GB-I00; RED2022-134784-T, funded by MCIN/AEI/10.13039/ 501100011033 and by J. Andalucia (FQM-116). This work is also partially supported by the IMAG–Maria de Maeztu grant CEX2020-001105-M / AEI / 10.13039/501100011033, Fondi Ateneo - Sapienza, PRIN (Prot. 20227HX33Z) and INdAM-GNAMPA Project 2023, codice CUP E53C2200193000 and INdAM -GNAMPA Project 2024, codice CUP E53C23001670001 and INdAM -GNAMPA Professore Visitatore.

We are indebted to Alessio Figalli, Mar\'{i}a del Mar Gonz\'{a}lez, Enno Lenzmann, Yannick Sire and Juncheng Wei for insightful comments.

\end{document}